\def\DateTime{27/November/2011, 13:30 (Kyoto)}
\def\Version{Version 2.5}
\def\yes{\if00}
\def\no{\if01}
\def\iftenpt{\no}
\def\ifelevenpt{\no}
\def\iftwelvept{\yes}

\def\ifusepdf{\no}
\def\ifpsfont{\no}
\def\iftxfont{\no}
\def\ifpxfont{\yes}
\def\ifeulerfont{\no}
\def\ifquery{\no}

\iftenpt
\documentclass[leqno]{amsart}
\else\fi
\ifelevenpt
\documentclass[leqno,11pt]{amsart}
\else\fi
\iftwelvept
\documentclass[leqno,12pt]{amsart}
\else\fi

\usepackage{amssymb}
\usepackage{amscd}
\usepackage{verbatim}
\usepackage{mathrsfs}

\ifusepdf
\usepackage{hyperref}
\else\fi
\ifpsfont
\usepackage[T1]{fontenc}
\usepackage{times}
\else\fi
\iftxfont
\usepackage{txfonts}
\else\fi
\ifpxfont
\usepackage{pxfonts}
\else\fi
\ifeulerfont
\usepackage{euler}
\else\fi

\ifelevenpt
\else\fi

\iftwelvept
\else\fi

\theoremstyle{plain}
\newtheorem{Theorem}{Theorem}[section]
\newtheorem{Proposition}[Theorem]{Proposition}
\newtheorem{Lemma}[Theorem]{Lemma}

\newtheorem{Corollary}[Theorem]{Corollary}

\newtheorem{Claim}{Claim}[Theorem]

\theoremstyle{definition}
\newtheorem{Definition}[Theorem]{Definition}
\newtheorem{Remark}[Theorem]{Remark}
\newtheorem{Example}[Theorem]{Example}

\newtheorem{Question}[Theorem]{Question}
\newtheorem{FQuestion}[Theorem]{Fundamental question}

\def\rom{\textup}
\newcommand{\ZZ}{{\mathbb{Z}}}
\newcommand{\QQ}{{\mathbb{Q}}}
\newcommand{\RR}{{\mathbb{R}}}
\newcommand{\KK}{{\mathbb{K}}}
\newcommand{\CC}{{\mathbb{C}}}
\newcommand{\PP}{{\mathbb{P}}}

\newcommand{\OO}{{\mathscr{O}}}

\newcommand{\Proj}{\operatorname{Proj}}

\newcommand{\Pic}{\operatorname{Pic}}
\newcommand{\Rat}{\operatorname{Rat}}

\newcommand{\Ker}{\operatorname{Ker}}

\newcommand{\Spec}{\operatorname{Spec}}

\newcommand{\Supp}{\operatorname{Supp}}

\newcommand{\ord}{\operatorname{ord}}
\newcommand{\acherncl}{\widehat{{c}}}
\newcommand{\adeg}{\widehat{\operatorname{deg}}}

\newcommand{\Lap}{\text{\Large$\Box$}}

\newcommand{\Nef}{\operatorname{Nef}}
\newcommand{\aChow}{\widehat{\operatorname{CH}}}

\newcommand{\aNef}{\widehat{\operatorname{Nef}}}

\newcommand{\Div}{\operatorname{Div}}
\newcommand{\aDiv}{\widehat{\operatorname{Div}}}
\newcommand{\WDiv}{\operatorname{WDiv}}
\newcommand{\aWDiv}{\widehat{\operatorname{WDiv}}}
\newcommand{\vol}{\operatorname{vol}}
\newcommand{\avol}{\widehat{\operatorname{vol}}}
\newcommand{\aH}{\hat{H}^0}

\newcommand{\admCont}{C^0_{{\rm ad}}}
\newcommand{\proCur}{D_{{\rm pr}}}
\newcommand{\admCur}{B_{{\rm ad}}}
\newcommand{\Tpsh}{\operatorname{PSH}}

\newcommand{\Tqpsh}{\operatorname{QPSH}}

\newcommand{\aPDiv}{\widehat{\operatorname{PDiv}}}
\newcommand{\PDiv}{\operatorname{PDiv}}

\newcommand{\Cur}{\operatorname{Cur}}

\newcommand{\QED}{{\unskip\nobreak\hfil\penalty50\quad\null\nobreak\hfil
{$\Box$}\parfillskip0pt\finalhyphendemerits0\par\medskip}}
\newcommand{\rest}[2]{\left.{#1}\right\vert_{{#2}}}
\ifquery
\def\query#1{\setlength\marginparwidth{65pt} 
\marginpar{\raggedright\fontsize{7.81}{9} 
\selectfont\upshape\hrule\smallskip 
#1\par\smallskip\hrule}} 

\else
\def\query#1{}
\fi

\begin{document}

\title[Toward Dirichlet's unit theorem on arithmetic varieties]%
{Toward Dirichlet's unit theorem on arithmetic varieties}
\author{Atsushi Moriwaki}
\address{Department of Mathematics, Faculty of Science,
Kyoto University, Kyoto, 606-8502, Japan}
\email{moriwaki@math.kyoto-u.ac.jp}
\date{\DateTime, (\Version)}
\begin{abstract}
In this paper, we would like to propose a fundamental question about a higher dimensional analogue of Dirichlet's unit theorem.
We also give a partial answer to the question as an application of the arithmetic Hodge index theorem.
\end{abstract}


\maketitle

\setcounter{tocdepth}{2}
\tableofcontents

\section*{Introduction}
Let $X$ be a $d$-dimensional, generically smooth, normal and projective arithmetic variety.
Let $\Rat(X)^{\times}_{\RR} := \Rat(X)^{\times} \otimes_{\ZZ} \RR$ and let
\[
\widehat{(\ )}_{\RR}: \Rat(X)^{\times}_{\RR} \to \aDiv_{C^{\infty}}(X)_{\RR}
\]
be the natural extension of the homomorphism $\Rat(X)^{\times} \to \aDiv_{C^{\infty}}(X)$ given by
$\phi \mapsto \widehat{(\phi)}$.
Let $\overline{D} = (D,g)$ be an arithmetic $\RR$-Cartier divisor of $C^{0}$-type on $X$.
We say $\overline{D}$ is {\em pseudo-effective} if
$\overline{D} + \overline{A}$ is big for any big arithmetic $\RR$-Cartier divisor $\overline{A}$ of
$C^{0}$-type on $X$.
In this paper, we would like to propose a fundamental question:

\renewcommand{\theTheorem}{\!\!}
\begin{FQuestion}
Are the following conditions (1) and (2) equivalent ?
\begin{enumerate}
\renewcommand{\labelenumi}{(\arabic{enumi})}
\item
$\overline{D}$ is pseudo-effective.

\item
$\overline{D} + \widehat{(\varphi)}_{\RR}$ is effective 
for some $\varphi \in \Rat(X)^{\times}_{\RR}$.
\end{enumerate}
\end{FQuestion}
\renewcommand{\theTheorem}{\arabic{section}.\arabic{Theorem}}
\addtocounter{Theorem}{-1}

Obviously (2) implies (1). Moreover,
if $\aH(X, a \overline{D}) \not= \{ 0 \}$ for some $a \in \RR_{>0}$,
then (2) holds.
Indeed, as we can choose $\phi \in \Rat(X)^{\times}$ with
$a \overline{D} + \widehat{(\phi)} \geq 0$, we have $\phi^{1/a} \in \Rat(X)^{\times}_{\RR}$ and
$\overline{D} + \widehat{(\phi^{1/a})}_{\RR} \geq 0$.
In the geometric case, (1) does not necessarily imply (2).
For example, let $\vartheta$ be a divisor on a compact Riemann surface $M$ such that
$\deg(\vartheta) = 0$ and the class of $\vartheta$ in $\Pic(M)$ is not a torsion element.
Then it is easy to see that $\vartheta$ is pseudo-effective and there is no
element $\psi$ of $\Rat(M)^{\times} \otimes_{\ZZ} \RR$ such that
$\vartheta + (\psi)_{\RR}$ is effective (cf. Remark~\ref{rem:pseudo:alg:curve}).
In this sense,
the above question is a purely arithmetic problem.

We assume $d=1$, that is, $X = \Spec(O_K)$, where
$K$ is a number field and $O_K$ is the ring of integers in $K$.
Let $K(\CC)$ be the set of all embeddings $K$ into $\CC$ and let
$\xi$ be an element of $\RR^{K(\CC)}$ such that
$\sum_{\sigma \in K(\CC)} \xi_{\sigma} = 0$ and $\xi_{\sigma} = \xi_{\bar{\sigma}}$  for all $\sigma \in K(\CC)$.
As we will observe in Subsection~\ref{subsec:Dirichlet:unit:theorem}, 
$(0,\xi) + \widehat{(\varphi)}_{\RR} \geq 0$ for some $\varphi \in \Rat(X)^{\times}_{\RR}$ if and only if
there exist $a_1, \ldots, a_r \in \RR$ and $u_1, \ldots, u_r \in O_K^{\times}$ such that
\[
\varphi = u_1^{a_2/2} \cdots u_r^{a_r/2}\quad\text{and}\quad
(0, \xi) + \widehat{(\varphi)}_{\RR} = 0,
\]
which is nothing more than Dirichlet's unit theorem.
In this sense, the above problem is a generalization of Dirichlet's unit theorem on arithmetic varieties.
The following theorem is our partial answer to the above question.

\begin{Theorem}
If $\overline{D}$ is pseudo-effective and $D$ is numerically trivial on $X_{\QQ}$,
then there exists $\varphi \in \Rat(X)^{\times}_{\RR}$ such that $\overline{D} + \widehat{(\varphi)}_{\RR}$
is effective.
\end{Theorem}

On an arithmetic curve, the assertion actually follows from Dirichlet's units theorem.
On a higher  dimensional arithmetic variety, 
it is a consequence of arithmetic Hodge index theorem and Dirichlet's units theorem.
Our theorem however treats only the case where $D$ is scanty.
For example, if $D$ is ample,
the problem seems to be difficult to get a solution.
For this purpose, we introduce a notion of multiplicative generators of approximately smallest sections.

Let $(\ )_{\RR} : \Rat(X)^{\times}_{\RR} \to \Div(X)_{\RR}$ be the natural extension of
the homomorphism $\Rat(X)^{\times} \to \Div(X)$ given by $\phi \mapsto (\phi)$.
Here we define $\Gamma^{\times}_{\RR}(X, D)$
to be
\[
\Gamma^{\times}_{\RR}(X, D) := \left\{ \varphi \in \Rat(X)^{\times}_{\RR} \mid D + (\varphi)_{\RR} \geq 0 \right\}.
\]
Let
$\ell : \Rat(X)^{\times} \to L^1_{loc}(X(\CC))$ be a homomorphism
given by $\varphi \mapsto \log \vert \varphi \vert$.
It extends to a linear map 
$\ell_{\RR} : \Rat(X)^{\times}_{\RR} \to L^1_{loc}(X(\CC))$.
For $\varphi \in \Rat(X)^{\times}_{\RR}$, we denote $\exp(\ell_{\RR}(\varphi))$ by $\vert \varphi \vert$.
If $\varphi \in \Gamma^{\times}_{\RR}(X, D)$, then
$\vert \varphi \vert \exp(-g/2)$ is continuous (cf. Lemma~\ref{lem:rational:tensor:real:number}), so that
we define $\Vert \varphi \Vert_{g, \sup}$ to be
\[
\Vert \varphi \Vert_{g, \sup} := \max \left\{ \left(\vert \varphi \vert \exp(-g/2)\right)(x) \mid x \in X(\CC) \right\}.
\]

Let $\varphi_1, \ldots, \varphi_l$ be elements of  $\Rat(X)^{\times}_{\RR}$.
We say $\varphi_1, \ldots, \varphi_l$ are 
{\em multiplicative generators of approximately smallest sections for $\overline{D}$} 
if, for  a given $\epsilon > 0$, there is $n_0 \in \ZZ_{>0}$ such that, for any integer $n$ with
$n \geq n_0$ and
$H^0(X, nD) \not= \{ 0 \}$, 
we can find $a_1, \ldots, a_l \in \RR$ satisfying 
$\varphi_1^{a_1} \cdots \varphi_l^{a_l} \in \Gamma_{\RR}^{\times}(X, nD)$ and
\[
\Vert \varphi_1^{a_1} \cdots \varphi_l^{a_l} \Vert_{ng,\sup} \leq e^{\epsilon n}
\min \left\{ \Vert \phi \Vert_{ng,\sup} \mid \phi \in H^0(X, nD) \setminus \{ 0 \} \right\}.
\]
If we admit the existence of multiplicative generators of approximately smallest sections,
then we can find $\varphi \in \Gamma^{\times}_{\RR}(X, D)$ such that
$\Vert \varphi \Vert_{g,\sup}$ gives rise to 
\[
\inf \left\{ \Vert \psi \Vert_{g, \sup} \mid \psi \in \Gamma^{\times}_{\RR}(X, D) \right\}
\]
(cf. Theorem~\ref{thm:inf:R}).
As a consequence, we have  the following partial answer to the above question.

\begin{Theorem}
If $\overline{D}$ is pseudo-effective, $D$ is big on the generic fiber of $X \to \Spec(\ZZ)$ and
$\overline{D}$ possesses multiplicative generators of approximately smallest sections,
then there exists $\varphi \in \Rat(X)^{\times}_{\RR}$ such that $\overline{D} + \widehat{(\varphi)}_{\RR} \geq 0$.
\end{Theorem}

For the proof, we need the following compactness theorem.

\begin{Theorem}
Let $\overline{H}$ be an ample arithmetic $\RR$-Cartier divisor on $X$.
Let $\Lambda$ be a finite set and let $\left\{ \overline{D}_{\lambda} \right\}_{\lambda \in \Lambda}$ be a family of 
arithmetic $\RR$-Cartier divisors 
of $C^{\infty}$-type with the following properties:
\begin{enumerate}
\renewcommand{\labelenumi}{(\roman{enumi})}
\item
$\adeg(\overline{H}^{d-1} \cdot \overline{D}_{\lambda}) = 0$ for all $\lambda \in \Lambda$.

\item
For each $\lambda \in \Lambda$, there is an $F_{\infty}$-invariant locally constant  function $\rho_{\lambda}$ on $X(\CC)$ such that
\[
c_1(\overline{D}_{\lambda}) \wedge c_1(\overline{H})^{\wedge d-2} = \rho_{\lambda} c_1(\overline{H})^{\wedge d-1}.
\]

\item
$\left\{ \overline{D}_{\lambda} \right\}_{\lambda \in \Lambda}$ is linearly independent in $\aDiv_{C^{\infty}}(X)_{\RR}$.
\end{enumerate}
Then the set
\[
\left\{ \pmb{a} \in \RR^{\Lambda} \ \left| \ \overline{D} + \sum_{\lambda \in \Lambda} \pmb{a}_{\lambda} \overline{D}_{\lambda}  \geq 0\right\}\right.
\]
is convex and compact for $\overline{D} \in \aDiv_{C^{0}}(X)_{\RR}$.
\end{Theorem}

For example, as an application of the above compactness theorem and the arithmetic Riemann-Roch theorem 
on arithmetic curves,
we can give a proof of Dirichlet's unit theorem (cf. Subsection~\ref{subsec:Dirichlet:unit:theorem}), which
indicates that the theory of arithmetic $\RR$-Cartier divisors is not an artificial material, 
but it actually provides realistic tools for arithmetic problems.
Here we would like to give the following question:

\begin{Question}
If $D$ is big on the generic fiber of $X \to \Spec(\ZZ)$, then does $\overline{D}$ have
multiplicative generators of approximately smallest sections ?
\end{Question}


\renewcommand{\thesubsubsection}{\arabic{subsubsection}}

\bigskip
\renewcommand{\theequation}{CT.\arabic{subsubsection}.\arabic{Claim}}
\subsection*{Conventions and terminology}
We basically use the same notation as in \cite{MoArZariski}.
Here we fix several conventions and the terminology of this paper.
Let $\KK$ be either $\QQ$ or $\RR$.
Moreover, in the following 3 $\sim$ 6,
$X$ is a $d$-dimensional, generically smooth, normal and projective arithmetic variety.

\subsubsection{}
\label{CT:currents}
Let $M$ be a $k$-equidimensional complex manifold.
The space of real valued continuous functions  (reps $C^{\infty}$-functions) 
on $M$ is denoted by
$C^0(M)$ (resp $C^{\infty}(M)$).
Moreover, the space of currents of bidegree $(p,q)$ is denoted by
$D^{p,q}(M)$.
Let $N^{p,q}(M)$ be the space of currents $T$ of bidegree $(p,q)$ such that
$T(\eta) = 0$ for all $d$-closed $C^{\infty}$ $(k-p, k -q)$-forms with
compact support.

\subsubsection{}
\label{CV:R:principal:div}
Let $S$ be a normal and integral noetherian scheme.
We denote the group of Cartier divisors (resp. Weil divisors) on $S$ by $\Div(S)$ (resp. $\WDiv(S)$).
We set 
\[
\Div(S)_{\KK} := \Div(S) \otimes_{\ZZ} \KK\quad\text{and}\quad
\WDiv(S)_{\KK} := \WDiv(S) \otimes_{\ZZ} \KK.
\]
An element of $\Div(S)_{\KK}$ (resp. $\WDiv(S)_{\KK}$) is called a {\em $\KK$-Cartier divisor} (resp.
{\em $\KK$-Weil divisor}) on $S$. 
We denote the group of principal divisors on $S$ by $\PDiv(S)$.
Let $\Rat(S)^{\times}_{\KK} := \Rat(S)^{\times} \otimes_{\ZZ} \KK$, that is,
\[
\Rat(S)^{\times}_{\KK} = \left\{ \phi_1^{\otimes a_1} \cdots \phi_l^{\otimes a_l} \mid
\text{$\phi_1, \ldots, \phi_l \in \Rat(S)^{\times}$ and $a_1, \ldots, a_l \in \KK$} \right\}.
\]
The homomorphism $\Rat(S)^{\times} \to \Div(S)$ given by $\phi \mapsto (\phi)$ naturally extends
to a homomorphism 
\[
(\ )_{\KK} : \Rat(S)^{\times}_{\KK} \to \Div(S)_{\KK},
\]
i.e. $(\phi_1^{\otimes a_1} \cdots \phi_l^{\otimes a_l}) = a_1 (\phi_1) + \cdots + a_l(\phi_l)$.
By abuse of notation, we sometimes denote $(\ )_{\KK}$ by $(\ )$.
We define
$\PDiv(S)_{\KK}$ to be 
\[
\PDiv(S)_{\KK} := \left\{ (\varphi)_{\KK} \mid \varphi \in \Rat(S)^{\times}_{\KK} \right\}.
\]
Note that
\[
\PDiv(S)_{\KK} := \left\langle \{ (\phi) \mid \phi \in \Rat(S)^{\times} \} \right\rangle_{\KK}
\subseteq \Div(S)_{\KK}.
\]
An element of $\PDiv(S)_{\KK}$ is called a {\em principal $\KK$-divisor} on $S$.

\subsubsection{}
\label{CV:arithmetic:connected:comp}
Let $X \overset{\pi}{\longrightarrow} \Spec(O_K) \to \Spec(\ZZ)$
be the Stein factorization of $X \to \Spec(\ZZ)$, where $K$ is a number field and $O_K$ is the ring of
integers in $K$. We denote by $K(\CC)$ the set of all embedding of $K$ into $\CC$.
For $\sigma \in K(\CC)$, we set $X_{\sigma} := X \times^{\sigma}_{\Spec(O_K)} \Spec(\CC)$,
where $\times^{\sigma}_{\Spec(O_K)}$ means the fiber product over $\Spec(O_K)$ with respect to $\sigma$.
Then $\{ X_{\sigma} \}_{\sigma \in K(\CC)}$ gives rise to the set of all connected components of $X(\CC)$.
For a locally constant function $\lambda$ on $X(\CC)$ and $\sigma \in K(\CC)$,
the value of $\lambda$ on the connected component $X_{\sigma}$ is denoted by $\lambda_{\sigma}$.
Clearly the set of all locally constant real valued functions on $X(\CC)$ can be identified with $\RR^{K(\CC)}$.
The complex conjugation map $X(\CC) \to X(\CC)$ is denoted by $F_{\infty}$.
Note that $F_{\infty}(X_{\sigma}) = X_{\bar{\sigma}}$.

\subsubsection{}
\label{CV:arithmetic:divisors:1}
Let $\mathcal{C}$ be a class of real valued continuous functions.
As examples of $\mathcal{C}$,
we can consider $C^{0}$, $C^{\infty}$, $C^0 \cap \Tpsh$ and so on,
where $C^0 \cap \Tpsh$ is the class of 
continuous plurisubharmonic functions.
A pair $\overline{D} = (D, g)$ 
is called an {\em arithmetic $\KK$-Cartier divisor
of $\mathcal{C}$-type} if
the following conditions are satisfied:
\begin{enumerate}
\renewcommand{\labelenumi}{(\alph{enumi})}
\item
$D$ is a $\KK$-Cartier divisor on $X$, that is,
$D = \sum_{i=1}^r a_i D_i$  for some $D_1, \ldots, D_r \in \Div(X)$ and
$a_1, \ldots, a_r \in \KK$.

\item
$g : X(\CC) \to \RR \cup \{\pm\infty\}$ is a locally integrable function and
$g \circ F_{\infty} = g \ (a.e.)$, where $F_{\infty} : X(\CC) \to X(\CC)$
is the complex conjugation map.

\item
For any point $x \in X(\CC)$, there are an open neighborhood $U_{x}$ of $x$ and
a function $u_x$ on $U_x$ such that $u_x$ belongs to the class $\mathcal{C}$ and
\[
g = u_x + \sum_{i=1}^r (-a_i) \log \vert f_i \vert^2\quad(a.e.)
\]
on $U_x$,
where $f_i$ is a local equation of $D_i$ over $U_x$ for each $i$.
\end{enumerate}
Let $\aDiv_{\mathcal{C}}(X)_{\KK}$ be
 the set of all arithmetic $\KK$-Cartier divisors of $\mathcal{C}$-type.
Note that 
there are natural surjective homomorphisms
\[
\aDiv_{C^0}(X) \otimes_{\ZZ} \RR \to \aDiv_{C^0}(X)_{\RR}\quad\text{and}\quad
\aDiv_{C^{\infty}}(X) \otimes_{\ZZ} \RR \to \aDiv_{C^{\infty}}(X)_{\RR}
\]
and that they are not isomorphisms respectively.
For details, see \cite{MoArZariski}.
For $\overline{D} \in \aDiv_{C^0}(X)_{\KK}$, the current
$dd^c([g]) + \delta_D$ is denoted by $c_1(\overline{D})$.
Note that $c_1(\overline{D})$ is locally equal to $dd^c([u_x])$.
If $\overline{D}$ is of $C^{\infty}$-type, then
$c_1(\overline{D})$ is represented by a $C^{\infty}$-form.
By abuse of notation, we also denote the $C^{\infty}$-form by $c_1(\overline{D})$.

\subsubsection{}
\label{CV:arithmetic:divisors:2}
The group of arithmetic principal divisors  
on $X$ is denoted by $\aPDiv(X)$.
The homomorphism $\Rat(X)^{\times} \to \aDiv_{C^{\infty}}(X)$ given by $\phi \mapsto \widehat{(\phi)}$
has the natural extension 
\[
\widehat{(\ )}_{\KK} : \Rat(X)^{\times}_{\KK} \to \aDiv_{C^{\infty}}(X)_{\KK},
\]
that is, $\widehat{(\varphi)} = a_1 \widehat{(\phi_1)} + \cdots + a_l \widehat{(\phi_l)}$ for
$\varphi = \phi_1^{\otimes a_1} \cdots \phi_l^{\otimes a_l}$.
For simplicity, $\widehat{(\ )}_{\KK}$ is occasionally denoted by $\widehat{(\ )}$.
We define $\aPDiv(X)_{\KK}$  
to be
\[
\aPDiv(X)_{\KK} :=\left\{ \widehat{(\varphi)}_{\KK} \mid \varphi \in \Rat(X)^{\times}_{\KK} \right\}.
\]
Note that
\[
\aPDiv(X)_{\KK} := \Big\langle \{ \widehat{(\phi)} \mid \phi \in \Rat(X)^{\times} \} \Big\rangle_{\KK} 
\subseteq \aDiv_{C^{\infty}}(X)_{\KK}. 
\]
An element of $\aPDiv(X)_{\KK}$  
is called
an {\em arithmetic principal $\KK$-divisor}  
on $X$.

\subsubsection{}
\label{CV:arithmetic:divisors:3}
An {\em arithmetic $\KK$-Weil divisor of $C^0$-type \rom{(}resp.  $C^{\infty}$-type\rom{)} on $X$} is 
a pair $\overline{D} = (D, g)$ 
consisting of
a $\KK$-Weil divisor $D$ on $X$ and a $D$-Green function $g$ of $C^0$-type (resp. $C^{\infty}$-type).
We denote the group of arithmetic $\KK$-Weil divisors of $C^0$-type (resp. of $C^{\infty}$-type) on $X$ by
$\aWDiv_{C^0}(X)_{\KK}$ (resp. $\aWDiv_{C^{\infty}}(X)_{\KK}$).
It is easy to see that there is a unique multi-linear form
\[
\alpha : \left(\aDiv_{C^{\infty}}(X)_{\KK}\right)^{d-1} \times \WDiv(X)_{\KK} \to \RR
\]
such that $\alpha(\overline{D}_1, \ldots, \overline{D}_{d-1}, \Gamma) =
\adeg(\rest{\overline{D}_1}{\widetilde{\Gamma}} \cdots \rest{\overline{D}_{d-1}}{\widetilde{\Gamma}})$
for $\overline{D}_1, \ldots, \overline{D}_{d-1} \in \aDiv_{C^{\infty}}(X)$ and
a prime divisor $\Gamma$ with $\Gamma \not\subseteq \Supp(D_1) \cup \cdots \cup \Supp(D_{d-1})$,
where $\widetilde{\Gamma}$ is the normalization of $\Gamma$.
We denote $\alpha(\overline{D}_1, \ldots, \overline{D}_{d-1}, D)$ by
$\adeg(\overline{D}_1 \cdots \overline{D}_{d-1} \cdot (D, 0))$.
Further, for $\overline{D}_1, \ldots, \overline{D}_{d-1} \in \aDiv_{C^{\infty}}(X)_{\KK}$
and $\overline{D} = (D, g) \in \aWDiv_{C^0}(X)_{\KK}$,
we define $\adeg(\overline{D}_1 \cdots \overline{D}_{d-1} \cdot \overline{D})$ to be
\[
\adeg(\overline{D}_1 \cdots \overline{D}_{d-1} \cdot \overline{D})
:=
\adeg(\overline{D}_1 \cdots \overline{D}_{d-1} \cdot (D,0)) + \frac{1}{2}
\int_{X(\CC)} g c_1(\overline{D}_1) \wedge \cdots \wedge c_1(\overline{D}_{d-1}).
\]

\subsubsection{}
\label{CV:vector:space:generated:by:set}
For a set $\Lambda$,
let $\RR^{\Lambda}$ be the set of all maps from $\Lambda$ to $\RR$.
The vector space generated by $\Lambda$ over $\RR$ is denoted by $\RR(\Lambda)$, that is,
\[
\RR(\Lambda) = \{ \pmb{a} \in \RR^{\Lambda} \mid \text{$\pmb{a}(\lambda) = 0$ except finitely many $\lambda \in \Lambda$} \}.
\]
For $\pmb{a} \in \RR^{\Lambda}$ and $\lambda \in \Lambda$, 
we often denote $\pmb{a}(\lambda)$ by $\pmb{a}_{\lambda}$.

\subsubsection{}
\label{CV:vol:parallelotope}
Let $V$ be a vector space over $\RR$ and let $\langle\ ,\ \rangle$ be an inner product on $V$.
For a finite subset $\{ x_1, \ldots, x_r \}$ of $V$,
we define $\vol(\{ x_1, \ldots, x_r \})$ to be the square root of the Gramian of $x_1, \ldots, x_r$ with respect to
$\langle\ ,\ \rangle$, that is,
\[
\vol(\{ x_1, \ldots, x_r \}) = \sqrt{\det \begin{pmatrix}
\langle x_1, x_1 \rangle & \langle x_1, x_2 \rangle & \cdots & \langle x_1, x_r \rangle \\
\langle x_2, x_1 \rangle & \langle x_2, x_2 \rangle & \cdots & \langle x_2, x_r \rangle \\
\cdots & \cdots & \cdots & \cdots \\
\langle x_r, x_1 \rangle & \langle x_r, x_2 \rangle & \cdots & \langle x_r, x_r \rangle
\end{pmatrix}}.
\]
For convenience, we set $\vol(\emptyset) = 1$.
Note that if $V = \RR^n$ and $\langle\ ,\ \rangle$ is the standard inner product, then
$\vol(\{ x_1, \ldots, x_r \})$ is the volume of the parallelotope given by
$\{ a_1x_1 + \cdots + a_r x_r \mid 0 \leq a_1 \leq 1, \ldots, 0 \leq a_r \leq 1 \}$.

\renewcommand{\theTheorem}{\arabic{section}.\arabic{subsection}.\arabic{Theorem}}
\renewcommand{\theClaim}{\arabic{section}.\arabic{subsection}.\arabic{Theorem}.\arabic{Claim}}
\renewcommand{\theequation}{\arabic{section}.\arabic{subsection}.\arabic{Theorem}.\arabic{Claim}}

\section{Preliminaries}
In this section, we prepare several materials for later sections.
Let us begin with elementary results on linear algebra.

\subsection{Lemmas of linear algebra}
\setcounter{Theorem}{0}

Here we would like to provide the following four lemmas of linear algebra.

\begin{Lemma}
\label{lem:Z:module:tensor:R}
Let $M$ be a $\ZZ$-module. Then we have the following:
\begin{enumerate}
\renewcommand{\labelenumi}{(\arabic{enumi})}
\item
For $x \in M \otimes_{\ZZ} \RR$, there are $x_1, \ldots, x_l \in M$ and $a_1, \ldots, a_l \in \RR$
such that $a_1, \ldots, a_l$ are linearly independent over $\QQ$ and
$x = x_1 \otimes a_1 + \cdots + x_l \otimes a_l$.

\item
Let $x_1, \ldots, x_l \in M$ and
$a_1, \ldots, a_l \in \RR$ such that $a_1, \ldots, a_l$ are linearly independent over $\QQ$.
If $x_1 \otimes a_1 + \cdots + x_l \otimes a_l = 0$ in $M \otimes_{\ZZ} \RR$,
then $x_1, \ldots, x_l$ are torsion elements in $M$.

\item
If $N$ is a submodule of $M$, then $(M \otimes_{\ZZ} \QQ) \cap (N \otimes_{\ZZ} \RR) = N \otimes_{\ZZ} \QQ$.
\end{enumerate}
\end{Lemma}

\begin{proof}
(1) As $x \in M \otimes_{\ZZ} \RR$, 
there are $a'_1, \ldots, a'_{r} \in \RR$ and $x'_1, \ldots, x'_{r} \in M$ such that
$x = {x'_1} \otimes {a'_1} + \cdots + {x'_{r}} \otimes {a'_{r}}$.
Let $a_1, \ldots, a_l$ be a basis of $\langle a'_1, \ldots, a'_{r} \rangle_{\QQ}$ over $\QQ$.
Then there are $c_{ij} \in \QQ$ such that $a'_i = \sum_{j=1}^l c_{ij} a_j$.
Replacing $a_j$ by $a_j/n$ ($n \in \ZZ_{>0}$) if necessarily,
we may assume that $c_{ij} \in \ZZ$.
If we set $x_j = \sum_{i=1}^r {c_{ij}} {x'_i}$, then $x_1, \ldots, x_l \in M$,
$x = x_1 \otimes {a_1} + \cdots + x_{s} \otimes {a_{s}}$ and
$a_1, \ldots, a_s$ are linearly independent over $\QQ$.

(2) We set $M' = \ZZ x_1 + \cdots + \ZZ x_l$.
Then, since $\RR$ is flat over $\ZZ$, the natural homomorphism
$M' \otimes \RR \to M \otimes \RR$ is injective, and hence we may assume that
$M$ is finitely generated.
Let $M_{tor}$ be the set of all torsion elements in $M$.
Considering $M/M_{tor}$, we may further assume that $M$ is free.
Note that the natural homomorphism
$\ZZ a_1 \oplus \cdots \oplus \ZZ a_l \to \RR$ is injective.
Thus $M \otimes_{\ZZ} (\ZZ a_1 \oplus \cdots \oplus \ZZ a_l) \to M \otimes_{\ZZ} \RR$ is also injective
because $M$ is flat over $\ZZ$. Namely,
\[
(M \otimes_{\ZZ} \ZZ a_1) \oplus \cdots \oplus (M \otimes_{\ZZ} \ZZ a_l) \to M \otimes_{\ZZ} \RR
\]
is injective. Therefore, $x_1 \otimes a_1 = \cdots = x_l \otimes a_l = 0$.
Thus $x_1 = \cdots = x_l = 0$ because the homomorphism $M \to M \otimes \RR$ given by
$x \mapsto x \otimes a_i$ is also injective for each $i$.

(3) It actually follows from \cite[Lemma~1.1.3]{MoArLin}.
For reader's convenience, we continue its proof in an elementary way.
Let us consider the following commutative diagram:
\[
\begin{CD}
0 @>>> N \otimes_{\ZZ} \QQ @>{\iota_{\QQ}}>> M \otimes_{\ZZ} \QQ @>{\varrho_{\QQ}}>> (M/N) \otimes_{\ZZ} \QQ @>>> 0 \\
@. @VV{\tau_N}V @VV{\tau_M}V @VV{\tau_{M/N}}V @. \\
0 @>>> N \otimes_{\ZZ} \RR @>{\iota_{\RR}}>> M \otimes_{\ZZ} \RR @>{\varrho_{\RR}}>> (M/N) \otimes_{\ZZ} \RR @>>>  0
\end{CD}
\]
Note that horizontal sequences are exact and vertical homomorphisms are injective.
Therefore,  we have
\[
(M \otimes_{\ZZ} \QQ) \cap (N \otimes_{\ZZ} \RR) = \Ker(\varrho_{\RR} \circ \tau_M) = \Ker(\tau_{M/N} \circ \varrho_{\QQ}) = \Ker(\varrho_{\QQ}) = N \otimes_{\ZZ} \QQ.
\]
\end{proof}

\begin{Lemma}
\label{lem:vol:ratio}
Let $V$ be a finite dimensional vector space over $\RR$ and let $\langle\ ,\ \rangle$ be an inner product on $V$.
Let $\Sigma$ be a non-empty finite subset of $V$ and $x \in \Sigma$.
Let $h$ be the distance between $x$ and $\langle \Sigma \setminus \{ x \}\rangle_{\RR}$
\rom{(}note that $\langle \emptyset \rangle_{\RR} = \{ 0 \}$\rom{)}.
Then we have the following \rom{(}\,for the definition of $\vol(\Sigma)$, 
see Conventions and terminology~\rom{\ref{CV:vol:parallelotope}}\rom{)}:
\begin{enumerate}
\renewcommand{\labelenumi}{(\arabic{enumi})}
\item
$\vol(\Sigma) = \vol(\Sigma \setminus \{ x \})h$.

\item
$\vol(\Sigma) \leq \vol(\Sigma \setminus \{ x \})\sqrt{\langle x, x \rangle}$.
In the case where $\Sigma \setminus \{ x \}$ consists of linearly independent vectors,
the equality holds if and only if $x$ is orthogonal to $\langle \Sigma \setminus \{ x \} \rangle_{\RR}$.

\item
We assume that $\Sigma \setminus \{ x \}$ consists of linearly independent vectors and $x \not= 0$.
If $\theta$ is the angle between $x$ and $\langle \Sigma \setminus \{ x \} \rangle_{\RR}$,
then
\[
\frac{\vol(\Sigma)}
{\sqrt{\langle x, x \rangle}\vol(\Sigma \setminus \{ x \})}
= \sin(\theta).
\]
\end{enumerate}
\end{Lemma}

\begin{proof}
(1) 
If $\#(\Sigma) = 1$, then the assertion is obvious, so that we may set $\Sigma = \{x_1, \ldots, x_n \}$,
where $x_1 = x$ and $n = \#(\Sigma) \geq 2$.
If $x_2, \ldots, x_n$ are linearly dependent, then $\vol(\Sigma) = \vol(\Sigma \setminus \{ x_1 \}) = 0$.
Thus the assertion is also obvious for this case. 
Moreover, if $x_1 \in \langle x_2, \ldots, x_r \rangle_{\RR}$, then $h = \vol(\Sigma) =0$.
Thus we may assume that $x_1, x_2, \ldots, x_n$ are linearly independent.
Let $\{ e_1, e_2, \ldots, e_r \}$ be an orthonormal basis of $\langle x_1, x_2, \ldots, x_r \rangle_{\RR}$ such that
$\{ e_2, \ldots, e_r \}$ yields an orthonormal basis of $\langle x_2, \ldots, x_r \rangle_{\RR}$.
We set $x_i = \sum_{j=1}^r a_{ij} e_j$. Then $h = \vert a_{11} \vert$ and $a_{i1} = 0$ for $i=2, \ldots, r$.
Further, if we set $A = (a_{ij})_{1 \leq i, j \leq r}$ and 
$A' = (a_{ij})_{2 \leq i, j \leq r}$, then $\vol(\Sigma) = \vert \det(A) \vert$ and
$\vol(\Sigma \setminus \{ x_1 \}) = \vert \det(A') \vert$. Thus the assertion follows.

(2) and (3) follow from (1).
\end{proof}

\begin{Lemma}
\label{lem:negative:semidefinite}
Let $V$ be a vector space over $\RR$ and let $\langle\ ,\ \rangle : V \times V \to \RR$
be a negative semi-definite symmetric bi-linear form, that is,
$\langle v, v \rangle \leq 0$ for all $v \in V$.
For $x \in V$, the following are equivalent:
\begin{enumerate}
\renewcommand{\labelenumi}{(\arabic{enumi})}
\item
$\langle x, x \rangle = 0$.

\item
$\langle x, y \rangle = 0$ for all $y \in V$.
\end{enumerate}
\end{Lemma}

\begin{proof}
Clearly (2) implies (1).
We assume $\langle x, x \rangle = 0$ and $\langle x, y \rangle \not= 0$ for some $y \in V$. First of all, 
\[
0 \geq \langle y + tx, y + tx \rangle = \langle y, y \rangle + 2t \langle x, y \rangle
\]
for all $t \in \RR$. Thus, if we set $t = -\langle y, y \rangle/\langle x, y \rangle$,
then the above implies $\langle y, y \rangle \geq 0$, and hence $\langle y, y \rangle = 0$.
Therefore, if we set $t = \langle x, y \rangle/2$, then we have
$\langle x, y \rangle^2 \leq 0$, which is a contradiction because $\langle x, y \rangle \not= 0$.
\end{proof}

\begin{Lemma}[Zariski's lemma for vector spaces]
\label{lem:Zariski:vector:sp}
Let $\KK$ be either $\QQ$ or $\RR$.
Let $V$ be a finite dimensional vector space over $\KK$, and let
$Q : V \times V \to \RR$ be a symmetric bi-linear form. We assume that there are $e \in V$ and
generators $e_1, \ldots, e_n$ of $V$ with the following properties:
\begin{enumerate}
\renewcommand{\labelenumi}{(\roman{enumi})}
\item
$e = a_1 e_1 + \cdots + a_n e_n$ for some
$a_1, \ldots, a_n \in \KK_{>0}$.

\item
$Q(e, e_i) \leq 0$ for all $i$.

\item
$Q(e_i, e_j) \geq 0$ for all $i \not= j$.

\item
If we set $S = \{ (i, j) \mid \text{$i \not= j$ and $Q(e_i, e_j) > 0$} \}$,
then, for any $i \not= j$, there is a sequence $i_1, \ldots, i_l$ such that
$i_1 = i$, $i_l = j$, and $(i_t, i_{t+1}) \in S$ for all $1 \leq t < l$.
\end{enumerate}
Then we have the following:
\begin{enumerate}
\renewcommand{\labelenumi}{(\arabic{enumi})}
\item
If $Q(e, e_i) < 0$ for some $i$, then
$Q$ is negative definite, that is,
$Q(x, x) \leq 0$ for all $x \in V$, and 
$Q(x, x) = 0$ if and only if $x = 0$.

\item
If $Q(e, e_i) = 0$ for all $i$, then
$Q$ is negative semi-definite and its kernel is $\KK e$, that is,
$Q(x, x) \leq 0$ for all $x \in V$, and 
$Q(x, x) = 0$ if and only if $x \in \KK e$.
\end{enumerate}
\end{Lemma}

\begin{proof}
Replacing $e_i$ by $a_i e_i$, we may assume that $a_1 = \cdots = a_n = 1$.
If we set $x = x_1 e_1 + \cdots + x_n e_n$ for some $x_1, \ldots, x_n \in \KK$,
then we can show
\[
Q(x, x) = \sum_i x_i^2 Q(e_i, e) - \sum_{i < j} (x_i - x_j)^2 Q(e_i, e_j).
\]
Thus our assertions follow from easy observations.
\end{proof}

\subsection{Proper currents and admissible continuous functions}
\label{subsec:admissible:cont:function}
\setcounter{Theorem}{0}
Throughout this subsection, we fix 
a $k$-equidimensional complex manifold $M$.
A current of bidegree $(l,l)$ on $M$ is said to be
{\em proper} if, for any $x \in M$,
there are an open neighborhood $U_x$ of $x$ and
$d$-closed positive currents $T_1, T_2$ of bidegree $(l,l)$ 
on $U_x$ such that $T = T_1 - T_2$ over $U_x$.
We denote the space of proper currents of bidegree $(l,l)$ by
$\proCur^{l,l}(M)$.
As a proper current is of order $0$,
for $f \in C^0(M)$ and $T \in \proCur^{l,l}(M)$,  we define the wedge product
$dd^c([f]) \wedge T$ of $dd^c([f])$ and $T$ to be
\[
dd^c([f]) \wedge T := dd^c(fT),
\]
that is, $(dd^c([f]) \wedge T)(\eta) = T(f dd^c(\eta))$ for a
$C^{\infty}$-form $\eta$ of bidegree $(k-l-1, k-l-1)$.
It is easy to see that the map 
\[
C^0(M) \times \proCur^{l,l}(M) \to D^{l+1, l+1}(M)
\]
given by $(f, T) \mapsto dd^c([f]) \wedge T$ is multi-linear.

A continuous function $f : M \to \RR$ is said to be {\em admissible}
if, for any point $x \in M$, there are an open neighborhood $U_x$ of $x$ and
continuous plurisubharmonic functions $\phi_1, \phi_2$ on $U_x$ 
such that $f = \phi_1 - \phi_2$ over $U_x$.
Note that $dd^c([f])$ is a proper current of bidegree $(1,1)$.
The space of admissible continuous functions on $M$
is denoted by $\admCont(M)$. 
It is easy to see that $C^{\infty}(M) \subseteq \admCont(M)$
(cf. the proof of (3) in Lemma~\ref{lem:wedge:B:proper:current}).
Moreover,  let $\admCur^{1,1}(M)$ be the space of currents $T$ of bidegree $(1,1)$
such that $T = dd^c([\varphi])$ locally for some admissible continuous function
$\varphi$ on each local open neighborhood.
As a $d$-closed positive $C^{\infty}$-form of bidegree $(1,1)$ can be locally written as
$dd^c(\text{$C^{\infty}$-function})$ (cf. \cite[Chapter~3, (1.18)]{Dem}),
any $d$-closed real $C^{\infty}$-form of bidegree $(1,1)$ on $M$ belongs to
$\admCur^{1,1}(M)$.

An upper semicontinuous function $f : X \to \RR \cup \{ -\infty \}$ is called a 
{\em quasiplurisubharmonic function} on $X$ if $f$ is locally a sum of a plurisubharmonic function and
a $C^{\infty}$-function.
We denote the space of all continuous quasiplurisubharmonic functions on $M$ by 
$(C^0 \cap\Tqpsh)(M)$.
Clearly $(C^0 \cap\Tqpsh)(M) \subseteq \admCont(M)$.
The subspace generated by $(C^0 \cap\Tqpsh)(M)$ in
$\admCont(M)$ is denoted by $\langle (C^0 \cap\Tqpsh)(M) \rangle_{\RR}$.
For a real continuous form $\alpha$ of
bidegree $(1,1)$, we define $\admCont(M;\alpha)$ to be
\[
\admCont(M;\alpha) := \left\{ f \in \admCont(M) \mid dd^c([f]) + \alpha \geq 0 \right\}.
\]
Note that $\admCont(M;\alpha) \subseteq (C^0 \cap\Tqpsh)(M)$
(cf. the proof of (3) in Lemma~\ref{lem:wedge:B:proper:current}).
Let us begin with the following lemma.

\begin{Lemma}
\label{lem:wedge:B:proper:current}
\begin{enumerate}
\renewcommand{\labelenumi}{(\arabic{enumi})}
\item
If $A \in \admCur^{1,1}(X)$ and $T \in \proCur^{l,l}(X)$, then
$A \wedge T \in \proCur^{l+1,l+1}(X)$.
Moreover, if $A$ and $T$ are positive,
then $A \wedge T$ is also positive.

\item
For $A_1, \ldots, A_r  \in \admCur^{1,1}(M)$ and
$T \in \proCur^{l,l}(M)$,
the wedge product
\[
A_1  \wedge \cdots \wedge A_r  \wedge T
\]
of currents $A_1, \ldots, A_r$ and $T$ is defined inductively as an element of
$\proCur^{r+l,r+l}(M)$ by using \rom{(1)},  that is,
\[
A_1  \wedge \cdots \wedge A_r  \wedge T = A_1  \wedge (A_2 \wedge \cdots \wedge A_r  \wedge T).
\] 
Then
the map $\admCur^{1,1}(M)^r \to \proCur^{r+l,r+l}(M)$ given by
\[
(A_1, \ldots, A_r) \mapsto A_1  \wedge \cdots \wedge A_r \wedge T
\]
is multi-linear and symmetric.

\item
Let $\alpha$ be a real continuous form of bidegree $(1,1)$.
Let $\{ f_{1, n} \}_{n=1}^{\infty}, \ldots, \{ f_{r, n} \}_{n=1}^{\infty}$ be
sequences in $\admCont(M; \alpha)$ such that
$\{ f_{i, n} \}_{n=1}^{\infty}$ converges locally uniformly
to $f_i \in \admCont(M; \alpha)$ for each $i$.
Then, for $T \in \proCur^{l,l}(M)$,
a sequence
\[
\left\{ f_{1,n} dd^c([f_{2,n}]) \wedge \cdots \wedge dd^c([f_{r,n}]) \wedge T \right\}_{n=1}^{\infty}
\]
converges weakly to
\[
f_{1} dd^c([f_{2}]) \wedge \cdots \wedge dd^c([f_{r}]) \wedge T.
\]
\end{enumerate}
\end{Lemma}

\begin{proof}
(1) This is a local question, so that we may assume that
there are continuous plurisubharmonic functions $\phi_1, \phi_2$
and $d$-closed positive currents $T_1, T_2$ such that
$A = dd^c([\phi_1]) - dd^c([\phi_2])$ and $T = T_1 - T_2$.
Therefore,
\[
A \wedge T =
\left( dd^c([\phi_1]) \wedge T_1 + dd^c([\phi_2]) \wedge T_2 \right)
- \left( dd^c([\phi_1]) \wedge T_2  + dd^c([\phi_2]) \wedge T_1 \right),
\]
as required.
The second assertion is obvious.

\medskip
(2)
The multi-linearity of $\admCur^{1,1}(M)^r \to \proCur^{r+l,r+l}(M)$ is obvious. 
For symmetry, it is sufficient to see
that following claim:

\begin{Claim}
Let $f$ and $g$ be continuous plurisubharmonic functions on $M$ and let
$T$ be a proper current on $M$. Then
$dd^c([f]) \wedge dd^c([g]) \wedge T =
dd^c([g]) \wedge dd^c([f]) \wedge T$.
\end{Claim}

\begin{proof}
If $f$ is $C^{\infty}$, then,
for a $C^{\infty}$-form $\eta$,
\begin{multline*}
(dd^c(f) \wedge dd^c([g]) \wedge T)(\eta) =
(dd^c([g]) \wedge T)(dd^c(f) \wedge\eta) =
T(g dd^c(dd^c(f) \wedge\eta)) \\
= T(g dd^c(f) \wedge dd^c(\eta)) =
(dd^c(f) \wedge T)(g dd^c(\eta)) = (dd^c([g]) \wedge dd^c(f) \wedge T)(\eta).
\end{multline*}
Otherwise, as the question is a local problem,
we can find a sequence of $C^{\infty}$ plurisubharmonic functions
$\{ f_n \}$ such that $\{ f_n \}$ converges locally uniformly to $f$.
Then $\{ dd^c(f_n) \wedge dd^c([g]) \wedge T\}$ and
$\{ dd^c([g]) \wedge dd^c(f_n) \wedge T\}$ converge weakly to
$dd^c([f]) \wedge dd^c([g]) \wedge T$ and
$dd^c([g]) \wedge dd^c([f]) \wedge T$ respectively (cf. \cite[Corollary~3.6 in Chapter~3]{Dem}), 
and hence the assertion follows.
\end{proof}

(3) This is also a local question.
For $x \in M$, let us consider a local coordinate $(z_1, \ldots, z_k)$ over an
open neighborhood $U_x$ of $x$.
As $dd^c(\log ( 1 + \vert z_1 \vert^2 + \cdots + \vert z_k \vert^2))$ is a positive form,
shrinking $U_x$ if necessarily, we can find $\lambda > 0$ such that
\[
\lambda dd^c(\log ( 1 + \vert z_1 \vert^2 + \cdots + \vert z_k \vert^2)) \geq \alpha
\]
over $U_x$. Thus,
if we set $\psi = \lambda \log ( 1 + \vert z_1 \vert^2 + \cdots + \vert z_k \vert^2)$, then
$f_i + \psi$, $g_i + \psi$, $f_{i,n} + \psi$ and
$g_{i,n} + \psi$ are continuous and plurisubharmonic over $U_x$ for all $i$ and $n$.
Therefore, (3) is a consequence of the convergence theorem for
plurisubharmonic functions (cf. \cite[Corollary~3.6 in Chapter~3]{Dem}).
\end{proof}

Next we consider the following lemma.

\begin{Lemma}
\label{lem:wedge:B:proper:current:compact}
We assume that $M$ is compact.
\begin{enumerate}
\renewcommand{\labelenumi}{(\arabic{enumi})}
\item
Let $\alpha$ be a positive continuous form of bidegree $(1,1)$.
If $f \in (C^0 \cap \Tqpsh)(M)$,
then there is a positive number $t_0$ such that
$f \in \admCont(M;t\alpha)$ for all $t \geq t_0$.

\item
For $f, g \in \langle (C^0 \cap \Tqpsh)(M) \rangle_{\RR}$ and $T \in \proCur^{l,l}(M)$,
\[
f dd^c([g]) \wedge T \equiv
g dd^c([f]) \wedge T \mod N^{l+1,l+1}(M)
\]
\rom{(}for the definition of $N^{l+1,l+1}(M)$,
see Conventions and terminology~\rom{\ref{CT:currents}}\rom{)}.

\item
Let $T$ be a $d$-closed positive  current of bidegree $(k-1,k-1)$.
Then
\[
\int_M f dd^c([f]) \wedge T  \leq 0
\]
 for $f \in \langle (C^0 \cap \Tqpsh)(M) \rangle_{\RR}$.
\end{enumerate}
\end{Lemma}

\begin{proof}
(1)
For each point $x \in X$, there are an open neighborhood $U_x$ of $x$,
a plurisubharmonic function $p_x$ on $U_x$ and a $C^{\infty}$-function $q_x$ on $U_x$ 
such that  $f = p_x + q_x$ over $U_x$.
If we consider a smaller $U_x$, then
we can write $\alpha$ and $dd^c(q_x)$ as follows:
\[
\alpha = \sqrt{-1}\sum_{ij} \alpha_{ij} dz_i \wedge d\bar{z}_j
\quad\text{and}\quad
dd^c(q_x) = \sqrt{-1} \sum_{ij} \beta_{ij} dz_i \wedge d\bar{z}_j,
\]
where $(z_1, \ldots, z_k)$ is a local coordinate on $U_x$.
As $(\alpha_{ij}(x))$ is a positive definite hermitian matrix,
we can find a positive number $s_x$ such that
$s_x(\beta_{ij}(x)) +  (\alpha_{ij}(x))$ is positive.
Note that $s_x(\beta_{ij}) +  (\alpha_{ij})$ is continuous on $U_x$. Thus,
shrinking $U_x$ if necessarily, 
$s_x(\beta_{ij}) +  (\alpha_{ij})$ is positive on $U_x$, and hence,
for $t \geq t_x := 1/s_x$,
\[
dd^c(q_x) + t \alpha = (t-t_x) \alpha + t_x(s_x dd^c(q_x) + \alpha) \geq 0
\]
on $U_x$.
Because of the compactness of $X$, there are finitely many $x_1, \ldots, x_r \in X$
with $X = U_{x_1} \cup \cdots \cup U_{x_r}$. If we set $t_0 = \max \{t_{x_1}, \ldots, x_{x_r} \}$,
then, for $t \geq t_0$, 
\[
dd^c([f]) + t \alpha = dd^c([p_{x_i}]) + (dd^c(q_{x_i}) + t \alpha)
\]
is positive over $U_{x_i}$, as required.

\medskip
(2) By our assumption, there are $f_1, f_2, g_1, g_2 \in (C^0 \cap \Tqpsh)(M)$
such that $f = f_1 - f_2$ and $g = g_1 - g_2$.
Therefore, we may assume that $f, g \in (C^0 \cap \Tqpsh)(M)$.
If $f$ is $C^{\infty}$, then, 
for a $d$-closed $C^{\infty}$-form $\eta$ of bidegree $(k-l-1, k-l-1)$,
\[
(f dd^c([g]) \wedge T)(\eta) =T(g dd^c(f\eta)) = T(g dd^c(f) \wedge \eta) 
= (g dd^c(f) \wedge T)(\eta).
\]
Otherwise, by (1), we can take a positive $C^{\infty}$-form $\alpha$ of bidegree $(1,1)$ with
$f \in \admCont(X; \alpha)$.
Thus, by \cite{BK} or \cite[Lemma~4.2]{MoArZariski}, 
we can find a sequence of $C^{\infty}$-functions
$\{ f_{n} \}$ in $\admCont(M;\alpha)$ such that $\{ f_{n} \}$ converges uniformly to $f$.
Therefore, by (3) in Lemma~\ref{lem:wedge:B:proper:current}, 
\[
f_{n} dd^c([g]) \wedge T\quad\text{and}\quad
g dd^c(f_{n}) \wedge T
\]
converges weakly to $f dd^c([g]) \wedge T$ and
$g dd^c([f]) \wedge T$ respectively.
Thus (2) follows from the case where $f$ is $C^{\infty}$.

\medskip
(3) First we assume $f$ is $C^{\infty}$.
Then, as
\[
\partial \left( \frac{\sqrt{-1}}{2\pi} f \bar{\partial}(f) \right) = \frac{\sqrt{-1}}{2\pi} \partial(f)\wedge
 \bar{\partial}(f) + f dd^c(f)
\]
and $T$ is $\partial$-closed, we have
\[
0 = -(\partial T) \left( \frac{\sqrt{-1}}{2\pi} f \bar{\partial}(f) \right) = 
T  \left(\partial \left( \frac{\sqrt{-1}}{2\pi} f \bar{\partial}(f) \right) \right) =
T \left(  \frac{\sqrt{-1}}{2\pi} \partial(f)\wedge
 \bar{\partial}(f)\right) + T( f dd^c(f)).
\]
Note that 
\[
T \left(  \frac{\sqrt{-1}}{2\pi} \partial(f)\wedge
 \bar{\partial}(f)\right) \geq 0.
\]
Thus we have the assertion in the case where $f$ is $C^{\infty}$.

In general, by using (1),
we can find continuous functions $g, h$ on $M$  and
a positive $C^{\infty}$-form $\alpha$ such that
$g, h \in \admCont(M; \alpha)$ and $f = g - h$.
Thus, by \cite{BK} or \cite[Lemma~4.2]{MoArZariski}, 
there are sequences $\{ g_n \}_{n=1}^{\infty}$ and
$\{ h_n \}_{n=1}^{\infty}$ of $C^{\infty}$-functions on $M$
such that $g_n, h_n \in \admCont(M; \alpha)$ for all $n \geq 1$ and
\[
\lim_{n\to\infty} \Vert g_n - g \Vert_{\sup} = \lim_{n\to\infty} \Vert h_n - h \Vert_{\sup} = 0.
\]
Then, by (3) in Lemma~\ref{lem:wedge:B:proper:current}, a sequence
$\{ (g_n - h_n) dd^c(g_n - h_n) \wedge T\}$
of currents
converge weakly to
$(g - h) dd^c([g-h])\wedge T = f dd^c([f]) \wedge T$.
Thus, (3) follows from the previous case.
\end{proof}

From now on, we assume that $M$ is compact and K\"{a}hler.
Let $T$ be a $d$-closed positive current of bidegree $(k-1,k-1)$.
For $f, g \in   \admCont(M)$,
we define $I_{T}(f, g)$ to be
\[
I_{T}(f, g) := \int_{M} f dd^{c}([g])\wedge T.
\]
Then we have the following proposition.

\begin{Proposition}
\label{prop:sym:semidefinite:bilinear}
$I_{T}$ is a symmetric and negative semidefinite bi-linear form on
\[
\langle (C^0 \cap \Tqpsh)(M) \rangle_{\RR},
\]
that is,
the following properties are satisfied:
\begin{enumerate}
\renewcommand{\labelenumi}{(\arabic{enumi})}
\item
$I_{T}(af + bf', g) = a I_{T}(f, g) + b I_{T}(f', g)$ and
$I_{T}(f, ag+bg') = a I_{T}(f, g) + b I_{T}(f, g')$
hold for
all $f, f', g, g' \in \admCont(M)$ and $a, b \in \RR$.

\item
$I_{T}(f, g) = I_{T}(g, f)$ for all $f, g \in  \langle (C^0 \cap \Tqpsh)(M) \rangle_{\RR}$.

\item
$I_{T}(f, f) \leq 0$ for all $f \in  \langle (C^0 \cap \Tqpsh)(M) \rangle_{\RR}$. 
\end{enumerate}
Moreover, let $A_1, \ldots, A_{k-1} \in \admCur^{1,1}(M)$ and let $\omega$ be a K\"{a}hler form
of $M$.
We assume that, for each $i=1, \ldots, k-1$, 
there is $\epsilon_i \in \RR_{>0}$ with
$A_i \geq \epsilon_i \omega$  .
If $T = A_1 \wedge \cdots \wedge A_{k-1}$,
then
\[
I_{T}(f, f) = 0\quad\Longleftrightarrow\quad\text{$f$ is a constant}.
\]
\end{Proposition}

\begin{proof}
(1) is obvious.  (2) follows from (2) in Lemma~\ref{lem:wedge:B:proper:current:compact}.
(3) is a consequence of (3) in  Lemma~\ref{lem:wedge:B:proper:current:compact}.
Finally we consider the last assertion.
Clearly if $f$ is a constant, then $I_T(f, f) = 0$.
We set 
\[
T' =  (\epsilon_1^{-1}A_1) \wedge \cdots \wedge  (\epsilon_{k-1}^{-1}A_{k-1}) = (\epsilon_1 \cdots \epsilon_{k-1})^{-1}T.
\]
Then, as $ \epsilon_i^{-1}A_i -\omega$ is positive, 
by (1) in Lemma~\ref{lem:wedge:B:proper:current},
there is a $d$-closed positive current $T''$ of bidegree $(k-1,k-1)$ such that
$T' = \omega^{k-1} + T''$.
In particular, by (3),
\[
I_{T'}(f, f) \leq I_{\omega^{k-1}}(f, f) \leq 0
\]
for $f \in \langle (C^0 \cap \Tqpsh)(M) \rangle_{\RR}$.
Note that we can define a Laplacian $\Lap_{\omega}$
by the equation: 
\[
-dd^c(f) \wedge \omega^{k-1} = \Lap_{\omega}(f) \omega^{k}\quad (f \in C^{\infty}(X)).
\]
Therefore, 
\begin{align*}
I_{T}(f, f) = 0 & \quad\Longrightarrow\quad  I_{T'}(f, f) = 0 
\quad\Longrightarrow\quad  I_{\omega^{k-1}}(f, f) = 0 \\
& \quad\Longrightarrow\quad \text{$I_{\omega^{k-1}}(g, f) = 0$ for all $g \in C^{\infty}(X)$
($\because$ Lemma~\ref{lem:negative:semidefinite})}\\
& \quad\Longrightarrow\quad  \text{$dd^c([f]) \wedge \omega^{k-1} = 0$ as a current} \\
& \quad\Longrightarrow\quad \Lap_{\omega}([f]) = 0 \\
& \quad\Longrightarrow\quad \text{$f$ is harmonic ($\because$ the regularity of elliptic operators)} \\
& \quad\Longrightarrow\quad \text{$f$ is a constant},
\end{align*}
as required.
\end{proof}

\subsection{A variant of Gromov's inequality for $\RR$-Cartier divisors}
\setcounter{Theorem}{0}

In this subsection, we would like to consider a generalization of \cite[Lemma~1.1.4]{MoCont} to 
$\RR$-Cartier divisors.

\begin{Lemma}
\label{lem:comp:norm:X:U}
Let $X$ be a $d$-dimensional compact K\"{a}hler manifold and let
$\omega$ be a K\"{a}hler form on $X$.
Let $D_1, \ldots, D_l$ be $\RR$-Cartier divisors on $X$.
For each $i=1, \ldots, l$, let $g_i$ be a $D_i$-Green function of $C^{\infty}$-type.
Let $U$ be an open set of $X$ such that $U$ is not empty on each connected component of $X$.
Then there are constants $C_1, \ldots, C_l \geq 1$ such that
$C_i$ depends only on $g_i$ and $U$, and that
\[
\sup_{x \in X} \{ \vert s \vert_{m_1g_1 + \cdots +m_lg_l} (x) \} \leq C_1^{m_1} \cdots C_l^{m_l} 
\sup_{x \in U} \{ \vert s \vert_{m_1g_1 + \cdots +m_lg_l}(x) \}.
\]
for all $m_1, \ldots, m_l \in \RR_{\geq 0}$
and all 
$s \in H^0(X, m_1 D_1 + \cdots + m_l D_l)$.
Moreover, if $D_i = 0$ and $g_i$ is a constant function, then
$C_i = 1$.
\end{Lemma}

\begin{proof}
Clearly we may assume that $X$ is connected.
Shrinking $U$ if necessarily, we may
identify $U$ with 
$\{ x \in \CC^d \mid \vert x \vert < 1\}$.
We set $W = \{ x \in \CC^d \mid \vert x \vert < 1/2 \}$.
In this proof, we define a Laplacian $\Lap_{\omega}$ by
the formula:
\[
-\frac{\sqrt{-1}}{2 \pi} \partial\bar{\partial}(g) \wedge \omega^{\wedge (d-1)} =
\Lap_{\omega}(g) \omega^{\wedge d}.
\]
Let $\omega_i$ be a $C^{\infty}$-form of $(1,1)$-type given by
$dd^c([g_i]) + \delta_{D_i} = [\omega_i]$.
Let $a_i$ be a $C^{\infty}$-function given by
$\omega_i \wedge \omega^{\wedge (d-1)} = a_i \omega^{\wedge d}$.
We choose a $C^{\infty}$-function $\phi_i$ on $X$ such that
\[
\int_{X} a_i \omega^{\wedge d} = \int_{X} \phi_i \omega^{\wedge d}
\]
and that $\phi_i$ is identically zero on $X \setminus W$.
Thus we can find a $C^{\infty}$-function $F_i$ with
$\Lap_{\omega}(F_i) = a_i - \phi_i$.
Note that $\Lap_{\omega}(F_i) = a_i$ on $X \setminus W$.

Let $s \in H^0(X, m_1 D_1+ \cdots + m_l D_l)$. We set
\[
f = \vert s \vert_{m_1g_1 + \cdots +m_lg_l}^2 \exp(-(m_1F_1 + \cdots + m_l F_l)).
\]
Note that $f$ is continuous over $X$ and $\log(f)$ is $C^{\infty}$ over $X \setminus Z_s$,
where 
\[
Z_s = \Supp((s) + m_1 D_1 + \cdots + m_l D_l).
\]

\begin{Claim}
$\max_{ x \in X \setminus W} \{ f(x) \}
= \max_{ x \in  \partial(W)} \{ f(x) \}$.
\end{Claim}

If $f$ is a constant over $X \setminus W$, then our assertion is obvious,
so that we assume that $f$ is not a constant over $X \setminus W$.
In particular, $s \not= 0$. Since
\[
-\frac{\sqrt{-1}}{2 \pi} \partial\bar{\partial}(\log(\vert s \vert_{m_1g_1 + \cdots +m_lg_l}^2)) =
m_1	\omega_1 + \cdots + m_l \omega_l\quad\text{over $X \setminus Z_s$},
\]
we have $\Lap_{\omega}(\log(f))= 0$
on $X \setminus (W \cup Z_s)$.
Let us choose $x_0 \in X \setminus W$ such that
the continuous function $f$ over $X \setminus W$
takes the maximum value at $x_0$.
Note that 
\[
 x_0  \in X \setminus (W \cup Z_s).
\]
For, if $Z_s = \emptyset$,
then our assertion is obvious.
Otherwise, $f$ is zero
at any point of $Z_s$. 
Since $\log(f)$ is harmonic over
$X \setminus (W \cup Z_s)$,
$\log(f)$ takes the maximum value at $x_0$ and
$\log(f)$ is not a constant, we have $x_0 \in  \partial(W)$
by virtue of the maximum principle of harmonic functions.
Thus the claim follows.
\QED

We set
\[
b_i = \min_{ x \in X \setminus W} \{ \exp(-F_i)\},\quad
B_i = \max_{x \in \partial(W)} \{ \exp(-F_i) \}\quad\text{and}\quad
C_i = B_i/b_i.
\]
Then
\[
b_1^{m_1} \cdots b_l^{m_l} \vert s \vert_{m_1g_1 + \cdots + m_lg_l}^2  \leq f
\]
over $X \setminus W$ and
\[
f \leq B_1^{m_1} \cdots B_l^{m_l} \vert s \vert_{m_1g_1 + \cdots +m_lg_l}^2
\]
over $\partial(W)$. Hence
\[
\max_{x \in X \setminus W} \{ \vert s \vert_{m_1g_1 + \cdots +m_lg_l}^2 \} 
 \leq
C_1^{m_1} \cdots C_l^{m_l} \max_{ x \in  \partial(W)} \{ \vert s \vert_{m_1g_1 + \cdots +m_lg_l}^2 \}
\leq
C_1^{m_1}\cdots C_l^{m_l} \max_{ x \in  \overline{W}} \{ \vert s \vert_{m_1g_1 + \cdots +m_lg_l}^2 \}.
\]
which implies that
\[
\max_{ x \in X} \{ \vert s \vert_{m_1g_1 + \cdots +m_lg_l}^2\} \leq 
C_1^{m_1} \cdots C_l^{m_l} \max_{ x \in  \overline{W}} \{ \vert s \vert_{m_1g_1 + \cdots +m_lg_l}^2 \},
\]
as required. The last assertion is obvious by our construction because $F_i = 0$ in this case.
\end{proof}

\renewcommand{\theTheorem}{\arabic{section}.\arabic{subsection}.\arabic{Theorem}}
\renewcommand{\theClaim}{\arabic{section}.\arabic{subsection}.\arabic{Theorem}.\arabic{Claim}}
\renewcommand{\theequation}{\arabic{section}.\arabic{subsection}.\arabic{Theorem}.\arabic{Claim}}

\section{Hodge index theorem for arithmetic $\RR$-Cartier divisors}
In this section, we would like to observe the Hodge index theorem for arithmetic 
$\RR$-Cartier divisors and
apply it to the pseudo-effectivity of arithmetic divisors.
A negative definite quadric form over $\QQ$ does not necessarily extend to a negative definite
quadric form over $\RR$. For example, the quadric form $q(x,y) = -(x + \sqrt{2}y)^2$ on $\QQ^2$ is negative definite, but it
is not negative definite on $\RR^2$.
In this sense, the equality condition of Hodge index theorem 
for arithmetic $\RR$-Cartier divisors is not an obvious generalization.
In addition, the equality condition is crucial to consider the pseudo-effectivity of 
$\RR$-Cartier divisors.

Throughout this section, $X$ will be a $d$-dimensional, generically smooth, normal projective arithmetic variety.
Moreover, let
\[
X \overset{\pi}{\longrightarrow} \Spec(O_K) \to \Spec(\ZZ)
\]
be
the Stein factorization of $X \to \Spec(\ZZ)$, where $K$ is a number field and
$O_K$ is the ring of integers in $K$.

\subsection{Generalized intersection pairing on arithmetic varieties}
\setcounter{Theorem}{0}
Let $\aDiv_{C^{0}}^{\Nef}(X)_{\RR}$ be the subspace of $\aDiv_{C^{0}}(X)_{\RR}$
generated by $\aNef_{C^{0}}(X)_{\RR}$.
For $\overline{D}_1, \ldots, \overline{D}_d \in
\aDiv_{C^{0}}^{\Nef}(X)_{\RR}$,
we can define the intersection number 
$\adeg(\overline{D}_1\cdots\overline{D}_d)$ as follows:
If $\overline{D}_1, \ldots, \overline{D}_d \in
\aNef_{C^{0}}(X)_{\RR}$, then
it is given by
\[
\adeg(\overline{D}_1 \cdots \overline{D}_d) = \frac{1}{d!} \sum_{\emptyset \not= I \subseteq \{ 1, \ldots, d\}} 
(-1)^{d - \#(I)} \avol\left(\sum_{i \in I} \overline{D}_i \right).
\]
In general, we extend the above by multi-linearity (for details, see \cite[\S~6.4]{MoArZariski}).
Note that if $\overline{D}_1, \ldots, \overline{D}_d \in \aDiv_{C^{\infty}}(X)_{\RR}$,
then
$\adeg(\overline{D}_1 \cdots \overline{D}_d)$ 
coincides with the usual arithmetic intersection number
because the self intersection number of a nef
arithmetic $\RR$-Cartier divisor of 
$C^{\infty}$-type
in the usual sense is equal to its arithmetic volume (cf. \cite[Claim~6.4.2.2]{MoArZariski}).
The following proposition is the main result of this subsection.
Especially, (3) means that
the above intersection number coincides with other definitions \cite[Lemma~6.5]{ZhPos}, \cite[\S1]{ZhSm} and \cite[\S5]{MGA}.
In this sense, this subsection provides a quick introduction to
the generalized intersection pairing on arithmetic varieties.

Here we need to fix a notation.
Let $u_1, \ldots, u_p \in  \langle (C^0 \cap \Tqpsh)(X(\CC)) \rangle_{\RR}$ and
$B_1, \ldots, B_p \in \admCur^{1,1}(X(\CC))$.
Let $I$ be a non-empty subset of $\{ 1, \ldots, p \}$ and $J = \{ 1, \ldots, p \} \setminus I$.
If we set $I = \{ i_1, \ldots, i_k \}$ and $J = \{ j_1, \ldots, j_l \}$,
then, by Lemma~\ref{lem:wedge:B:proper:current}, the class of
\[
u_{i_1} dd^c([u_{i_2}]) \wedge \cdots \wedge dd^c([u_{i_k}]) \wedge
B_{j_1} \wedge \cdots \wedge B_{j_l}
\]
in $D^{p-1,p-1}(X(\CC))/N^{p-1,p-1}(X(\CC))$
does not depend on the choice of $i_1, \ldots, i_k$ and
$j_1, \ldots, j_l$, so that
it is denoted by $u dd^c(u_I) \wedge B_J$.

\begin{Proposition}
\label{prop:intersection:capacity}
\begin{enumerate}
\renewcommand{\labelenumi}{(\arabic{enumi})}
\item
If  $\overline{D}= \overline{D}' + (0, \eta)$
for $\overline{D}, \overline{D}'
\in \aDiv_{C^0}^{\Nef}(X)_{\RR}$ and
$\eta \in C^0(X)$,
then 
$\eta  \in \langle (C^0 \cap \Tqpsh)(X(\CC)) \rangle_{\RR}$.

\item
Let $\overline{D}_1, \ldots, \overline{D}_d \in \aDiv_{C^0}^{\Nef}(X)_{\RR}$,
$\overline{A}_1, \ldots, \overline{A}_d \in \aDiv_{C^{\infty}}(X)_{\RR}$ and
$u_1, \ldots, u_d \in C^0(X)$ such that
$\overline{D}_i = \overline{A}_i + (0, u_i)$ for $i=1, \ldots, d$.
Then the quantity 
\[
\adeg(\overline{A}_1\cdots\overline{A}_d) 
+ \frac{1}{2} \sum_{\emptyset \not= I \subseteq \{ 1, \ldots, d\}} \int_{X(\CC)} udd^c(u_I) \wedge c_1(\overline{A}_J)
\]
does not depend on the choice of
$\overline{A}_1, \ldots, \overline{A}_d$ and
$u_1, \ldots, u_d$.
If we denote the above number by $\adeg'(\overline{D}_1\cdots\overline{D}_d)$,
then the map 
\[
\left(\aDiv_{\admCont}(X)_{\RR}\right)^d \to \RR
\]
given by
$(\overline{D}_1, \ldots, \overline{D}_d) \mapsto \adeg'(\overline{D}_1\cdots\overline{D}_d)$
is symmetric and multi-linear.

\item
$\adeg(\overline{D}_1 \cdots \overline{D}_d) = \adeg'(\overline{D}_1 \cdots \overline{D}_d)$
for $\overline{D}_1, \ldots, \overline{D}_d \in \aDiv^{\Nef}_{C^{0}}(X)_{\RR}$.

\item
Let $\overline{D}_1, \cdots, \overline{D}_d, \overline{D}_1', \cdots, \overline{D}_d' 
\in \aDiv_{C^0}^{\Nef}(X)_{\RR}$ and
$\eta_1, \ldots, \eta_d \in C^0(X)$ such that
$\overline{D}_i = \overline{D}' _i+ (0, \eta_i)$ for $i=1, \ldots, d$.
Then 
\[
\adeg(\overline{D}_1 \cdots \overline{D}_d) =
\adeg(\overline{D}'_1\cdots \overline{D}'_d) +
\frac{1}{2} \sum_{\emptyset \not= I \subseteq \{ 1, \ldots, d\}} \int_{X(\CC)} \eta dd^c(\eta_I) \wedge c_1(\overline{D}'_J).
\]
\end{enumerate}
\end{Proposition}

\begin{proof}
(1) We can find
$\overline{E}, \overline{F}, \overline{E}', \overline{F}' \in \aNef_{C^{0}}(X)_{\RR}$
such that $\overline{D} = \overline{E} - \overline{F}$ and
$\overline{D}' = \overline{E}' - \overline{F}'$.
Then, as $\overline{E} + \overline{F}' = \overline{E}' + \overline{F}  + (0, \eta)$,
the assertion of (1) is obvious if we compare two local equations of the Green functions
in $\overline{E} + \overline{F}'$ and $\overline{E}' + \overline{F}$.

(2) 
In order to proceed with arguments,
we need several notations.
Let $\widehat{Z}^p(X)_{\RR}$ be the set of all pairs $(Z, T)$
such that
$Z$ is a codimension $p$ $\RR$-cycle on $X$ (i.e. $Z = a_1 Z_1 + \cdots + a_r Z_r$ for some
$a_1, \ldots, a_r \in \RR$ and codimension $p$ integral subschemes $Z_1, \ldots, Z_r$ of  $X$)
and $T$ is a real current of bidegree $(p-1, p-1)$ on $X(\CC)$.
Let $\widehat{R}^p(X)'_{\RR}$ be the vector subspace generated by the following elements:
\begin{enumerate}
\renewcommand{\labelenumi}{(\alph{enumi})}
\item 
$((f), - [\log |f|^2])$, 
where $f$ is a rational function on some
integral closed subscheme $Y$ of codimension $p-1$ and $[\log |f|^2]$ 
is the current defined by
\[
[\log |f|^2](\gamma) = 
        \int_{Y(\CC)} (\log |f|^2)\gamma.
\]

\item
$(0, T)$, where $T$ is
a real current in $N^{p-1,p-1}(X(\CC))$.
(for the definition of $N^{p-1,p-1}(X(\CC))$,
see Conventions and terminology~\ref{CT:currents}).
\end{enumerate}
We set 
\[
\aChow^p(X)'_{\RR} := \widehat{Z}^p(X)_{\RR}/\widehat{R}^p(X)'_{\RR}.
\]
Let $\overline{A}$ be an arithmetic $\RR$-Cartier divisor of $C^{\infty}$-type.
Then we can define a homomorphism
\[
\acherncl_1(\overline{A})\cdot :
\aChow^p(X)'_{\RR} \to \aChow^{p+1}(X)'_{\RR}
\]
given by $\acherncl_1(\overline{A})\cdot (Z, T) = \acherncl_1(\overline{A})\cdot (Z, 0) +
(0, c_1(\overline{A}) \wedge T)$.
Note that 
\[
\acherncl_1(\overline{A})\cdot \acherncl_1(\overline{B})\cdot =
\acherncl_1(\overline{B})\cdot \acherncl_1(\overline{A})\cdot
\]
for arithmetic $\RR$-Cartier divisors $\overline{A}$ and $\overline{B}$ of $C^{\infty}$-type.

\begin{Claim}
The class of
\[
Z(\overline{A}_1, \ldots. \overline{A}_p, u_1, \ldots, u_p) 
:= \acherncl_1(\overline{A}_1) \cdots \acherncl_1(\overline{A}_p) +
\sum_{\emptyset \not= I \subseteq \{ 1, \ldots, p\}}
(0, udd^c(u_I) \wedge c_1(\overline{A}_J))
\]
in $\aChow^p(X)'_{\RR}$ does not depend on the choice
of $\overline{A}_1, \ldots, \overline{A}_p$ and $u_1, \ldots, u_p$
for $p=1, \ldots, d$.
\end{Claim}

\begin{proof}
Let $\overline{B}_1, \ldots, \overline{B}_p$ be
arithmetic $\RR$-Cartier divisors of $C^{\infty}$-type and
$v_1, \ldots, v_p \in \admCont(X)$
such that $\overline{D}_i = \overline{B}_i + (0, v_i)$ for $i=1, \ldots, p$.
Then we can find $C^{\infty}$-function $\phi_1, \ldots, \phi_p$ such that
$u_i = v_i + \phi_i$ and $\overline{B}_i = \overline{A}_i + (0, \phi_i)$
for $i=1, \ldots, p$.
We need to see
that
\[
Z(\overline{A}_1, \ldots, \overline{A}_p, u_1, \ldots, u_p) =
Z(\overline{B}_1, \ldots, \overline{B}_p, v_1, \ldots, v_p)
\]
in $\aChow^p(X)'_{\RR}$.
We prove it by induction on $p$.
If $p=1$, then the assertion is obvious, so that we assume $p > 1$.
By the hypothesis of induction,
we have
\[
Z(\overline{A}_2, \ldots, \overline{A}_p, u_2, \ldots, u_p) =
Z(\overline{B}_2, \ldots, \overline{B}_p, v_2, \ldots, v_p)
\]
in $\aChow^{p-1}(X)'_{\RR}$,
which implies
\[
\acherncl_1(\overline{A}_1) \cdot  Z(\overline{A}_2, \ldots, \overline{A}_p, u_2, \ldots, u_p) =
(\acherncl_1(\overline{B}_1) - \acherncl_1(0, \phi_1)) \cdot Z(\overline{B}_2, \ldots, \overline{B}_p, v_2, \ldots, v_p)
\]
in $\aChow^{p-1}(X)'_{\RR}$. The left hand side is equal to
\begin{multline*}
Z(\overline{A}_1, \ldots, \overline{A}_p, u_1, \ldots, u_p) - \sum_{1 \in I \subseteq \{ 1, \ldots, p\} }
(0, udd^c(u_I) \wedge c_1(\overline{A}_J)) \\
= Z(\overline{A}_1, \ldots, \overline{A}_p, u_1, \ldots, u_p) - \sum_{I' \subseteq \{ 2, \ldots, p\}}
(0, u_1 dd^c(u_{I'}) \wedge c_1(\overline{A}_{J'})),
\end{multline*}
where $J' = \{2, \ldots, p\} \setminus I'$.
Moreover, the right hand side is equal to
\begin{multline*}
Z(\overline{B}_1, \ldots, \overline{B}_p, v_1, \ldots, v_p) - \sum_{I' \subseteq \{ 2, \ldots, p\}}
(0, v_1 dd^c(v_{I'}) \wedge c_1(\overline{B}_{J'}))\\
 - \acherncl_1(\overline{B}_2) \cdots \acherncl_1(\overline{B}_p) \cdot \acherncl_1(0, \phi_1) -
\sum_{\emptyset \not= I' \subseteq \{ 2, \ldots, p\} }
 \acherncl_1(0, \phi_1) \cdot (0, vdd^c(v_{I'}) \wedge c_1(\overline{B}_{J'})) \\
 =
Z(\overline{B}_1, \ldots, \overline{B}_p, v_1, \ldots, v_p) 
- \sum_{I' \subseteq \{ 2, \ldots, p\}}
(0, v_1 dd^c(v_{I'}) \wedge c_1(\overline{B}_{J'}))
-\sum_{I' \subseteq \{ 2, \ldots, p\}}
 (0, \phi_1 dd^c(v_{I'}) \wedge c_1(\overline{B}_{J'})) \\
 = Z(\overline{B}_1, \ldots, \overline{B}_p, v_1, \ldots, v_p) - \sum_{I' \subseteq \{ 2, \ldots, p\}}
(0, u_1 dd^c(v_{I'}) \wedge c_1(\overline{B}_{J'})).
\end{multline*}
in  $\aChow^{p-1}(X)'_{\RR}$. Therefore, we can see that
\[
Z(\overline{A}_1, \ldots, \overline{A}_p, u_1, \ldots, u_p)
- Z(\overline{B}_1, \ldots, \overline{B}_p, v_1, \ldots, v_p)
\]
is equal to
\[
\left(0, u_1 \sum_{I' \subseteq \{ 2, \ldots, p\} } \left( dd^c(u_{I'}) \wedge c_1(\overline{A}_{J'}) -  dd^c(v_{I'}) \wedge c_1(\overline{B}_{J'})\right) \right),
\]
which is zero by the following Lemma~\ref{lem:mult:linear:eqn}.
\end{proof}
Applying the above claim to the case where $p = d$,
the first assertion follows. The second assertion can be
easily checked by using its definition.

\medskip
(3)
For this purpose, it is sufficient to show that
$\adeg' (\overline{D}^d) = \avol(\overline{D})$ for 
$\overline{D} = (D, g) \in \aNef_{C^0}(X)_{\RR}$.
Let $\overline{A}$ be an ample arithmetic Cartier divisor of $C^{\infty}$-type.
We assume 
\[
\adeg'((\overline{D} + (1/n) \overline{A} )^{d+1}) = \avol(\overline{D}
+ (1/n)\overline{A})
\]
for all $n > 0$.
Then, using the continuity of $\avol$, we can see
$\adeg'(\overline{D}^{d}) = \avol(\overline{D})$.
Thus we may assume $D$ is ample, so that
there is a $D$-Green function $h$ such that
$\alpha :=c_1(D, h)$ is positive. We set $\overline{D}' = (D, h)$ and
$\phi = g - h$.
Then $\phi$ is continuous and $dd^c([\phi]) + \alpha \geq 0$.
Therefore, by \cite{BK} or \cite[Lemma~4.2]{MoArZariski}, 
we can take a sequence of $C^{\infty}$-functions
$\{ \phi_n \}$ such that $\lim_{n\to\infty} \Vert \phi_n - \phi\Vert_{\sup} = 0$, and that
$\phi \leq \phi_n$ and 
$\phi_n \in \admCont(X; \alpha)$ for all $n$.
We set  $\overline{D}_n = \overline{D}' + (0, \phi_n)$.
Then $\overline{D}_n$ is a nef arithmetic $\RR$-Cartier divisor of $C^{\infty}$-type, and hence
$\adeg' (\overline{D}_n^{d}) = \avol(\overline{D}_n)$ for all $n$ by 
\cite[Claim~6.4.2.2]{MoArZariski}.
As $\lim_{n\to\infty}  \avol(\overline{D}_n) = \avol(\overline{D})$ by the continuity of $\avol$,
it is sufficient to see that
\[
\lim_{n\to\infty} \adeg' (\overline{D}_n^{d}) = \adeg' ( \overline{D}^{d} ).
\]
Note that
\begin{multline*}
\adeg' (\overline{D}_n^{d}) = \adeg' ((\overline{D}' + (0, \phi_n) )^{d})
= \adeg' ({\overline{D}'}^{d}) 
+ \sum_{i=1}^{d} \binom{d}{i} \int_{X(\CC)}
\phi_n dd^c(\phi_n)^{i-1} \wedge \alpha^{d-i}.
\end{multline*}
In addition, by (3) in Lemma~\ref{lem:wedge:B:proper:current},
$\{ \phi_n dd^c(\phi_n)^{i-1} \wedge \alpha^{d-i}\}$ converges weakly
to 
\[
\phi dd^c([\phi])^{i-1} \wedge \alpha^{d-i}
\]
for each $i$. Thus we have the assertion.

\medskip
(3) By using the symmetry and multi-linearity of $\adeg(\overline{D}_1 \cdots \overline{D}_d)$,
it is sufficient to see that
\[
\adeg((0, \eta_1) \cdot \overline{D}_2 \cdots \overline{D}_d) = \frac{1}{2} \sum_{I \subseteq \{ 2, \ldots, d\}} \int_{X(\CC)} \eta_1 dd^c(u_I) \wedge c_1(\overline{D}_J),
\]
which is a straightforward calculation by using the definition in (2).
\end{proof}

\begin{Lemma}
\label{lem:mult:linear:eqn}
Let $V$ and $W$ be vector spaces over $\RR$ and
let $f : V^s \to W$ be a symmetric multi-linear map.
Let $a_1, \ldots, a_s, b_1, \ldots, b_s$ be elements of $V$.
For a subset $I$ of $\{ 1, \ldots, s \}$, we set
$I = \{ i_1, \ldots, i_k \}$ and $J = \{ j_1, \ldots, j_l \}$,
where $J = \{ 1, \ldots, s \} \setminus I$ and $k + l = s$.
Then
\[
f(a_{i_1}, \ldots, a_{i_k}, b_{j_1}, \ldots, b_{j_l})
\]
does not depend on the choice of $i_1, \ldots, i_k$ and
$ j_1, \ldots, j_l$, so that it is denoted by
$f(a_I, b_J)$.
Let $a_1, \ldots, a_s, b_1, \ldots, b_s, c_1, \ldots, c_s, d_1, \ldots, d_s$ be
elements of $V$.
We assume that there are $u_1, \ldots, u_s \in V$ 
such that $a_i = c_i + u_i$ and $b_i = d_i - u_i$ for all $i=1, \ldots, s$.
Then
\[
\sum_{I \subseteq \{ 1, \ldots, s \}} f(a_I, b_J) =
\sum_{I \subseteq \{ 1, \ldots, s \}} f(c_I, d_J).
\]
\end{Lemma}

\begin{proof}
We prove the lemma by induction on $s$.
If $s = 1$, then
\[
\sum_{I \subseteq \{ 1, \ldots, s \}} f(a_I, b_J)  = f(a_1) + f(b_1)
= f(c_1 + u_1) + f(d_1 - u_1) = f(c_1) + f(d_1) = \sum_{I \subseteq \{ 1, \ldots, s \}} f(c_I, d_J).
\]
Thus we assume $s > 1$.
By the hypothesis of induction, we have
\[
\sum_{I' \subseteq \{ 2, \ldots, s \}} f(a_1, a_{I'}, b_{J'}) =
\sum_{I' \subseteq \{ 2, \ldots, s \}} f(a_1, c_{I'}, d_{J'})
\]
and
\[
\sum_{I' \subseteq \{ 2, \ldots, s \}} f(b_1, a_{I'}, b_{J'}) =
\sum_{I' \subseteq \{ 2, \ldots, s \}} f(b_1, c_{I'}, d_{J'}),
\]
where $J' = \{ 2, \ldots, s \} \setminus I'$.
The first equation and the second equation imply that
\[
\sum_{1 \in I \subseteq \{ 1, \ldots, s \}} f(a_{I}, b_{J}) =
\sum_{1 \in I \subseteq \{ 1, \ldots, s \}} f(c_{I}, d_{J}) +
\sum_{I' \subseteq \{ 2, \ldots, s \}} f(u_1, c_{I'}, d_{J'})
\]
and
\[
\sum_{1 \not\in I \subseteq \{ 1, \ldots, s \}} f(a_{I}, b_{J}) =
\sum_{1 \not\in I \subseteq \{ 1, \ldots, s \}} f(c_{I}, d_{J}) -
\sum_{I' \subseteq \{ 2, \ldots, s \}} f(u_1, c_{I'}, d_{J'})
\]
respectively. Thus the lemma follows.
\end{proof}

\subsection{Hodge index theorem for arithmetic $\RR$-Cartier divisors}
\label{subsec:Hodge:index:thm}
\setcounter{Theorem}{0}
First of all, let us fix notation.
Let $\KK$ be either $\QQ$ or $\RR$.
Let $H$ be an ample $\KK$-Cartier divisor on $X$.
Let $D$ be a $\KK$-Cartier divisor on $X$ and let $E$ be a vertical $\KK$-Weil divisor on $X$.
We set $E = \sum_{i=1}^l a_i \Gamma_i$, where
$a_1, \ldots, a_l \in \KK$ and $\Gamma_1, \ldots, \Gamma_l$ are vertical prime divisors.
Then a quantity
\[
\sum_{i=1}^l a_i \deg\left( \left(\rest{H}{\Gamma_i}\right)^{d-2} \cdot \left(\rest{D}{\Gamma_i} \right) \right)
\]
is denoted by $\deg_H(D \cdot E)$.
Note that if $X$ is regular and $D$ and $E$ are vertical, then
$\deg_H(D \cdot E) = \deg_H(E \cdot D)$.
We say $D$ is {\em divisorially $\pi$-nef with respect to $H$} if
$\deg_H(D \cdot \Gamma) \geq 0$ for all vertical prime divisors $\Gamma$ on $X$.
Moreover, $D$ is said to be {\em divisorially $\pi$-numerically trivial with respect to $H$}
if $D$ and $-D$ is divisorially $\pi$-nef with respect to $H$, that is,
$\deg_H(D \cdot \Gamma) = 0$ for all vertical prime divisors $\Gamma$ on $X$.

\begin{Lemma}
\label{lem:Zariski:fiber}
We assume that $X$ is regular.
Let $P \in \Spec(O_K)$ and let $\pi^{-1}(P) = a_1 \Gamma_1 + \cdots + a_n \Gamma_n$
be the irreducible decomposition as a cycle, that is, $a_1, \ldots, a_n \in \ZZ_{>0}$ and
$\Gamma_1, \ldots, \Gamma_n$ are prime divisors.
Let us consider a linear map $T_P : \KK^n \to \KK^n$ given by
\[
\begin{pmatrix} x_1 \\ \vdots \\ x_n \end{pmatrix} \mapsto
\begin{pmatrix}  \deg_{H}(\Gamma_1 \cdot \Gamma_1) & \cdots & \deg_{H}(\Gamma_1 \cdot \Gamma_n) \\
\vdots & \ddots & \vdots \\
\deg_{H}(\Gamma_n \cdot \Gamma_1) & \cdots & \deg_{H}(\Gamma_n \cdot \Gamma_n)
\end{pmatrix}
\begin{pmatrix} x_1 \\ \vdots \\ x_n \end{pmatrix}
\]
Then $\Ker(T_P) = \langle (a_1, \ldots, a_n) \rangle_{\KK}$ and
$T_P(\KK^n) = \{ (y_1, \ldots, y_n) \in \KK^n \mid a_1 y_1 + \cdots + a_n y_n = 0 \}$.
\end{Lemma}

\begin{proof}
This is a consequence of Zariski's lemma (cf. Lemma~\ref{lem:Zariski:vector:sp}).
\end{proof}

\begin{Lemma}
\label{lem:vertical:trivial}
We assume that $X$ is regular.
Let $D$ be a $\KK$-Cartier divisor on $X$ with $\deg(H_{\QQ}^{d-2} \cdot D_{\QQ}) = 0$.
Then there is a vertical effective $\KK$-Cartier divisor $E$ such that
$D + E$ is divisorially $\pi$-numerically trivial with respect to $H$.
\end{Lemma}

\begin{proof}
We can choose $P_1, \ldots, P_n \in \Spec(O_K)$ such that
$\deg_H(D \cdot \Gamma) = 0$ for all vertical prime divisors $\Gamma$ with
$\pi(\Gamma) \not\in \{ P_1, \ldots, P_n \}$.
We set
$\pi^{-1}(P_k) = \sum_{i=1}^{n_k} a_{ki} \Gamma_{ki}$
for each $k = 1, \ldots, n$, where $a_{ki} \in \ZZ_{>0}$ and $\Gamma_{ki}$ is
a vertical prime divisor over $P_k$.
Since 
\[
\sum_{j=1}^{n_k} a_{kj} \deg_H \left(D \cdot \Gamma_{kj}\right) = \deg_H\left(D \cdot \pi^{-1}(P_k)\right) = 0,
\]
by virtue of Lemma~\ref{lem:Zariski:fiber},
we can find $x_{ki} \in \KK$
\[
\sum_{i=1}^{n_k} x_{ki} \deg_H\left(\Gamma_{ki} \cdot \Gamma_{kj}\right) = -\deg_H(D \cdot \Gamma_{kj})
\]
for all $k$. 
Moreover, 
replacing $x_{ki}$ by $x_{ki} + na_{ki}$ ($n \gg 1$),
we may assume that $x_{ki} > 0$.
Here we set
\[
E = \sum_{k=1}^n \sum_{i=1}^{n_k} x_{ki} \Gamma_{ki}.
\]
Then $D+E$ is divisorially $\pi$-numerically trivial.
\end{proof}

First let us consider the Hodge index theorem for $\RR$-Cartier divisors on an arithmetic surface.
It was actually treated in \cite[Theorem~5.5]{BostP}.
Here we would like to present a slightly different version.

\begin{Theorem}
\label{thm:Hodge:index:surface}
We assume $d=2$. Let $\Div_0(X_{\QQ})_{\RR}$ be a vector subspace of  $\Div(X_{\QQ})_{\RR}$ given by
\[
 \Div_0(X_{\QQ})_{\RR} := \{ \vartheta \in \Div(X_{\QQ})_{\RR} \mid \deg(\vartheta) = 0 \}.
\]
Let $\overline{D} = (D, g)$ be an arithmetic $\RR$-Cartier divisor in 
$\aDiv_{C^0}^{\Nef}(X)_{\RR}$ with
$D_{\QQ} \in \Div_0(X_{\QQ})_{\RR}$. Then
\[
\adeg(\overline{D}^2) \leq -2 [K : \QQ] \langle D_{\QQ}, D_{\QQ} \rangle_{NT},
\]
where $\langle\ ,\ \rangle_{NT}$ is the N\'{e}ron-Tate pairing on $\Div_0(X_{\QQ})_{\RR}$ 
\rom{(}cf. Remark~\rom{\ref{rem:Neron:Tate}}\rom{)}.
Moreover, the equality holds if and only if the following conditions \rom{(a)}, \rom{(b)} and \rom{(c)} hold:
\begin{enumerate}
\renewcommand{\labelenumi}{(\alph{enumi})}
\item
$D$ is divisorially $\pi$-numerically trivial.

\item 
$g$ is of $C^{\infty}$-type.

\item
$c_1(\overline{D}) = 0$.
\end{enumerate}
\end{Theorem}

\begin{proof}
Let $\mu : X' \to X$ be a resolution of singularities of $X$ (cf. \cite{Lip}).
Then, since the arithmetic volume function is invariant under birational morphisms (cf. \cite[Theorem~4.3]{MoCont}),
we can see $\adeg(\overline{D}^2) = \adeg(\mu^*(\overline{D})^2)$.
Thus we may assume that $X$ is regular.

Let $g'$ be an $F_{\infty}$-invariant $D$-Green function of $C^{\infty}$-type with
$c_1(D, g') = 0$.
Let $\eta$ be an $F_{\infty}$-invariant continuous function on $X(\CC)$ with $g = g' + \eta$.
Then, by (1) in Proposition~\ref{prop:intersection:capacity},
$\eta \in \langle (C^0 \cap \Tqpsh)(X(\CC)) \rangle_{\RR}$

By Lemma~\ref{lem:vertical:trivial},
we can find an effective and vertical $\RR$-Cartier divisor $E$ such that
$D + E$ is divisorially $\pi$-numerically trivial.
If we set $\overline{D}' = (D+E, g')$, then
$\overline{D}'$ satisfies the above conditions (a), (b) and (c).
Moreover, as $\overline{D} = \overline{D}' - (E,0) + (0,\eta)$,
\[
\adeg(\overline{D}^2) = \adeg({\overline{D}'}^2) + \adeg((E,0)^2) 
+ \frac{1}{2} \int_{X(\CC)}\eta dd^c(\eta).
\]
Thus, 
by Proposition~\ref{prop:sym:semidefinite:bilinear} and Zariski's lemma (cf. Lemma~\ref{lem:Zariski:vector:sp}),
in order to prove the assertions of the theorem,
it is sufficient to see
\[
\adeg(\overline{D}^2) = -2 [K : \QQ] \langle D_K, D_K \rangle_{NT}.
\]
under the assumptions (a), (b) and (c).

By (1) in Lemma~\ref{lem:Z:module:tensor:R},
we can choose $D_1, \ldots, D_l \in \Div(X)$ and $a_1, \ldots, a_l \in \RR$ such that
$D = a_1 D_1 + \cdots + a_l D_l$ and $a_1, \ldots, a_l$ are linearly independent over $\QQ$.
Let $C$ be a $1$-dimensional vertical closed integral subscheme.
Since
\[
0 = \deg(\rest{D}{C}) = a_1 \deg(\rest{D_1}{C}) + \cdots + a_n \deg(\rest{D_n}{C}),
\]
we have $\deg(\rest{D_i}{C}) = 0$ for all $i$, and hence $D_i$ 
is divisorially $\pi$-numerically trivial for every $i$, so that
we can also choose a $D_i$-Green function $h_i$ of $C^{\infty}$-type such that
$\overline{D} = a_1\overline{D}_1 + \cdots + a_l\overline{D}_l$ and $c_1(\overline{D}_i) = 0$ for all $i$, where
$\overline{D}_i = (D_i, h_i)$ for $i=1,\ldots,l$.
We need to show
\[
\adeg\left((a_1\overline{D}_1 + \cdots + a_l\overline{D}_l)^2\right) = -2 [K : \QQ] 
\langle a_1D_1 + \cdots + a_lD_l, a_1D_1 + \cdots + a_lD_l \rangle_{NT}.
\]
Note that it holds for $a_1, \ldots, a_l \in \QQ$ by Faltings-Hriljac (\cite{FaCAS}, \cite{Hri}). 
Moreover, each hand side is continuous with
respect to $a_1, \ldots, a_l$. Thus the equality follows in general.
\end{proof}

\begin{Remark}
\label{rem:Neron:Tate}
Let $\Div_0(X_{\QQ})$ be the group of divisors $\vartheta$ on $X_{\QQ}$ with $\deg(\vartheta) = 0$.
By using (1) in Lemma~\ref{lem:Z:module:tensor:R}, we can see $\Div_0(X_{\QQ})  \otimes_{\ZZ} \RR = \Div_0(X_{\QQ})_{\RR}$.
Let 
\[
\langle\ ,\ \rangle_{NT} : \Div_0(X_{\QQ}) \times \Div_0(X_{\QQ}) \to \RR
\]
be the N\'{e}ron-Tate height pairing on $\Div_0(X_{\QQ})$,
which extends to 
\[
\Div_0(X_{\QQ})_{\RR} \times \Div_0(X_{\QQ})_{\RR} \to \RR
\]
in the natural way.
By abuse of notation, the above bi-linear map is also denoted by $\langle\ ,\ \rangle_{NT}$.
By virtue of  \cite[Proposition~B.5.3]{HS},
we can see that 
\[
\PDiv(X_{\QQ})_{\RR} = \{ \vartheta \in \Div_0(X_{\QQ})_{\RR} \mid \langle \vartheta, \vartheta \rangle_{NT} = 0 \}.
\]
\end{Remark}

Finally let us consider the Hodge index theorem on a higher dimensional arithmetic variety.
The proof is almost same as \cite{MoHodge}, but we need a careful treatment at the final step.

\begin{Theorem}
\label{thm:Hodge:index:arith:var}
Let $\overline{D} = (D, g)$ be an arithmetic $\RR$-Cartier divisor in
$\aDiv_{C^0}^{\Nef}(X)_{\RR}$
and let
$\overline{H} = (H, h)$ be an ample arithmetic $\QQ$-Cartier divisor on $X$.
If $\deg(D_{\QQ} \cdot H_{\QQ}^{d-2}) = 0$, then
\[
\adeg (\overline{D}^2 \cdot \overline{H}^{d-2}) \leq 0.
\]
Moreover, if the equality holds, then $D_{\QQ} \in \aPDiv(X_{\QQ})_{\RR}$.
\end{Theorem}

\begin{proof}
By (1) in Lemma~\ref{lem:Z:module:tensor:R},
we can  choose $D_1, \ldots, D_l \in \Div(X)$ and $a_1, \ldots, a_l \in \RR$ such that
$a_1, \ldots, a_l$ are linearly independent over $\QQ$ and
$D = a_1 D_1 + \cdots + a_l D_l$. Since
\[
0 = \deg(D_{\QQ} \cdot H_{\QQ}^{d-2}) = \sum_{i=1}^l a_i \deg({D_i}_{\QQ} \cdot H_{\QQ}^{d-2})
\]
and $\deg({D_i}_{\QQ} \cdot H_{\QQ}^{d-2}) \in \QQ$ for all $i$,
we have $\deg({D_i}_{\QQ} \cdot H_{\QQ}^{d-2}) = 0$ for all $i$.
Let us also choose an $F_{\infty}$-invariant $D_i$-Green function $g_i$ of $C^{\infty}$-type 
such that $c_1(D_i, g_i) \wedge c_1(\overline{H})^{d-2} = 0$.
If we set $g' = a_1 g_1 + \cdots + a_l g_l$,
then, by (1) in Proposition~\ref{prop:intersection:capacity},
there is $\eta \in  \langle (C^0 \cap \Tqpsh)(X(\CC)) \rangle_{\RR}$ such that
$g = g' + \eta$. By (4) in Proposition~\ref{prop:intersection:capacity},
\[
\adeg (\overline{D}^2 \cdot \overline{H}^{d-2}) =
\adeg ((D, g')^2 \cdot \overline{H}^{d-2}) 
+ \frac{1}{2} \int_{X(\CC)} \eta dd^c(\eta) c_1(\overline{H})^{d-2}
\]
because $c_1(D, g') \wedge c_1(\overline{H})^{d-2} = 0$.
Therefore, by Proposition~\ref{prop:sym:semidefinite:bilinear},
\[
\adeg (\overline{D}^2 \cdot \overline{H}^{d-2}) \leq
\adeg ((D, g')^2 \cdot \overline{H}^{d-2}) 
\]
and the equality holds if and only if $\eta$ is a constant.
Thus we may assume that $\eta$ is a constant, that is,
$g = g'$ by replacing $g_l$ by $g_l + \eta/a_l$.

By virtue of \cite[Theorem~1.1]{MoHodge},
\[
\adeg \left( \left(\alpha_1 (D_1, g_1) + \cdots + \alpha_l(D_l, g_l)\right)^2 \cdot \overline{H}^{d-2} \right) \leq 0
\]
for all $\alpha_1, \ldots, \alpha_l \in \QQ$,
and hence $\adeg (\overline{D}^2 \cdot \overline{H}^{d-2}) \leq 0$.

We need to check the equality condition. We prove it by induction on $d$.
If $d=2$, then the assertion follows from Theorem~\ref{thm:Hodge:index:surface} and Remark~\ref{rem:Neron:Tate}.
We assume that $d > 2$ and
$\adeg (\overline{D}^2 \cdot \overline{H}^{d-2}) = 0$. 
By using arithmetic Bertini's theorem (cf. \cite{MoABG}), we can find $m \in \ZZ_{>0}$ and $f \in \Rat(X)^{\times}$ with the following
properties:
\begin{enumerate}
\renewcommand{\labelenumi}{(\roman{enumi})}
\item
If we set $\overline{H}' = (H', h') = m \overline{H} + \widehat{(f)}$,
then $(H', h') \in \aDiv_{C^{\infty}}(X)$, $H'$ is effective, $h' >  0$ and
$H'$ is smooth over $\QQ$. 

\item
If $H' = Y' + c_1 F_1 + \cdots + c_r F_r$ is the irreducible decomposition such that
$Y'$ is horizontal and $F_i$'s are vertical, then
$F_i$'s are connected components of smooth fibers over $\ZZ$.

\item
$D$ and $H'$ have no common irreducible component.
\end{enumerate}
Let $Y$ be the normalization of $Y'$.
Then
\begin{multline*}
0 = m^{d-2} \adeg(\overline{D}^2 \cdot \overline{H}^{d-2}) = \adeg(\overline{D}^2 \cdot {\overline{H}'}^{d-2})
= \adeg(\rest{\overline{D}}{Y}^2 \cdot {\rest{\overline{H}'}{Y}}^{d-3}) \\
+ \sum c_i \deg(\rest{D}{F_i}^2 \cdot \rest{H'}{F_i}^{d-3}) + 
\frac{1}{2}\int_{X(\CC)} h' c_1(\overline{D})^2 c_1(\overline{H}')^{d-3}.
\end{multline*}
Therefore, by using \cite[Lemma~1.1.2]{MoHodge},
we can see that $\adeg(\rest{\overline{D}}{Y}^2 \cdot {\rest{\overline{H}'}{Y}}^{d-3}) = 0$ and
$c_1(\overline{D}) = 0$. In particular, by hypothesis of induction,
$\rest{D_{\QQ}}{Y} \in \aPDiv(Y_{\QQ})_{\RR}$.
Let $C$ be a closed and integral curve of $X_{\QQ}$. 
Then, since
\[
0 = \int_{C(\CC)} c_1(\overline{D}) = \deg(D_{\QQ} \cdot C) =
\sum a_i \deg({D_i}_{\QQ} \cdot C)
\]
and $a_{1},\ldots,a_{l}$ are linearly independent over $\QQ$,
we have $\deg({D_i}_{\QQ} \cdot C) = 0$ for all $i$. Therefore, if we set $L_i = \OO_{X_{\QQ}}(D_i)$,
then $L_i$ is numerically trivial, and hence $(L_i)_{\CC}$ is also numerically trivial on $X(\CC)$.
This means that $(L_i)_{\CC}$ comes from a representation $\rho_i : \pi_1(X(\CC)) \to \CC^{\times}$.
Let $\iota$ be the natural homomorphism $\iota : \pi_1(Y(\CC)) \to \pi_1(X(\CC))$ and let
\[
\rho'_i = \rho_i \circ \iota : \pi_1(Y(\CC)) \overset{\iota}{\longrightarrow} \pi_1(X(\CC))  \overset{\rho_i}{\longrightarrow}\CC^{\times}.
\]
Then $\rho'_i$ yields $\rest{(L_i)_{\CC}}{Y(\CC)}$.
Let 
\[
\rho : \pi_1(X(\CC)) \to \CC^{\times} \otimes_{\ZZ} \RR\quad\text{and}\quad
\rho' : \pi_1(Y(\CC)) \to \CC^{\times} \otimes_{\ZZ} \RR
\]
be homomorphisms 
given by $\rho = \rho_1^{\otimes a_1} \cdots \rho_l^{\otimes a_l}$ and
$\rho' = {\rho'_1}^{\otimes a_1} \cdots {\rho'_l}^{\otimes a_l}$.
Since 
\[
\left(\rest{(L_1)_{\CC}}{Y(\CC)}\right)^{\otimes a_1} \otimes \cdots \otimes \left(\rest{(L_l)_{\CC}}{Y(\CC)}\right)^{\otimes a_l} = 1
\]
in $\Pic(Y_{\QQ}) \otimes {\RR}$,
we have $\rho' = 1$.
Note that $\iota$ is surjective (cf. \cite[Theorem~7.4]{Milnor} and the homotopy exact sequence).
Thus $\rho = 1$ because $\rho' = \rho \circ \iota$. Therefore, by (2) in Lemma~\ref{lem:Z:module:tensor:R}, the image of $\rho_i$ is finite for all $i$.
This means that there is a positive integer $n$ such that $(L_i)_{\CC}^{\otimes n} \simeq \OO_{X(\CC)}$ for 
all $i$.
If we fix $\sigma \in K(\CC)$, 
then
\[
\dim_{K} H^0(X_{\QQ}, L_i^{\otimes n}) = \dim_{\CC} H^0(X_{\QQ} \times_{\Spec(K)}^{\sigma} \Spec(\CC), L_i^{\otimes n} \otimes^{\sigma}_K \CC) = 1,
\]
and hence $L_i^{\otimes n} \simeq \OO_{X_{\QQ}}$ because $\deg(L_i \cdot H_{\QQ}^{d-2}) = 0$.
Therefore,
\[
L_1^{\otimes a_1} \otimes \cdots \otimes L_l^{\otimes a_l} = 
(L_1^{\otimes n})^{\otimes a_1/n} \otimes \cdots \otimes (L_l^{\otimes n})^{\otimes a_l/n} = 1
\]
in $\Pic(X_{\QQ})_{\RR}$. Thus $D_{\QQ} \in \PDiv(X_{\QQ})_{\RR}$.
\end{proof}

\begin{Remark}
\label{rem:typo:MoHodge}
There is a typo in \cite[Lemma~1.1.2]{MoHodge}.
The form $\omega$ should be real, that is,
$\bar{\omega} = \omega$.
\end{Remark}

\subsection{Hodge index theorem and pseudo-effectivity}
\setcounter{Theorem}{0}
In this subsection, let us observe the pseudo-effectivity of arithmetic 
$\RR$-Cartier divisors as an application of Hodge index theorem.
Let us begin with the following lemma:

\begin{Lemma}
\label{lem:ample:R:div}
We assume that $X$ is regular.
Let $\overline{D} =(D,g)$ be an arithmetic $\RR$-Cartier divisor of $C^0$-type.
If $D$ is semi-ample on $X_{\QQ}$ \rom{(}that is, there are semi-ample divisors $A_1, \ldots, A_r$ on $X_{\QQ}$ and
$a_1, \ldots, a_r \in \RR_{>0}$ such that $D_{\QQ} = a_1 A_1 + \cdots + a_r A_r$\rom{)}, then
there are $\varphi_1, \ldots, \varphi_l \in \Rat(X)^{\times}_{\RR}$ and $c \in \RR$ such that
$\overline{D} + \widehat{(\varphi_i)}_{\RR} + (0, c) \geq 0$ for all $i$ and
\[
\bigcap_{i=1}^l \Supp(D + (\varphi_i)_{\RR}) = \emptyset
\]
on $X_{\QQ}$ \rom{(}\,for the definition of $\Rat(X)^{\times}_{\RR}$ and arithmetic principal 
$\RR$-Cartier divisors, 
see Conventions and terminology~\rom{\ref{CV:R:principal:div}} and \rom{\ref{CV:arithmetic:divisors:2}}\rom{)}.
\end{Lemma}

\begin{proof}
Let us consider the assertion of the lemma for $\overline{D} = (D,g)$:

\renewcommand{\theequation}{$\ast$}
\begin{equation}
\label{eqn:lem:ample:R:div:1}
\begin{array}{l}
\text{There exist $\varphi_1, \ldots, \varphi_l \in \Rat(X)^{\times}_{\RR}$ and $c \in \RR$ such that} \\
\text{$\overline{D} + \widehat{(\varphi_i)}_{\RR} + (0, c) \geq 0$
for all $i$ and $\bigcap_{i=1}^l \Supp(D + (\varphi_i)_{\RR}) = \emptyset$ on $X_{\QQ}$.}
\end{array}
\end{equation}
\renewcommand{\theequation}{\arabic{section}.\arabic{subsection}.\arabic{Theorem}.\arabic{Claim}}

\begin{Claim}
\begin{enumerate}
\renewcommand{\labelenumi}{(\arabic{enumi})}
\item
If $D$ is a $\QQ$-Cartier divisor and $D$ is semi-ample on $X_{\QQ}$
\rom{(}i.e. $nD$ is base-point free on $X_{\QQ}$ for some $n > 0$\rom{)}, then
\eqref{eqn:lem:ample:R:div:1} holds for $\overline{D}$.

\item
If $D$ is vertical, then \eqref{eqn:lem:ample:R:div:1} holds for $\overline{D}$.

\item
If $a \in \RR_{>0}$ and \eqref{eqn:lem:ample:R:div:1} holds for $\overline{D}$, then
so does for $a\overline{D}$.

\item
If \eqref{eqn:lem:ample:R:div:1} holds for $\overline{D}$ and $\overline{D}'$,
so does for $\overline{D} + \overline{D}'$.
\end{enumerate}
\end{Claim}

\begin{proof}
(1) Since $D$ is a semi-ample $\QQ$-Cartier divisor on $X_{\QQ}$, 
there are a positive integer $n$ and 
$\phi_1, \ldots, \phi_l \in H^0(X, nD) \setminus \{ 0 \}$ such that
$\bigcap_{i=1}^l \Supp(nD + (\phi_i)) = \emptyset$ on $X_{\QQ}$.
Since $D + (\phi_i^{1/n})_{\RR}$ is effective, we can find $c \in \RR$ such that
$\overline{D} + \widehat{(\phi_i^{1/n})}_{\RR} +(0,c) \geq 0$ for all $i$.

(2) We choose $x \in O_K \setminus \{ 0 \}$ such that $D + (x) \geq 0$, and hence
there is $c \in \RR$ such that $\overline{D} + \widehat{(x)} + (0,c) \geq 0$.

(3) Let $\varphi_1, \ldots, \varphi_l \in \Rat(X)^{\times}_{\RR}$ and $c \in \RR$ such that
$\overline{D} + \widehat{(\varphi_i)}_{\RR} + (0, c) \geq 0$
for all $i$ and $\bigcap_{i=1}^l \Supp(D + (\varphi_i)_{\RR}) = \emptyset$ on $X_{\QQ}$.
Then $a\overline{D} + \widehat{(\varphi_i^a)}_{\RR} + (0, ac) \geq 0$
for all $i$ and $\bigcap_{i=1}^l \Supp(aD + (\varphi_i^a)_{\RR}) = \emptyset$ on $X_{\QQ}$.

(4) By our assumption,
there exist $\varphi_1, \ldots, \varphi_l, \varphi'_1, \ldots, \varphi'_{l'} \in \Rat(X)^{\times}_{\RR}$ and 
$c, c' \in \RR$ such that
\[
\begin{cases}
\text{$\overline{D} + \widehat{(\varphi_i)}_{\RR} + (0, c) \geq 0$ for all $i$}.\\
\text{$\bigcap_{i=1}^l \Supp(D + (\varphi_i)_{\RR}) = \emptyset$ on $X_{\QQ}$},\\
\text{$\overline{D}' + \widehat{(\varphi'_j)}_{\RR} + (0, c') \geq 0$
for all $j$}, \\
\text{$\bigcap_{j=1}^{l'} \Supp(D' + (\varphi'_j)_{\RR}) = \emptyset$
on $X_{\QQ}$}.
\end{cases}
\]
Then $\overline{D} + \overline{D}' + \widehat{(\varphi_i  \varphi'_j)}_{\RR} + (0, c + c') \geq 0$ for all $i,j$ and
\[
\bigcap_{i,j} \Supp(D +  D' + (\varphi_i \varphi'_j)_{\RR}) = \emptyset
\]
on $X_{\QQ}$ because
\[
\bigcap_{i,j} \Supp(D + D' + (\varphi_i \varphi'_j)_{\RR}) \subseteq \bigcap_{i,j} \left( \Supp(D +  (\varphi_i)_{\RR}) \cup \Supp(D' + (\varphi'_{j})_{\RR})\right).
\]
\end{proof}

Let us go back to the proof of the lemma.
Since $X$ is regular and $D$ is semi-ample on $X_{\QQ}$, there are arithmetic
$\QQ$-Cartier divisors 
$\overline{D}_1, \ldots, \overline{D}_r$ of $C^0$-type, $a_1,\ldots, a_r \in \RR_{>0}$,
a vertical $\RR$-Cartier divisor $E$ and 
an $F_{\infty}$-invariant continuous function $\varphi$ on $X(\CC)$ such that
$D_i$'s are semi-ample on $X_{\QQ}$ and $\overline{D} = a_1\overline{D}_1 + \cdots + a_r\overline{D}_r + (E,\varphi)$.
Thus the assertion follows from the above claim.
\end{proof}

Let us fix an ample arithmetic $\QQ$-Cartier divisor $\overline{H}$ on $X$.
For arithmetic $\RR$-Cartier divisors $\overline{D}_1$ and $\overline{D}_2$ of 
$C^{\infty}$-type on $X$,
we denote $\adeg(\overline{H}^{d-2} \cdot \overline{D}_1 \cdot \overline{D}_2)$ by
$\adeg_{\overline{H}}(\overline{D}_1 \cdot \overline{D}_2)$.
Let us consider the following lemma, which is a useful criterion of pseudo-effectivity.

\begin{Lemma}
\label{lem:self:negative:not:peff}
We assume that $X$ is regular.
Let $\overline{D} = (D, g)$ be an arithmetic $\RR$-Cartier divisor of $C^{\infty}$-type on $X$ with the following properties:
\begin{enumerate}
\renewcommand{\labelenumi}{(\arabic{enumi})}
\item
$D$ is nef on $X_{\QQ}$ and $\deg(D_{\QQ} \cdot H_{\QQ}^{d-2}) = 0$.

\item
$c_1(\overline{D})$ is semipositive.

\item
$D$ is divisorially $\pi$-nef with respect to $H$.

\item
$\adeg_{\overline{H}}(\overline{D}^2) < 0$.
\end{enumerate}
Then $\overline{D}$ is not pseudo-effective.
\end{Lemma}

\begin{proof}
Let $\overline{B}$ be an ample arithmetic $\RR$-Cartier divisor on $X$.
We set $\overline{A} = \overline{D} + \overline{B}$.
Since 
\[
\adeg_{\overline{H}}(\overline{D}^2) < 0,
\]
we have
\[
\adeg_{\overline{H}}(\overline{A} \cdot \overline{D}) < \adeg_{\overline{H}}(\overline{B}\cdot \overline{D}).
\]
Here we claim the following:

\begin{Claim}
There is an arithmetic $\RR$-Cartier divisor $\overline{L} = (L, h)$ of $C^{\infty}$-type with the following properties:
\begin{enumerate}
\renewcommand{\labelenumi}{(\alph{enumi})}
\item
$L$ is ample on $X_{\QQ}$.

\item
$c_1(\overline{L})$ is positive.

\item
$L$ is divisorially $\pi$-nef with respect to $H$.

\item
$\adeg_{\overline{H}}(\overline{L}\cdot \overline{D}) < 0$.
\end{enumerate}
\end{Claim}

\begin{proof}
If $\adeg_{\overline{H}}( \overline{B}\cdot \overline{D}) \leq 0$, then $\adeg_{\overline{H}}(\overline{A}\cdot \overline{D}) < 0$.
Thus the assertion is obvious if we set $\overline{L} = \overline{A}$.

Next we assume that $\adeg_{\overline{H}}(\overline{B}\cdot \overline{D}) >0$.
Since
$\adeg_{\overline{H}}( \overline{A}\cdot \overline{D})/\adeg_{\overline{H}}( \overline{B}\cdot \overline{D}) < 1$,
we choose a positive number $\alpha$ such that
\[
\frac{\adeg_{\overline{H}}( \overline{A}\cdot \overline{D})}{\adeg_{\overline{H}}( \overline{B}\cdot \overline{D})} < \alpha < 1.
\]
We set $\overline{L} = \overline{D} + (1-\alpha)\overline{B}$.
The conditions (a), (b) and (c) are obvious.
Moreover, since $\overline{L} = \overline{A} - \alpha \overline{B}$,
\begin{align*}
\adeg_{\overline{H}}(\overline{L}\cdot \overline{D}) & = \adeg_{\overline{H}}( \overline{A}\cdot \overline{D}) - \alpha \adeg_{\overline{H}}(\overline{B}\cdot \overline{D}) \\
& < \adeg_{\overline{H}}( \overline{A}\cdot \overline{D}) - \frac{\adeg_{\overline{H}}(\overline{D})}{\adeg_{\overline{H}}( \overline{B}\cdot \overline{D})} \adeg_{\overline{H}}( \overline{B}\cdot \overline{D}) = 0.
\end{align*}
\end{proof}

Let us go back to the proof of the lemma.
Since $L$ is ample on $X_{\QQ}$, by Lemma~\ref{lem:ample:R:div},
there are  $\varphi_1, \ldots, \varphi_l \in \Rat(X)^{\times}_{\RR}$ and $c \in \RR$ such that
$\overline{L} + \widehat{(\varphi_i)}_{\RR} +(0,c) \geq 0$ for all $i$ and $\bigcap_{i=1}^l \Supp(L + (\varphi_i)_{\RR}) = \emptyset$ on $X_{\QQ}$.
Let $\Gamma$ be a horizontal prime divisor. Then we can find $i$ such that 
$\Gamma \not\subseteq \Supp(L + (\varphi_i)_{\RR})$.
Thus
\begin{multline*}
\adeg_{\overline{H}}((\overline{L} + (0,c)) \cdot (\Gamma, 0))
= \adeg_{\overline{H}}((\overline{L} + \widehat{(\varphi_i)}_{\RR} + (0,c)) \cdot (\Gamma, 0)) \\
= \adeg\left(\rest{\overline{H}}{\Gamma}^{d-2} \cdot \rest{(\overline{L} + \widehat{(\varphi_i)}_{\RR} + (0,c))}{\Gamma}\right) \geq 0.
\end{multline*}
Furthermore,
the above inequality also holds for a vertical prime divisor $\Gamma$
because $L$ is divisorially $\pi$-nef with respect to $H$.
Therefore, if $\overline{G} = (G,k)$ is an effective arithmetic $\RR$-Cartier divisor of $C^0$-type,
then 
\[
\adeg_{\overline{H}}( (\overline{L} + (0,c))\cdot \overline{G}) = \adeg_{\overline{H}}( (\overline{L} + (0,c))\cdot (G,0)) +
\frac{1}{2} \int_{X(\CC)} k c_1(\overline{H})^{d-2} c_1(\overline{L}) \geq 0.
\]
In particular,
if $\overline{D}$ is pseudo-effective, then
\[
\adeg_{\overline{H}}( (\overline{L} + (0,c))\cdot \overline{D}) \geq 0.
\]
On the other hand, as $\deg(D_{\QQ} \cdot H_{\QQ}^{d-2}) = 0$,
\begin{align*}
\adeg_{\overline{H}}( (\overline{L} + (0,c))\cdot \overline{D}) & = \adeg_{\overline{H}}( \overline{L}\cdot \overline{D}) + \frac{c}{2}\deg(D_{\QQ} \cdot H_{\QQ}^{d-2}) \\
&=
\adeg_{\overline{H}}(\overline{L}\cdot \overline{D}) < 0.
\end{align*}
This is a contradiction.
\end{proof}

As consequence of Hodge index theorem and the above lemma, we have the following theorem
on pseudo-effectivity:

\begin{Theorem}
\label{thm:pseudo:principal}
We assume that $X$ is regular and $d \geq 2$.
Let $\overline{D} = (D,g)$ be an arithmetic $\RR$-Cartier divisor of $C^{0}$-type.
If $\overline{D}$ is 
pseudo-effective and $D$ is numerically trivial on $X_{\QQ}$,
then $D_{\QQ} \in \PDiv(X_{\QQ})_{\RR}$. 
\end{Theorem}

\begin{proof}
We assume that $D_{\QQ} \not\in \PDiv(X_{\QQ})_{\RR}$.
Since $D$ is numerically trivial on $X_{\QQ}$, by Lemma~\ref{lem:vertical:trivial},
we can find an effective vertical $\RR$-Cartier divisor $E$ such that
$D + E$ is divisorially  $\pi$-numerically trivial with respect to $H$.
Moreover, we can find an
$F_{\infty}$-invariant  $D$-Green function $g_0$ of $C^{\infty}$-type with $c_1(D,g_0) = 0$.
Then there is an $F_{\infty}$-invariant continuous function $\eta$ on $X(\CC)$ such that $g + \eta = g_0$.
Replacing $g_0$ by $g_0 + c$ ($c \in \RR$), we may assume that $\eta \geq 0$.
By the Hodge index theorem,
\[
\adeg_{\overline{H}}((D + E, g_0)^2) < 0.
\]
Thus $(D+ E, g_0)$ is not pseudo-effective by Lemma~\ref{lem:self:negative:not:peff}, and hence
\[
\overline{D} = (D+E, g_0) - (E, \eta)
\]
is also not pseudo-effective. This is a contradiction.
\end{proof}

Finally let us consider the following lemmas on pseudo-effectivity.

\begin{Lemma}
\label{lem:pseudo:plus:principal}
For $\overline{D} \in \aDiv_{C^0}(X)_{\RR}$ and $z \in \aPDiv(X)_{\RR}$,
if $\overline{D}$ is pseudo-effective, then $\overline{D} +z$ is also pseudo-effective.
\end{Lemma}

\begin{proof}
Let $\overline{A}$ be an ample arithmetic $\RR$-Cartier divisor on $X$.
Since $\overline{D}$ is pseudo-effective,
$\overline{D} + (1 /2) \overline{A}$ is big.
Moreover, $z + (1 /2) \overline{A}$ is ample because $z$ is nef.
Therefore,
\[
(\overline{D} + z) + \overline{A} = (\overline{D} + (1 /2) \overline{A}) + (z + (1 /2) \overline{A})
\]
is big, as required.
\end{proof}

\begin{Lemma}
\label{lem:pseudo:effective:cont:func}
Let $D$ be a vertical $\RR$-Cartier divisor on $X$ and let 
$\eta$ be an $F_{\infty}$-invariant continuous function on $X(\CC)$. 
Let $\lambda$ be an element of  $\RR^{K(\CC)}$ given by
$\lambda_{\sigma} = \inf_{x \in X_{\sigma}} \eta(x)$ for all $\sigma \in K(\CC)$.
We can view $\lambda$ as a locally constant function on $X(\CC)$, that is,
$\rest{\lambda}{X_{\sigma}} = \lambda_{\sigma}$.
If $(D,\eta)$ is pseudo-effective, then $(D, \lambda)$ is also pseudo-effective.
\end{Lemma}

\begin{proof}
Let us begin with the following claim:

\begin{Claim}
We may assume that $\lambda$ is a constant function.
\end{Claim}

\begin{proof}
We set $\lambda' = (1/[K:\QQ])\sum_{\sigma \in K(\CC)} \lambda_{\sigma}$ and
$\xi_{\sigma} = \lambda' - \lambda_{\sigma}$ for each $\sigma \in K(\CC)$.
Then $\sum_{\sigma \in K(\CC)} \xi_{\sigma} = 0$ and
$\xi_{\sigma} = \xi_{\bar{\sigma}}$ for all $\sigma \in K(\CC)$.
Thus, by Dirichlet's unit theorem (cf. Corollary~\ref{cor:Dirichlet:unit:theorem}), 
there are $a_{1}, \ldots, a_{s} \in \RR$ and
$u_{1}, \ldots, u_{s} \in O_{K}^{\times}$ such that
\[
\xi_{\sigma} = a_{1}\log\vert \sigma(u_{1})\vert + \cdots + a_{s}\log\vert \sigma(u_{s})\vert 
\]
for all $\sigma \in K(\CC)$. If we set
\[
(D, \eta') = (D, \eta) - \pi^{*}\left((a_{1}/2)\widehat{(u_{1})} + \cdots + (a_{s}/2)\widehat{(u_{s})}\right), 
\]
then $\inf_{x \in X_{\sigma}} \eta'(x) = \lambda'$ for all
$\sigma \in K(\CC)$. Moreover, by Lemma~\ref{lem:pseudo:plus:principal},
$(D, \eta')$ is pseudo-effective.
If the lemma holds for $\eta'$,
then $(D, \lambda')$ is pseudo-effective, and hence
\[
(D, \lambda) = (D, \lambda') + \pi^{*}\left((a_{1}/2)\widehat{(u_{1})} + \cdots + (a_{s}/2)\widehat{(u_{s})}\right)
\]
is also pseudo-effective by Lemma~\ref{lem:pseudo:plus:principal}.
\end{proof}

For a given positive number $\epsilon$, we set
\[
U_{\sigma} = \{ x \in X_{\sigma} \mid \eta(x) < \lambda_{\sigma} + (\epsilon/2) \}
\]
and $U = \coprod_{\sigma \in K(\CC)} U_{\sigma}$.
Let $\overline{A} = (A, h)$ be an ample arithmetic Cartier divisor on $X$.
Then, by Lemma~\ref{lem:comp:norm:X:U},
there is a constant $C \geq 1$ depending only on $\epsilon$ and $h$
such that
\addtocounter{Claim}{1}
\begin{equation}
\label{eqn:lem:pseudo:effective:cont:func:1}
\sup_{x \in X(\CC)} \left\{ \vert s \vert^2_{t + bh}(x) \right\}
\leq C^b \sup_{x \in U} \left\{ \vert s \vert^2_{t + bh}(x) \right\}
\end{equation}
for all $s \in H^0(X(\CC), bA)$, $b \in \RR_{\geq 0}$ and all constant functions $t$ on $X(\CC)$.
Let $n$ be an arbitrary positive integer with $n \geq (2\log(C))/\epsilon$.
Since $(D, \eta) + (1/n) \overline{A}$ is big,
there are a positive integer $m$ and $s \in H^0(X, mD + (m/n)A) \setminus \{ 0 \}$ such that
$\vert s \vert_{m\eta + (m/n)h} \leq 1$,
which implies that
\[
\vert s \vert^2_{(m/n)h} \leq \exp(m\eta).
\]
Therefore, $\vert s \vert^2_{(m/n)h} \leq \exp(m(\lambda + (\epsilon/2)))$
over $U$, that is,
\[
\sup_{x \in U} \left\{ \vert s \vert^2_{m(\lambda + (\epsilon/2)) + (m/n)h} \right\} \leq 1.
\]
Thus, by the estimation~\eqref{eqn:lem:pseudo:effective:cont:func:1}, we have
\[
C^{-(m/n)} \sup_{x \in X(\CC)} \left\{ \vert s \vert^2_{m(\lambda + (\epsilon/2)) + (m/n)h} \right\} \leq 1.
\]
Since $\log(C)/n \leq \epsilon/2$,
\begin{align*}
\sup_{x \in X(\CC)} \left\{ \vert s \vert^2_{m(\lambda + \epsilon) + (m/n)h} \right\} & \leq
\sup_{x \in X(\CC)} \left\{ \vert s \vert^2_{(m/n) \log(C) + m(\lambda + (\epsilon/2)) + (m/n)h} \right\} \\
& = C^{-(m/n)} \sup_{x \in X(\CC)} \left\{ \vert s \vert^2_{m(\lambda + (\epsilon/2)) + (m/n)h} \right\} \leq 1
\end{align*}
which yields
$\aH(X, m ( (1/n) \overline{A} + (D, \lambda + \epsilon))) \not= \{ 0 \}$.
Thus $(D, \lambda + \epsilon) + (1/n)\overline{A}$ is big if $n \gg 1$.
As a consequence, $(D, \lambda + \epsilon)$ is pseudo-effective for any positive number $\epsilon$,
and hence
$(D, \lambda)$ is also pseudo-effective.
\end{proof}

\section{Dirichlet's unit theorem on arithmetic varieties}
Let us fix notation throughout this section.
Let $X$ be a $d$-dimensional, generically smooth, normal and projective arithmetic variety.
Let 
\[
X \overset{\pi}{\longrightarrow} \Spec(O_K) \to \Spec(\ZZ)
\]
be the Stein factorization of $X \to \Spec(\ZZ)$, where $K$ is a number field and $O_K$ is the ring of
integers in $K$.

\subsection{Fundamental question}
\label{subsec:fundamental:question}
\setcounter{Theorem}{0}

Let $\KK$ be either $\QQ$ or $\RR$.
As in Conventions and terminology~\ref{CV:R:principal:div},
we set 
\[
\Rat(X)^{\times}_{\KK} := \Rat(X)^{\times} \otimes_{\ZZ} \KK,
\]
whose element is called a {\em $\KK$-rational function} on $X$.
Note that the zero function is not a $\KK$-rational function.
Let
\[
(\ )_{\KK} : \Rat(X)^{\times}_{\KK} \to \Div(X)_{\KK}\quad\text{and}\quad
\widehat{(\ )}_{\KK} : \Rat(X)^{\times}_{\KK} \to \aDiv_{C^{\infty}}(X)_{\KK}
\]
be the natural extensions of the homomorphisms 
\[
\Rat(X)^{\times} \to \Div(X)\quad\text{and}\quad
\Rat(X)^{\times} \to \aDiv_{C^{\infty}}(X)
\]
given by $\phi \mapsto (\phi)$
and $\phi \mapsto \widehat{(\phi)}$ respectively.
Note that
\[
\PDiv(X)_{\KK} = \left\{ (\varphi)_{\KK} \mid \varphi \in \Rat(X)^{\times}_{\KK}\right\}\quad\text{and}\quad
\aPDiv(X)_{\KK} = \left\{ \widehat{(\varphi)}_{\KK} \mid \varphi \in \Rat(X)^{\times}_{\KK}\right\}
\]
(cf. Conventions and terminology~\ref{CV:R:principal:div} and \ref{CV:arithmetic:divisors:2}).
Let $\overline{D} = (D, g)$ be an arithmetic $\RR$-Cartier divisor of $C^0$-type.
We define 
$\Gamma^{\times}(X,D)$, $\widehat{\Gamma}^{\times}(X,\overline{D})$,
$\Gamma^{\times}_{\KK}(X, D)$ and $\widehat{\Gamma}^{\times}_{\KK}(X, \overline{D})$ to be
\[
\begin{cases}
\Gamma^{\times}(X, D) := \big\{ \phi \in \Rat(X)^{\times} \mid D + (\phi) \geq 0 \big\} = H^0(X,D) \setminus \{ 0\}, \\ 
\widehat{\Gamma}^{\times}(X, \overline{D}) := \left\{ \phi \in \Rat(X)^{\times} \mid \overline{D} + \widehat{(\phi)} \geq 0 \right\} = \aH(X, \overline{D}) \setminus \{ 0 \}, \\
\Gamma^{\times}_{\KK}(X, D) := \left\{ \varphi \in \Rat(X)^{\times}_{\KK} \mid D + (\varphi)_{\KK} \geq 0 \right\}, \\ 
\widehat{\Gamma}^{\times}_{\KK}(X, \overline{D}) := \left\{ \varphi \in \Rat(X)^{\times}_{\KK} \mid \overline{D} + \widehat{(\varphi)}_{\KK} \geq 0 \right\}.
\end{cases}
\]
Let us consider a homomorphism
\[
\ell : \Rat(X)^{\times} \to L^1_{loc}(X(\CC))
\]
given by $\phi \mapsto \log \vert \phi \vert$.
It extends to a linear map
\[
\ell_{\KK} : \Rat(X)^{\times}_{\KK} \to L^1_{loc}(X(\CC)).
\]
For $\varphi \in \Rat(X)^{\times}_{\KK}$, we denote $\exp(\ell_{\KK}(\varphi))$ by $\vert \varphi \vert$.
First let us consider the following lemma.

\begin{Lemma}
\label{lem:rational:tensor:real:number}
\begin{enumerate}
\renewcommand{\labelenumi}{(\arabic{enumi})}
\item
If $\varphi \in \Gamma^{\times}_{\KK}(X, D)$, then
$\vert \varphi \vert \exp(-g/2)$ is continuous on $X(\CC)$, so that
we define $\Vert \varphi \Vert_{g, \sup}$ to be
\[
\Vert \varphi \Vert_{g, \sup} := \max \left\{ \left(\vert \varphi \vert \exp(-g/2)\right)(x) \mid x \in X(\CC) \right\}.
\]

\item
$\widehat{\Gamma}^{\times}_{\KK}(X, \overline{D}) = \left\{ \varphi \in \Gamma^{\times}_{\KK}(X, D) \mid \Vert \varphi \Vert_{g, \sup} \leq 1 \right\}$.

\item We have the following formulae in $\Rat(X)^{\times}_{\QQ}$ or $\Rat(X)^{\times}_{\RR}$:
\[
\hspace{3em}
\begin{cases}
\Gamma^{\times}_{\QQ}(X, D) = \bigcup_{n > 0}  \Gamma^{\times}(X, nD)^{1/n}, &
\widehat{\Gamma}^{\times}_{\QQ}(X, \overline{D}) = \bigcup_{n > 0} \widehat{\Gamma}^{\times}(X, n\overline{D})^{1/n}, \\
\Gamma^{\times}_{\QQ}(X, \alpha D) = \Gamma^{\times}_{\QQ}(X,D)^{\alpha}, & 
\widehat{\Gamma}^{\times}_{\QQ}(X, \alpha D) = \widehat{\Gamma}^{\times}_{\QQ}(X,D)^{\alpha} \hspace{2em} (\alpha \in \QQ_{>0}), \\
\Gamma^{\times}_{\RR}(X, aD) = \Gamma^{\times}_{\RR}(X,D)^a, & 
\widehat{\Gamma}^{\times}_{\RR}(X, aD) = \widehat{\Gamma}^{\times}_{\RR}(X,D)^a \hspace{2.3em} (a \in \RR_{>0}). \\
\end{cases}
\]
\end{enumerate}
\end{Lemma}

\begin{proof}
(1) We set $D = a_1 D_1 + \cdots + a_n D_n$ and
$\varphi = \varphi_1^{x_1} \cdots \varphi_l^{x_l}$, where $D_1, \ldots, D_n$ are prime divisors,
$\varphi_1, \ldots, \varphi_l \in \Rat(X)^{\times}$
and $a_1, \ldots, a_n, x_1, \ldots, x_l \in \KK$.
Let $f_1, \ldots, f_n$ be local equations of $D_1, \ldots, D_n$ around $P \in X(\CC)$.
Then there is a local continuous function $h$ such that
$g = -\sum_{i=1}^n a_i \log \vert f_i \vert^2 + h$ around $P$.
Here let us see that $\vert \varphi_1 \vert^{x_1} \cdots \vert \varphi_l \vert^{x_l}
\vert f_1 \vert^{a_1} \cdots \vert f_n \vert^{a_n}$ is continuous around $P$.
We set $f_i = u_i t_1^{\alpha_{i1}} \cdots t_r^{\alpha_{ir}}$ and
$\varphi_j = v_j t_1^{\beta_{j1}} \cdots t_r^{\beta_{jr}}$,
where $\alpha_{ik}, \beta_{jk} \in \ZZ$, $u_1, \ldots, u_n, v_1, \ldots, v_l$ are units of $\OO_{X(\CC), P}$ and $t_1, \ldots, t_r$ are prime elements of $\OO_{X(\CC), P}$.
Then
\begin{multline*}
\vert \varphi_1 \vert^{x_1} \cdots \vert \varphi_l \vert^{x_l}
\vert f_1 \vert^{a_1} \cdots \vert f_n \vert^{a_n} \\
= \vert u_1 \vert^{a_1} \cdots \vert u_n\vert^{a_n} \vert v_1 \vert^{x_1} \cdots \vert v_l\vert^{x_l} \vert t_1 \vert^{\sum_i a_i \alpha_{i1} + \sum_{j} x_j \beta_{j1}}
\cdots \vert t_r \vert^{\sum_i a_i \alpha_{ir} + \sum_{j} x_j \beta_{jr}}.
\end{multline*}
On the other hand, as
\[
D + (\varphi)_{\KK} = \left( \sum_i a_i \alpha_{i1} + \sum_{j} x_j \beta_{j1} \right)(t_1) +
\cdots + \left( \sum_i a_i \alpha_{ir} + \sum_{j} x_j \beta_{jr} \right)(t_r) \geq 0
\]
around $P$, we have
\addtocounter{Claim}{1}
\begin{equation}
\label{eqn:lem:rational:tensor:real:number:1}
\sum_i a_i \alpha_{i1} + \sum_{j} x_j \beta_{j1} \geq 0, \ \ldots\ ,
\sum_i a_i \alpha_{ir} + \sum_{j} x_j \beta_{jr} \geq 0.
\end{equation}
Thus the assertion follows. Therefore,
\[
\vert \varphi \vert \exp(-g/2) = \vert \varphi_1 \vert^{x_1} \cdots \vert \varphi_l \vert^{x_l}
\vert f_1 \vert^{a_1} \cdots \vert f_n \vert^{a_n} \exp(-h/2)
\]
is also continuous around $P$.

(2) We use the same notation as in (1). Note that
\[
\overline{D} + \widehat{(\varphi)}_{\KK} = \left(D + (\varphi)_{\KK}, g + \sum_{i=1}^n x_i (-\log \vert \varphi_i \vert^2)\right).
\]
Moreover,
\[
g + \sum_{i=1}^n x_i (-\log \vert \varphi_i \vert^2) = - \log (\vert \varphi_1 \vert^{2x_1} \cdots \vert \varphi_l \vert^{2x_l} 
\vert f_1 \vert^{2a_1} \cdots \vert f_n \vert^{2a_n} \exp(-h))
\]
locally. Thus $\Vert \varphi \Vert_{g, \sup} \leq 1$ if and only if
$g + \sum_{i=1}^n x_i (-\log \vert \varphi_i \vert^2) \geq 0$, and hence (2) follows.

(3) For $\varphi \in \Rat(X)^{\times}_{\RR}$ and $a \in \RR_{>0}$, $D + (\varphi)_{\RR} \geq 0$ (resp. $\overline{D} + \widehat{(\varphi)}_{\RR} \geq 0$)
if and only if $aD + (\varphi^a)_{\RR} \geq 0$ (resp. $a\overline{D} + \widehat{(\varphi^a)}_{\RR} \geq 0$).
Thus the assertions in (3) are obvious.
\end{proof}

\begin{Remark}
\label{rem:product:K:R}
We assume $d =1$, that is, $X = \Spec(O_K)$.
For $P \in \Spec(O_K) \setminus \{ 0 \}$ and 
$\sigma \in K(\CC)$, the homomorphisms $\ord_P : K^{\times} \to \ZZ$ and $\vert \cdot \vert_{\sigma} : K^{\times} \to \RR^{\times}$
given by $\phi \mapsto \ord_P(\phi)$ and $\phi \mapsto \vert \sigma(\phi) \vert$ naturally extend to homomorphisms 
$K^{\times} \otimes_{\ZZ} \RR \to \RR$ and
$K^{\times} \otimes_{\ZZ} \RR \to \RR^{\times}$ respectively.
By abuse of notation, we denote them by $\ord_P$ and $\vert \cdot \vert_{\sigma}$ respectively.
Clearly, for $\varphi \in K^{\times} \otimes_{\ZZ} \RR$, $\vert \varphi \vert_{\sigma}$ is the value of $\vert \varphi \vert$ at $\sigma$.
Moreover, by using the product formula on $K^{\times}$, we can see
\addtocounter{Claim}{1}
\begin{equation}
\label{eqn:rem:product:K:R:1}
\prod_{\sigma \in K(\CC)} \vert \varphi \vert_{\sigma} = \prod_{P \in \Spec(O_K) \setminus \{ 0 \}} \#(O_K/P)^{\ord_{P}(\varphi)}
\end{equation}
for $\varphi \in K^{\times} \otimes_{\ZZ} \RR$
\end{Remark}

Finally we would like to propose the fundamental question as in ``Introduction''.

\renewcommand{\theTheorem}{\!\!}
\begin{FQuestion}
\label{prob:fundamental:question}
Let $\overline{D}$ be an arithmetic $\RR$-Cartier divisor of $C^0$-type.
Are the following equivalent ?
\begin{enumerate}
\renewcommand{\labelenumi}{(\arabic{enumi})}
\item
$\overline{D}$ is pseudo-effective.

\item
$\widehat{\Gamma}^{\times}_{\RR}(X, \overline{D}) \not= \emptyset$.
\end{enumerate}
\end{FQuestion}
\renewcommand{\theTheorem}{\arabic{section}.\arabic{subsection}.\arabic{Theorem}}

Clearly (2) implies (1). Indeed, let $\varphi$ be an element of $\widehat{\Gamma}^{\times}_{\RR}(X, \overline{D})$.
Let $\overline{A}$ be an ample $\RR$-Cartier divisor on $X$.
Since $-\widehat{(\varphi)}_{\RR}$ is a nef $\RR$-Cartier divisor of $C^{\infty}$-type,
$\overline{A}-\widehat{(\varphi)}_{\RR}$ is ample, and hence $\overline{D} + \overline{A}$ is big because
$\overline{D} + \overline{A} \geq \overline{A}-\widehat{(\varphi)}_{\RR}$.
The observations in Subsection~\ref{subsec:Dirichlet:unit:theorem} show that the fundamental question
is nothing more than a generalization of Dirichlet's unit theorem.
Moreover,
the above question does not hold in the geometric case as indicated in the following remark.

\begin{Remark}
\label{rem:pseudo:alg:curve}
Let $C$ be a smooth algebraic curve over an algebraically closed field.
For $\vartheta \in \Div(C)_{\QQ}$ with $\deg(\vartheta) = 0$, the following are equivalent:
\begin{enumerate}
\renewcommand{\labelenumi}{(\arabic{enumi})}
\item
$\vartheta \in \PDiv(C)_{\QQ}$.

\item
There is $\varphi \in \Rat(C)^{\times}_{\RR}$ such that $\vartheta + (\varphi)_{\RR} \geq 0$.
\end{enumerate}
Indeed, ``(1) $\Longrightarrow$ (2)'' is obvious.
Conversely we assume (2). Then if we set $\theta = \vartheta + (\varphi)_{\RR}$, then $\theta$ is effective and $\deg(\theta) =0$, and hence
$\theta = 0$. Thus $\vartheta = (\varphi^{-1})_{\RR}$.
Therefore, by (3) in Lemma~\ref{lem:Z:module:tensor:R}, $\vartheta \in \PDiv(C)_{\QQ}$.

The above observation shows that if $\vartheta$ is a divisor on $C$ such that
$\deg(\vartheta) = 0$ and $\vartheta$ is not a torsion element in $\Pic(C)$,
then there is no $\varphi \in \Rat(C)^{\times}_{\RR}$ with $\vartheta + (\varphi)_{\RR} \geq 0$.
\end{Remark}

\subsection{Continuity of norms}
\setcounter{Theorem}{0}
Let us fix $p \in \RR_{\geq 1}$ and an $F_{\infty}$-invariant continuous volume form $\Omega$ on $X$ with
$\int_{X(\CC)} \Omega = 1$.
For $\varphi \in \Gamma^{\times}_{\RR}(X, D)$,
we define the {\em $L^p$-norm of $\varphi$ with respect to $g$} to be
\[
\Vert \varphi \Vert_{g, L^p} := \left( \int_{X(\CC)} \left( \vert \varphi \vert \exp(-g/2) \right)^p \Omega \right)^{1/p}.
\]
In this subsection, we consider the following proposition.

\begin{Proposition}
\label{prop:uniform:cont:L:p}
Let $\varphi_1, \ldots, \varphi_l \in \Rat(X)_{\RR}^{\times}$.
If we set 
\[
\Phi = \{ (x_1, \ldots, x_l) \in \RR^l \mid \varphi_1^{x_1} \cdots \varphi_l^{x_l} \in
\Gamma^{\times}_{\RR}(X, D) \},
\]
then the map $\upsilon_p : \Phi \to \RR$ given by $(x_1, \ldots, x_l) \mapsto \Vert \varphi_1^{x_1} \cdots \varphi_l^{x_l} \Vert_{g, L^p}$ is uniformly continuous on $K \cap \Phi$ for any compact set $K$ of $\RR^l$.
Moreover, the map $\upsilon_{\sup} : \Phi \to \RR$ given by $(x_1, \ldots, x_l) \mapsto \Vert \varphi_1^{x_1} \cdots \varphi_l^{x_l} \Vert_{g, \sup}$ is also uniformly continuous on $K \cap \Phi$ for any compact set $K$ of $\RR^l$.
\end{Proposition}

\begin{proof}
In order to obtain the first assertion,
we may clearly assume that $\varphi_1, \ldots, \varphi_l \in \Rat(X)^{\times}$.
Let us begin with the following claim:

\begin{Claim}
There is a constant $M$ such that
\[
\vert \varphi_1\vert^{x_1} \cdots \vert \varphi_l\vert ^{x_l} \exp(-g/2) \leq M
\]
for all $(x_1, \ldots, x_l) \in K \cap \Phi$.
\end{Claim}

\begin{proof}
Since $X(\CC)$ is compact, it is sufficient to see that the above assertion holds locally.
We set $D = a_1 D_1 + \cdots + a_n D_n$, where $a_1, \ldots, a_n \in \RR$ and
$D_1, \ldots, D_n$ are prime divisors.
Let us fix $P \in X(\CC)$ and let $f_1, \ldots, f_n$ be local equations of
$D_1, \ldots, D_n$ around $P$ respectively.
Let $g = \sum_i (-a_i) \log \vert f_i \vert^2 + h$ be the local expression of $g$ with respect to
$f_1, \ldots, f_r$, where $h$ is a continuous function around $P$.
We set $f_i = u_i t_1^{\alpha_{i1}} \cdots t_r^{\alpha_{ir}}$ and
$\phi_j = v_j t_1^{\beta_{j1}} \cdots t_r^{\beta_{jr}}$,
where $\alpha_{ik}, \beta_{jk} \in \ZZ$, $u_1, \ldots, u_n, v_1, \ldots, v_l$ are units of $\OO_{X(\CC), P}$ and $t_1, \ldots, t_r$ are prime elements of $\OO_{X(\CC), P}$.
Then
\begin{multline*}
\vert \phi_1 \vert^{x_1} \cdots \vert \phi_l \vert^{x_l} \exp(-g/2) \\
= \vert u_1 \vert^{a_1} \cdots \vert u_n\vert^{a_n} \vert v_1 \vert^{x_1} \cdots \vert v_l\vert^{x_l} \vert t_1 \vert^{\sum_i a_i \alpha_{i1} + \sum_{j} x_j \beta_{j1}}
\cdots \vert t_r \vert^{\sum_i a_i \alpha_{ir} + \sum_{j} x_j \beta_{jr}} \exp(-h/2).
\end{multline*}
Note that  $\sum_i a_i \alpha_{ik} + \sum_{j} x_j \beta_{jk}$ ($k=1, \ldots, r$) are bounded non-negative numbers
(cf. \eqref{eqn:lem:rational:tensor:real:number:1} in the proof of Lemma~\ref{lem:rational:tensor:real:number}).
Thus the claim follows.
\end{proof}

By the above claim, we obtain
\begin{multline*}
\left\vert \Vert \varphi_1^{x_1} \cdots \varphi_l^{x_l} \Vert_{g, L^p}^p - \Vert \varphi_1^{y_1} \cdots \varphi_l^{y_l} \Vert_{g, L^p}^p \right\vert \\
\leq \int_{X(\CC)} \left\vert 1 - \vert \varphi_1\vert^{p(y_1-x_1)} \cdots \vert \varphi_l\vert ^{p(y_l - x_l)} \right\vert \left( \vert \varphi_1\vert^{x_1} \cdots \vert \varphi_l\vert ^{x_l} \exp(-g/2) \right)^p \Omega \\
\leq \int_{X(\CC)} \left\vert 1 - \vert \varphi_1\vert^{p(y_1-x_1)} \cdots \vert \varphi_l\vert ^{p(y_l - x_l)} \right\vert M^p \Omega
\end{multline*}
for $(x_1, \ldots, x_l), (y_1, \ldots, y_l) \in \Phi$.
Thus the first assertion follows from the following Lemma~\ref{lem:limit:rational:functions}.

For the second assertion,
note that $\lim_{p\to\infty} \Vert \varphi_1^{x_1} \cdots \varphi_l^{x_l} \Vert_{g, L^p} =
\Vert \varphi_1^{x_1} \cdots \varphi_l^{x_l} \Vert_{g, \sup}$ for $(x_1, \ldots, x_l) \in \Phi$
(cf. \cite[the proof of Corollary~19.9]{Jost}).
Thus it follows from the first assertion.
\end{proof}

\begin{Lemma}
\label{lem:limit:rational:functions}
Let $M$ be a $d$-equidimensional complex manifold and let $\omega$ be
a continuous $(d, d)$-form on $M$ such that $\omega = \nu \Omega$,
where $\Omega$ is a volume form on $M$ and $\nu$ is a non-negative real valued
continuous function on $M$.
Let $\varphi_1, \ldots, \varphi_d$ be meromorphic functions such that
$\varphi_i$'s are non-zero on each connected component of $M$.
Then
\[
\lim_{(x_1, \ldots, x_l) \to (0, \ldots, 0)}
\int_{M} \left\vert 1 - \vert \varphi_1 \vert^{x_1} \cdots \vert \varphi_l \vert^{x_l} \right\vert 
\omega
= 0.
\]
\end{Lemma}

\begin{proof}
Clearly we may assume that $M$ is connected.
Let $\mu : M' \to M$ be a proper bimeromorphic morphism of compact complex manifolds such that
the principal divisors $(\mu^*(\varphi_1)), \ldots, (\mu^*(\varphi_l))$ are normal crossing.
Note that there are a volume form $\Omega'$ on $M'$ and a non-negative real valued
continuous function $\nu'$ on $M'$ such that $\mu^*(\omega) = \nu' \Omega'$.
Moreover,
\[
\int_{M'} \left\vert 1 - \vert \mu^*(\varphi_1) \vert^{x_1} \cdots \vert \mu^*(\varphi_l) \vert^{x_l} \right\vert 
\mu^*(\omega) = 
\int_{M} \left\vert 1 - \vert \varphi_1 \vert^{x_1} \cdots \vert \varphi_l \vert^{x_l} \right\vert 
\omega.
\]
Thus we may assume that the principal divisors $(\varphi_1), \ldots, (\varphi_l)$ are normal crossing.
Here let us consider the following claim:

\begin{Claim}
Let $\varphi_1, \ldots, \varphi_l$ be meromorphic functions on 
\[
\Delta^d = \{ (z_1, \ldots, z_d) \in \CC^d \mid
\vert z_1 \vert < 1, \ldots, \vert z_d \vert < 1 \} 
\]
such that
$\varphi_i = z_1^{c_{1i}} \cdots z_d^{c_{di}} \cdot u_i$ \rom{(}$i=1, \ldots, l$\rom{)}, where
$c_{ji} \in \ZZ$ and $u_i$'s are nowhere vanishing holomorphic functions on
$\{ (z_1, \ldots, z_d) \in \CC^d \mid
\vert z_1 \vert < 1 + \delta, \ldots, \vert z_d \vert < 1 + \delta\}$
for some $\delta \in \RR_{>0}$.
Then
\[
\lim_{(x_1, \ldots, x_l) \to (0, \ldots, 0)}
\int_{\Delta^d} \left\vert 1 - \vert \varphi_1 \vert^{x_1} \cdots \vert \varphi_l \vert^{x_l} \right\vert 
\left(\frac{\sqrt{-1}}{2}\right)^d dz_1 \wedge d\bar{z}_1 \wedge \cdots \wedge dz_d \wedge d \bar{z}_d
= 0.
\]
\end{Claim}

\begin{proof}
If we set $y_j = \sum_{i=1}^l c_{ji} x_i$, then
\[
\vert \varphi_1 \vert^{x_1} \cdots \vert \varphi_l \vert^{x_l} =
\vert z_1 \vert^{y_1} \cdots \vert z_d \vert^{y_d} \vert u_1 \vert^{x_1} \cdots \vert u_l \vert^{x_l}.
\]
Thus, if we put $z_i = r_i \exp(\sqrt{-1} \theta_i)$, then
\begin{multline*}
\int_{\Delta^d} \left\vert 1 - \vert \varphi_1 \vert^{x_1} \cdots \vert \varphi_l \vert^{x_l} \right\vert 
\left(\frac{\sqrt{-1}}{2}\right)^d dz_1 \wedge d\bar{z}_1 \wedge \cdots \wedge dz_d \wedge d \bar{z}_d \\
=
\int_{\left([0,1] \times [0, 2\pi]\right)^d} \left\vert r_1 \cdots r_d - r_1^{1+y_1} \cdots r_d^{1+y_d} \vert u_1 \vert^{x_1} \cdots \vert u_l \vert^{x_l} \right\vert
d r_1 \wedge d \theta_1 \wedge \cdots \wedge d r_d \wedge d \theta_d.
\end{multline*}
Note that
$r_1^{1+y_1} \cdots r_d^{1+y_d} \vert u_1 \vert^{x_1} \cdots \vert u_l \vert^{x_l} \to r_1 \cdots r_d$ uniformly, 
as $(x_1, \ldots, x_l) \to (0,\ldots, 0)$,
on $\left([0,1] \times [0, 2\pi]\right)^d$.
Thus the claim follows.
\end{proof}

Let us choose a covering $\{ U_j \}_{j=1}^N$ of $M$ with the following properties:
\begin{enumerate}
\renewcommand{\labelenumi}{(\alph{enumi})}
\item
For each $j$, there is a local parameter $(w_1, \ldots, w_d)$ of $U_j$ such that
$U_j$ can be identified with $\Delta^d$ in terms of $(w_1, \ldots, w_d)$.

\item
$\Supp((\phi_i)) \cap U_j \subseteq \{ w_1 \cdots w_d = 0 \}$ for all $i$ and $j$.
\end{enumerate}
Let $\{ \rho_j \}_{j=1}^N$ be a partition of unity subordinate to the covering $\{ U_j \}_{j=1}^N$.
Then
\[
\int_{M} \left\vert 1 - \vert \varphi_1 \vert^{x_1} \cdots \vert \varphi_l \vert^{x_l} \right\vert 
\omega = \sum_{j=1}^N
\int_{M} \left\vert 1 - \vert \varphi_1 \vert^{x_1} \cdots \vert \varphi_l \vert^{x_l} \right\vert 
\rho_j \omega.
\]
Note that there is a positive constant $C_j$ such that
\[
\rho_j \omega \leq C_j \left(\frac{\sqrt{-1}}{2}\right)^d dw_1 \wedge d\bar{w}_1 \wedge \cdots \wedge dw_d \wedge d \bar{w}_d.
\]
Thus the lemma follows from the above claim.
\end{proof}

\subsection{Compactness theorem}
\setcounter{Theorem}{0}
Let
$\overline{H}$ be an ample arithmetic $\RR$-Cartier divisor on $X$.
Let $\Gamma$ be a prime divisor on $X$ and let
$g_{\Gamma}$ be an  
$F_{\infty}$-invariant $\Gamma$-Green function of $C^{0}$-type such that
\[
\int_{X_{\sigma}} g_{\Gamma} c_1(\overline{H})^{d-1} = -\frac{2\adeg(\overline{H}^{d-1} \cdot (\Gamma, 0))}{[K :\QQ]}
\]
for each $\sigma \in K(\CC)$. 
We set $\overline{\Gamma} = (\Gamma, g_{\Gamma})$.
Note that 
\[
\overline{\Gamma} \in \aWDiv_{C^{0}}(X)_{\RR}\quad\text{and}\quad
\adeg(\overline{H}^{d-1} \cdot \overline{\Gamma}) = 0
\]
(see, Conventions and terminology~\ref{CV:arithmetic:divisors:3}).
Moreover, let $C^0_0(X)$ be the space of $F_{\infty}$-invariant real valued continuous functions $\eta$ on $X(\CC)$ with
$\int_{X(\CC)} \eta c_1(\overline{H})^{d-1} = 0$.

The following theorem will provide a useful tool to find an element of
$\widehat{\Gamma}^{\times}_{\RR}(X, \overline{D})$.

\begin{Theorem}
\label{thm:linsys:bounded:convex}
Let $X^{(1)}$ be the set of all prime divisors on $X$.
For an arithmetic $\RR$-Weil divisor $\overline{D}$ of $C^0$-type
\rom{(}cf. Conventions and terminology~\rom{\ref{CV:arithmetic:divisors:3}}\rom{)},
we set
\[
\Upsilon(\overline{D}) =
\left\{ (\pmb{a}, \eta) \in \RR(X^{(1)}) \oplus C^0_0(X)  \ \left| \  \overline{D} + \sum_{\Gamma} \pmb{a}_{\Gamma} \overline{\Gamma}  + (0, \eta)   \geq 0 \right\}\right.,
\]
where $\RR(X^{(1)})$ is the vector space generated by $X^{(1)}$ over $\RR$ 
\rom{(}cf. Conventions and terminology~\rom{\ref{CV:vector:space:generated:by:set}}\rom{)}.
Then $\Upsilon(\overline{D})$ has the following boundedness:
\begin{enumerate}
\renewcommand{\labelenumi}{(\arabic{enumi})}
\item
For each $\Gamma \in X^{(1)}$, $\{ \pmb{a}_{\Gamma} \}_{(\pmb{a},\eta) \in \Upsilon(\overline{D})}$ is bounded.

\item
For each $\sigma \in K(\CC)$,
\[
\left\{ \int_{X_{\sigma}} \eta c_1(\overline{H})^{d-1} \right\}_{(\pmb{a},\eta) \in \Upsilon(\overline{D})}
\]
is bounded.
\end{enumerate}
\end{Theorem}

\begin{proof}
We set $\overline{D} = \left(\sum_{\Gamma} d_{\Gamma} \Gamma, g\right)$.
Here we claim the following:

\begin{Claim}
\begin{enumerate}
\renewcommand{\labelenumi}{(\arabic{enumi})}
\item
For all $(\pmb{a},\eta) \in \Upsilon(\overline{D})$ and $\Gamma \in X^{(1)}$,
\[
- d_{\Gamma} \leq \pmb{a}_{\Gamma}  \leq  \frac{\displaystyle \frac{1}{2} \int_{X(\CC)} g c_1(\overline{H})^{\wedge d-1} + \sum_{\Gamma' \in X^{(1)} \setminus \{\Gamma\}} 
d_{\Gamma'} \adeg(\overline{H}^{d-1} \cdot (\Gamma',0))}{\adeg(\overline{H}^{d-1} \cdot (\Gamma,0))}.
\]

\item
For all $(\pmb{a},\eta) \in \Upsilon(\overline{D})$ and  $\sigma \in K(\CC)$,
\[
-\frac{2 \adeg(\overline{H}^{d-1} \cdot (D,0))}{[K : \QQ]}-\int_{X_{\sigma}} g c_1(\overline{H})^{d-1} \leq \int_{X_{\sigma}} \eta c_1(\overline{H})^{d-1}.
\]
\end{enumerate}
\end{Claim}

\begin{proof}
(1) Clearly if $(\pmb{a},\eta) \in \Upsilon(\overline{D})$, then
$-d_{\Gamma} \leq \pmb{a}_{\Gamma}$ for $\Gamma \in X^{(1)}$.
Moreover, for $\Gamma \in X^{(1)}$,
\[
0 = \adeg(\overline{H}^{d-1} \cdot \overline{\Gamma}) = \adeg(\overline{H}^{d-1} \cdot (\Gamma, 0)) + \frac{1}{2} \int_{X(\CC)} g_{\Gamma} c_1(\overline{H})^{\wedge d-1}.
\]
As $\sum_{\Gamma} \pmb{a}_{\Gamma} g_{\Gamma} + \eta + g \geq 0$, we have
\begin{multline*}
\sum_{\Gamma} \pmb{a}_{\Gamma}  \adeg(\overline{H}^{d-1} \cdot (\Gamma,0)) 
\leq \sum_{\Gamma}\pmb{a}_{\Gamma}  \adeg(\overline{H}^{d-1} \cdot (\Gamma,0)) \\
\hspace{13em}+ \frac{1}{2} \int_{X(\CC)} \left( \sum_{\Gamma} \pmb{a}_{\Gamma} g_{\Gamma} +
\eta + g \right)c_1(\overline{H})^{\wedge d-1} \\
=
\frac{1}{2}\int_{X(\CC)} g c_1(\overline{H})^{\wedge d-1},
\end{multline*}
and hence
\begin{multline*}
\pmb{a}_{\Gamma}  \adeg(\overline{H}^{d-1} \cdot (\Gamma,0)) = \sum_{\Gamma'  \in X^{(1)}} \pmb{a}_{\Gamma'}  \adeg(\overline{H}^{d-1} \cdot (\Gamma',0)) +
\sum_{\Gamma' \in X^{(1)} \setminus \{\Gamma\}} (-\pmb{a}_{\Gamma'})  \adeg(\overline{H}^{d-1} \cdot (\Gamma',0))\\
\qquad\qquad  \leq
\frac{1}{2} \int_{X(\CC)} g c_1(\overline{H})^{\wedge d-1} + \sum_{\Gamma' \not= \Gamma} d_{\Gamma'} \adeg(\overline{H}^{d-1} \cdot (\Gamma',0))\\
\end{multline*}
for all $\Gamma$.

(2)
Since $\sum_{\Gamma} \pmb{a}_{\Gamma} \Gamma + D \geq 0$, we obtain
\[
0 \leq \adeg \left(\overline{H}^{d-1} \cdot \left(\sum_{\Gamma} \pmb{a}_{\Gamma} \Gamma + D, 0\right)\right)
= \sum_{\Gamma}  \pmb{a}_{\Gamma}  \adeg(\overline{H}^{d-1} \cdot (\Gamma,0)) + \adeg(\overline{H}^{d-1} \cdot (D,0)).
\]
Therefore,
\begin{multline*}
0 \leq \int_{X_{\sigma}} \left( \sum_{\Gamma} \pmb{a}_{\Gamma} g_{\Gamma} + \eta + g\right)
c_1(\overline{H})^{d-1} \\
=  -\sum_{\Gamma} \frac{2 \ \pmb{a}_{\Gamma} \deg(\overline{H}^{d-1}(\Gamma, 0))}{[K:\QQ]} + \int_{X_{\sigma}} \eta c_1(\overline{H})^{d-1} + \int_{X_{\sigma}} g c_1(\overline{H})^{d-1} \\
\leq
\frac{2 \adeg(\overline{H}^{d-1} \cdot (D,0))}{[K : \QQ]} + \int_{X_{\sigma}} \eta c_1(\overline{H})^{d-1}
+\int_{X_{\sigma}} g c_1(\overline{H})^{d-1},
\end{multline*}
as required.
\end{proof}

By (1) in the above claim,
$\{ \pmb{a}_{\Gamma}\}_{(\pmb{a},\eta) \in \Upsilon(\overline{D})}$ is bounded for each $\Gamma$.
Further, by (2), there is a constant $M$ such that
\[
\int_{X_{\sigma}} \eta c_1(\overline{H})^{d-1} \geq M
\]
for all $(\pmb{a},\eta) \in \Upsilon(\overline{D})$ and $\sigma \in K(\CC)$, and hence
\[
M \leq \int_{X_{\sigma}} \eta c_1(\overline{H})^{d-1} = \sum_{\sigma' \in K(\CC) \setminus \{\sigma\}} 
- \int_{X_{\sigma'}} \eta c_1(\overline{H})^{d-1} \leq (\#(K(\CC)) - 1)(-M),
\]
as desired.
\end{proof}

\begin{Corollary}
\label{cor:linsys:bounded:convex}
Let $\Lambda$ be a finite set and let $\{ \overline{D}_{\lambda} \}_{\lambda \in \Lambda}$ be a family of arithmetic $\RR$-Weil divisors 
of $C^{\infty}$-type with the following properties:
\begin{enumerate}
\renewcommand{\labelenumi}{(\alph{enumi})}
\item
$\adeg(\overline{H}^{d-1} \cdot \overline{D}_{\lambda}) = 0$ for $\lambda \in \Lambda$.

\item
For each $\lambda \in \Lambda$, there is an $F_{\infty}$-invariant locally constant  function $\rho_{\lambda}$ such that
\[
c_1(\overline{D}_{\lambda}) \wedge c_1(\overline{H})^{\wedge d-2} = \rho_{\lambda} c_1(\overline{H})^{\wedge d-1}.
\]

\item
$\{ \overline{D}_{\lambda} \}_{\lambda \in \Lambda}$ is linearly independent in $\aWDiv_{C^{\infty}}(X)_{\RR}$.
\end{enumerate}
Then, for $\overline{D} \in \aWDiv_{C^{0}}(X)_{\RR}$,
the set
\[
\left\{ \pmb{a} \in \RR(\Lambda) \ \left| \ \overline{D} + \sum_{\lambda \in \Lambda} \pmb{a}_{\lambda} \overline{D}_{\lambda}  \geq 0 \right\}\right.
\]
is convex and compact.
\end{Corollary}

\begin{proof}
The convexity of the above set is obvious, so that we need to show compactness.
We pose more conditions to the $\Gamma$-Green function $g_{\Gamma}$, that is, we further assume that
$g_{\Gamma}$ is of $C^{\infty}$-type and 
$c_1(\overline{\Gamma}) \wedge c_1(\overline{H})^{\wedge d-2} = \nu_{\Gamma} c_1(\overline{H})^{\wedge d-1}$
for some locally constant function $\nu_{\Gamma}$ on $X(\CC)$.
Note that this is actually possible.
We set
\[
\Xi_X := \left\{ \xi : X(\CC) \to \RR \ \left| \ \text{$\xi$ is locally constant, $F_{\infty}$-invariant and $\sum_{\sigma \in K(\CC)} \xi_{\sigma} = 0$} \right\}\right..
\]
Then there are $\alpha_{\lambda\Gamma} \in \RR$ and $\xi_{\lambda} \in \Xi_X$ such that
\[
\overline{D}_{\lambda} = \sum_{\Gamma} \alpha_{\lambda \Gamma} \overline{\Gamma} + (0, \xi_{\lambda})
\]
for each $\lambda$.
Therefore,
\[
\sum_{\lambda} \pmb{a}_{\lambda} \overline{D}_{\lambda} =
 \sum_{\Gamma} \left( \sum_{\lambda} \pmb{a}_{\lambda} \alpha_{\lambda\Gamma} \right)\overline{\Gamma} + 
 \sum_{\lambda} \pmb{a}_{\lambda} \xi_{\lambda}.
 \]
Let us consider a linear map 
\[
T : \RR(\Lambda) \to \RR(X^{(1)}) \oplus \Xi_X
\]
given by $T(\pmb{a}) = (T_1(\pmb{a}), T_2(\pmb{a}))$, where
\[
T_1(\pmb{a})_{\Gamma} =  \sum_{\lambda} \pmb{a}_{\lambda} \alpha_{\lambda\Gamma}\ \text{($\Gamma \in X^{(1)}$)}\quad\text{and}\quad
T_2(\pmb{a})  =  \sum_{\lambda} \pmb{a}_{\lambda} \xi_{\lambda}.
\]
Then $T$ is injective. Indeed, if $T(\pmb{a}) = 0$, then
\[
\sum_{\lambda} \pmb{a}_{\lambda} \alpha_{\lambda\Gamma} = 0 \ \text{($\forall \Gamma$)}\quad\text{and}\quad
\sum_{\lambda} \pmb{a}_{\lambda} \xi_{\lambda} = 0.
\]
Thus $\sum_{\lambda}\pmb{a}_{\lambda} \overline{D}_{\lambda} = 0$, and hence $\pmb{a} = 0$.
Since $\Lambda$ is finite, we can find a finite subset $\Lambda'$ of $X^{(1)}$
such that the image of $T$ is contained in $\RR(\Lambda') \oplus \Xi_X$.
Moreover, by the previous theorem, $\Upsilon(\overline{D}) \cap (\RR(\Lambda') \oplus \Xi_X)$ is compact.
Thus
\[
\left\{ \pmb{a} \in \RR(\Lambda) \ \left| \ \overline{D} + \sum_{\lambda \in \Lambda} \pmb{a}_{\lambda} \overline{D}_{\lambda}  \geq 0\right\}\right. = T^{-1}\left(\Upsilon(\overline{D}) \cap (\RR(\Lambda') \oplus \Xi_X) \right)
\]
is also compact.
\end{proof}

\begin{Corollary}
\label{cor:existence:inf:R:rational:function}
Let $\varphi_1,\ldots, \varphi_l$ be $\RR$-rational functions on $X$ \rom{(}i.e. 
$\varphi_1, \ldots, \varphi_l \in \Rat(X)^{\times}_{\RR}$\rom{)} and
let $\overline{D} = (D, g)$ be an arithmetic $\RR$-Cartier divisor of $C^0$-type on $X$.
If 
\[
\Phi = \left\{ (a_1, \ldots, a_l) \in \RR^l \mid \varphi_1^{a_1} \cdots \varphi_l^{a_l} \in \Gamma^{\times}_{\RR}(X, D)
\right\} \not= \emptyset,
\]
then there exists $(b_1, \ldots, b_l) \in \Phi$ such that
\[
\Vert \varphi_1^{b_1} \cdots \varphi_l^{b_l} \Vert_{g, \sup} =
\inf_{(a_1, \ldots, a_l) \in \Phi} \left\{ \Vert \varphi_1^{a_1} \cdots \varphi_l^{a_l} \Vert_{g, \sup} \right\}.
\]
\end{Corollary}

\begin{proof}
Clearly we may assume that $\varphi_1, \ldots, \varphi_l$ are linearly independent in $\Rat(X)^{\times}_{\RR}$.
Replacing $g$ by $g + \lambda$ ($\lambda \in \RR$) if necessarily,
we may further assume that 
\[
\left\{ (a_1, \ldots, a_l) \in \RR^l \mid \varphi_1^{a_1} \cdots \varphi_l^{a_l} \in \widehat{\Gamma}^{\times}_{\RR}(X, \overline{D})
\right\} \not= \emptyset.
\]
We denote the above set by $\widehat{\Phi}$. As
\[
\widehat{\Phi} = \left\{ (a_1, \ldots, a_l) \in \Phi \mid 
\Vert \varphi_1^{a_1} \cdots \varphi_l^{a_l} \Vert_{g, \sup} \leq 1 \right\},
\]
we have 
\[
\inf_{(a_1, \ldots, a_l) \in \Phi} \left\{ \Vert \varphi_1^{a_1} \cdots \varphi_l^{a_l} \Vert_{g, \sup} \right\}
= \inf_{(a_1, \ldots, a_l) \in \widehat{\Phi}} \left\{ \Vert \varphi_1^{a_1} \cdots \varphi_l^{a_l} \Vert_{g, \sup} \right\}.
\]
On the other hand, $\widehat{\Phi}$ is compact by Corollary~\ref{cor:linsys:bounded:convex}.
Thus the assertion of the corollary follows from Proposition~\ref{prop:uniform:cont:L:p}.
\end{proof}

\subsection{Dirichlet's unit theorem on arithmetic curves}
\label{subsec:Dirichlet:unit:theorem}
\setcounter{Theorem}{0}
We assume $d=1$, that is, $X = \Spec(O_K)$.
In this subsection, we would like to give a proof of Dirichlet's unit theorem in flavor of Arakelov theory
(cf. \cite{Sz}).
Of course, the contents of this subsection are nothing new, but it provides the background of this paper
and a usage of the compactness theorem (cf. Corollary~\ref{cor:linsys:bounded:convex}).
The referee points out that Chambert-Loir give a similar proof based on a certain kind of compactness in
\cite[\S1.4, D]{Champ}
Let us begin with the following weak version of Dirichlet's unit theorem, which is much easier than Dirichlet's unit theorem.

\begin{Lemma}
\label{lem:weak:Dirichlet:unit:theorem}
$O_K^{\times}$ is a finitely generated abelian group.
\end{Lemma}

\begin{proof}
This is a standard fact. Indeed,
let us consider a homomorphism $L : O^{\times}_K \to \RR^{K(\CC)}$ 
given by $L(x)_{\sigma} = \log \vert \sigma(x) \vert$ for $\sigma \in K(\CC)$.
It is easy to see that, for any bounded set $B$ in $\RR^{K(\CC)}$, the set $\{ x \in O^{\times}_K \mid L(x) \in B \}$ is a finite set.
Thus the assertion of the lemma is obvious.
\end{proof}

We denote the set of all maximal ideals of $O_K$ by $M_K$.
For an $\RR$-Cartier divisor $E = \sum_{P \in M_K} e_P P$ on $X$, we define $\deg(E)$ and $\Supp(E)$ to be
\[
\deg(E) = \sum_{P \in M_K} e_P \log(\#(O_K/P))\quad\text{and}\quad
\Supp(E) := \{ P \in M_K \mid e_P \not= 0 \}.
\]

\begin{Lemma}
\label{lem:eff:deg:bound:finite}
For a constant $C$, the set $\{ E \in \Div(X) \mid \text{$E \geq 0$ and $\deg(E) \leq C$} \}$ is finite.
\end{Lemma}

\begin{proof}
This is obvious.
\end{proof}

\begin{Lemma}
\label{lem:fin:gen:K:times:Sigma}
If we set
$K^{\times}_{\Sigma} = \{ x \in K^{\times} \mid \Supp((x)) \subseteq \Sigma \}$
for a finite subset  $\Sigma$ of $M_K$, then
$K^{\times}_{\Sigma}$ is a finitely generated subgroup of $K^{\times}$.
\end{Lemma}

\begin{proof}
Let us consider a homomorphism $\alpha : K^{\times}_{\Sigma} \to \ZZ^{\Sigma}$
given by $\alpha(x)_P = \ord_P(x)$ for $P \in \Sigma$.
Then $\Ker(\alpha) = O^{\times}_K$ and the image of $\alpha$ is a finitely generated.
Thus the lemma follows from the above weak version of Dirichlet's unit theorem.
\end{proof}

\begin{Lemma}
\label{lem:non:vanishing:H:0}
We set $C_K = \log\left((2/\pi)^{r_2} \sqrt{\vert d_{K/\QQ} \vert}\right)$, where
$r_2$ is the number of complex embeddings of $K$ into $\CC$ and
$d_{K/\QQ}$ is the discriminant of $K$ over $\QQ$.
If $\adeg(\overline{D}) \geq C_K$ for $\overline{D} \in \aDiv(X)$, then
there is $x \in K^{\times}$ such that $\overline{D} + \widehat{(x)} \geq 0$.
\end{Lemma}

\begin{proof}
This is a consequence of Minkowski's theorem and
the arithmetic Riemann-Roch theorem on arithmetic curves.
\end{proof}

The following proposition is a core part of
Dirichlet's unit theorem in terms of Arakelov theory, and can be proved by using arithmetic Riemann-Roch theorem and the compactness theorem (cf. Corollary~\ref{cor:linsys:bounded:convex} and 
Corollary~\ref{cor:existence:inf:R:rational:function}).
As a corollary, it actually implies Dirichlet's unit theorem itself (cf. Corollary~\ref{cor:Dirichlet:unit:theorem}).

\begin{Proposition}
\label{prop:Dirichlet:unit:theorem:general}
Let $\overline{D} = (D, g)$ be an arithmetic $\RR$-Cartier divisor on $X$.
Then the following are equivalent:
\begin{enumerate}
\renewcommand{\labelenumi}{(\roman{enumi})}
\item
$\adeg(\overline{D}) = 0$.

\item
$\overline{D} \in \aPDiv(X)_{\RR}$.

\item
$\adeg(\overline{D}) = 0$ and $\widehat{\Gamma}^{\times}_{\RR}(X, \overline{D}) \not= \emptyset$.
\end{enumerate}
\end{Proposition}

\begin{proof}
``(iii) $\Longrightarrow$ (ii)'' : By our assumption, $\overline{D} + z \geq 0$ for some $z \in \aPDiv(X)_{\RR}$.
If we set $\overline{E} = \overline{D} + z$, then $\overline{E}$ is effective and $\adeg(\overline{E}) = \adeg(\overline{D}) + \adeg(z) = 0$.
Thus $\overline{E} = 0$, and hence $\overline{D} = -z \in \aPDiv(X)_{\RR}$.

\medskip
``(ii) $\Longrightarrow$ (i)''  is obvious.

\medskip
``(i) $\Longrightarrow$ (iii)'' :
First of all, we can find $\alpha_1, \ldots, \alpha_l \in \RR_{>0}$ and
$\overline{D}_1, \ldots, \overline{D}_l \in \aDiv(X)$ such that
$\overline{D} = \alpha_1 \overline{D}_1 + \cdots + \alpha_l \overline{D}_l$ and
$\deg(\overline{D}_i) = 0$ for all $i$.
If we can choose $\psi_i \in \widehat{\Gamma}^{\times}_{\RR}(X, \overline{D}_i)$ for all $i$,
then $\psi_1^{\alpha_1} \cdots \psi_l^{\alpha_l} \in \widehat{\Gamma}^{\times}_{\RR}(X, \overline{D})$. 
Thus
we may assume that $\overline{D} \in \aDiv(X)$ in order to show ``(i) $\Longrightarrow$ (iii)''.
For a positive integer $n$, we set 
\[
\overline{D}_n = \overline{D} + \left(0, \frac{2C_K}{n[K:\QQ]}\right).
\]
Since $\adeg(n\overline{D}_n) = C_K$, by Lemma~\ref{lem:non:vanishing:H:0},
there is $x_n \in K^{\times}$ such that
$n\overline{D}_n + \widehat{(x_n)} \geq 0$. 
In particular, $nD + (x_n) \geq 0$ and
\[
\deg(nD + (x_n)) \leq \adeg(n\overline{D}_n + \widehat{(x_n)}) = C_K.
\]
Thus, by Lemma~\ref{lem:eff:deg:bound:finite}, there is a finite subset $\Sigma'$ of $M_K$
such that 
\[
\Supp(nD + (x_n)) \subseteq \Sigma'
\]
for all $n \geq 1$.
Note that $\Supp((x_n)) \subseteq \Supp((x_n) + nD) \cup \Supp(D)$.
Therefore, we can find a finite subset $\Sigma$ of $M_K$ such that
$x_n \in K^{\times}_{\Sigma}$ for all $n \geq 1$.
By Lemma~\ref{lem:fin:gen:K:times:Sigma}, we can take a basis
$\varphi_1, \ldots, \varphi_s$ of $K^{\times}_{\Sigma} \otimes_{\ZZ} \RR$ over $\RR$.
Then, by Corollary~\ref{cor:existence:inf:R:rational:function},
if we set 
\[
\Phi = \{ (a_1, \ldots, a_s) \in \RR^s \mid \varphi_1^{a_1} \cdots \varphi_s^{a_s} \in
\Gamma^{\times}_{\RR}(X, D) \},
\]
then
there exists $(c_1, \ldots, c_s) \in \Phi$ such that
\[
\Vert \varphi_1^{c_1} \cdots \varphi_s^{c_s} \Vert_{g,\sup} =
\inf_{(a_1, \ldots, a_s) \in \Phi} \{ \Vert \varphi_1^{a_1} \cdots \varphi_s^{a_s} \Vert_{g,\sup} \},
\]
that is,
if we set $\psi = \varphi_1^{c_1} \cdots \varphi_s^{c_s}$, then
$\Vert \psi \Vert_{g,\sup} =
\inf_{\varphi \in \Gamma^{\times}_{\RR}(X,D) \cap (K^{\times}_{\Sigma} \otimes_{\ZZ} \RR)} \{ \Vert \varphi \Vert_{g,\sup} \}$.
On the other hand, as $\overline{D}_n + \widehat{(x_n^{1/n})}_{\RR} \geq 0$,
we have $x_n^{1/n} \in \Gamma^{\times}_{\RR}(X, D) \cap (K^{\times}_{\Sigma} \otimes_{\ZZ} \RR)$ and
$\Vert x_n^{1/n} \Vert_{g, \sup} \leq \exp(C_K/n[K : \QQ])$, so that
$\Vert \psi \Vert_{g,\sup} \leq \exp(C_K/n[K : \QQ])$
for all $n > 0$, and hence $\Vert \psi \Vert_{g,\sup} \leq 1$,
as required.
\end{proof}

As corollaries,
we have the following.
The second one is nothing more than of Dirichlet's unit theorem.

\begin{Corollary}
\label{cor:smallest:section:ArCurve}
Let $\overline{D} = (D,g)$ be an arithmetic $\RR$-Cartier divisor on $X$.
Then there exists $\psi \in  \Gamma^{\times}_{\RR}(X, D)$ such that
\[
\Vert \psi \Vert_{g, \sup}
= \inf \left\{ \Vert \phi \Vert_{g,\sup} \mid \phi \in \Gamma^{\times}_{\RR}(X,D) \right\}.
\]
\end{Corollary}

\begin{proof}
Clearly if the assertion holds for $\overline{D}$, then so does for $\overline{D} + (0,c)$ for all
$c \in \RR$. Thus we may assume that $\adeg(\overline{D}) = 0$.
We set $D = \sum_{P \in M_K} d_P P$.
Then, for $\phi \in \Gamma^{\times}_{\RR}(X, D)$, 
by using the product formula \eqref{eqn:rem:product:K:R:1} in Remark~\ref{rem:product:K:R},
\[
\prod_{\sigma \in K(\CC)} \vert \phi \vert_{\sigma} \exp(-g_{\sigma}/2) = 
\prod_{P \in X^{(1)}} \#(O_K/P)^{\ord_P(\phi) + d_P} \geq 1,
\]
and hence $\Vert \phi \Vert_{g, \sup} \geq 1$.
On the other hand, by Proposition~\ref{prop:Dirichlet:unit:theorem:general},
there is $\psi \in \Gamma^{\times}_{\RR}(X, D)$ with $\Vert \psi \Vert_{g,\sup} \leq 1$, as
required.
\end{proof}

\begin{Corollary}[Dirichlet's unit theorem]
\label{cor:Dirichlet:unit:theorem}
Let $\xi$ be an element of  $\RR^{K(\CC)}$ such that 
\[
\sum_{\sigma \in K(\CC)} \xi_{\sigma} = 0\quad\text{and}\quad
\xi_{\sigma} = \xi_{\bar{\sigma}} \ (\forall  \sigma \in K(\CC)).
\]
Then there are $u_1, \ldots, u_s \in O_K^{\times}$ and $a_1, \ldots, a_s \in \RR$ such that
\[
\xi_{\sigma} = a_1 \log \vert u_1 \vert_{\sigma} + \cdots +  a_s \log \vert u_s \vert_{\sigma}
\]
for all $\sigma \in K(\CC)$, that is,
$(0, \xi) + (a_1/2) \widehat{(u_1)} + \cdots + (a_s/2) \widehat{(u_s)} = 0$.
\end{Corollary}

\begin{proof}
Since $\adeg((0,\xi)) = 0$, by virtue of Proposition~\ref{prop:Dirichlet:unit:theorem:general} and
(1) in Lemma~\ref{lem:Z:module:tensor:R},
there are $a'_1, \ldots, a'_{s} \in \RR$ and $u_1, \ldots, u_{s} \in K^{\times}$ such that
$a'_1, \ldots, a'_s$ are linearly independent over $\QQ$ and
$(0,\xi) = a'_1 \widehat{(u_1)} + \cdots + a'_{s} \widehat{(u_{s})}$.
We set $(u_j) = \sum_{k=1}^l \alpha_{jk}P_k$ for each $j$, where $\alpha_{jk} \in \ZZ$ and $P_1, \ldots, P_l$ are distinct maximal ideals of $O_K$. Then
\[
0 = a'_1 (u_1) + \cdots + a'_s (u_s) = \left( \sum_{j=1}^sa'_j \alpha_{j1} \right)P_1 + \cdots + \left( \sum_{j=1}^s a'_j \alpha_{jl} \right)P_l.
\]
Thus $\sum_{j=1}^s a'_j \alpha_{jk} = 0$ for all $k$, and hence $\alpha_{jk} = 0$ for all $j,k$, which means that $u_1, \ldots, u_s \in O_K^{\times}$.
Therefore, if we set $a_j = -2a'_j$, then the corollary follows.
\end{proof}

\begin{Remark}
\label{rem:finiteness:ideal:class}
Similarly, the finiteness of $\Div(X)/\PDiv(X)$ is also a consequence of Lemma~\ref{lem:eff:deg:bound:finite} and 
Lemma~\ref{lem:non:vanishing:H:0} (cf. \cite{Sz}).
Indeed, if we set
\[
\Theta = \{ E \in \Div(X) \mid \text{$E \geq 0$ and $\deg(E) \leq C_K$} \},
\]
then
$\Theta$ is a finite set by Lemma~\ref{lem:eff:deg:bound:finite}.
Thus it is sufficient to show that, for $D \in \Div(X)$, there is $x \in K^{\times}$ such that
$D + (x) \in \Theta$.
Since
\[
\adeg \left( D, \frac{2(C_K - \deg(D))}{[K: \QQ]} \right) = C_K,
\]
by Lemma~\ref{lem:non:vanishing:H:0}, 
there is $x \in K^{\times}$ such that $\left(D, \frac{2(C_K - \deg(D))}{[K: \QQ]}\right) + \widehat{(x)} \geq 0$,
that is, 
$D + (x)  \geq 0$ and $\log \vert x \vert_{\sigma} \leq  \frac{C_K - \deg(D)}{[K: \QQ]}$ for all $\sigma \in K(\CC)$.
By using the product formula,
\[
\deg(D + (x)) = \deg(D) + \sum_{\sigma} \log \vert x \vert_{\sigma}  \leq \deg(D) + \sum_{\sigma}\frac{C_K - \deg(D)}{[K: \QQ]} = C_K.
\]
Therefore, $D + (x) \in \Theta$, as required.
\end{Remark}

\subsection{Dirichlet's unit theorem on higher dimensional arithmetic varieties}
\setcounter{Theorem}{0}
In this subsection, we will give a partial answer to the fundamental question as an application of Hodge index theorem.
First we consider the case where $d = 1$. 

\begin{Proposition}
\label{prop:arith:curve:case}
We assume $d = 1$, that is, $X = \Spec(O_K)$.
For an arithmetic $\RR$-Cartier divisor $\overline{D}$ on $X$,
the following are equivalent:
\begin{enumerate}
\renewcommand{\labelenumi}{(\roman{enumi})}
\item
$\overline{D}$ is pseudo-effective.

\item
$\deg(\overline{D}) \geq 0$.

\item
$\widehat{\Gamma}^{\times}_{\RR}(X, \overline{D}) \not= \emptyset$.
\end{enumerate}
\end{Proposition}

\begin{proof}
``(i) $\Longrightarrow$ (ii)'' :
Let $\overline{A}$ is an ample arithmetic Cartier divisor on $X$.
Then $\overline{D} + \epsilon \overline{A}$ is big for any $\epsilon > 0$, that is,
$\adeg(\overline{D} + \epsilon \overline{A}) > 0$. Therefore, $\adeg(\overline{D}) \geq 0$.

\medskip
``(ii) $\Longrightarrow$ (iii)'' :
If $\adeg(\overline{D}) > 0$, then the assertion is obvious because $\aH(X, n\overline{D}) \not= \{ 0 \}$ for $n \gg 1$,
so that we assume $\adeg(\overline{D}) = 0$.
Then $\overline{D} \in \aPDiv(X)_{\RR}$ by Proposition~\ref{prop:Dirichlet:unit:theorem:general}.

\medskip
``(iii) $\Longrightarrow$ (i)'' is obvious.
\end{proof}

To proceed with further arguments, we need the following lemma.

\begin{Lemma}
\label{existence:good:nef:div}
We assume that $X$ is regular. Let us fix an ample
$\QQ$-Cartier divisor $H$ on $X$.
Let $P_1, \ldots, P_l \in \Spec(O_K)$ and let $F_{P_1}, \ldots, F_{P_l}$ be prime divisors on $X$ such that
$F_{P_i} \subseteq \pi^{-1}(P_i)$ for all $i$.
If $A$ is an ample $\QQ$-Cartier divisor on $X$, then there is an effective $\QQ$-Cartier divisor $M$ on $X$ with the following properties:
\begin{enumerate}
\renewcommand{\labelenumi}{(\alph{enumi})}
\item
$\Supp(M) \subseteq \pi^{-1}(P_1) \cup \cdots \cup \pi^{-1}(P_l)$.

\item
$A - M$ is divisorially $\pi$-nef with respect to $H$, that is,
$\deg_H(A-M \cdot \Gamma) \geq 0$ for all vertical prime divisors $\Gamma$ on $X$
\rom{(}cf. Subsection~\rom{\ref{subsec:Hodge:index:thm}}\rom{)}.

\item
$\deg_H(A - M \cdot F) = 0$ for all closed integral integral curve $F$ on $X$
with $F \subseteq \pi^{-1}(P_1) \cup \cdots \cup \pi^{-1}(P_l)$ and
$F \not = F_{P_i}\ (\forall i)$.
\end{enumerate}
\end{Lemma}

\begin{proof}
Let us begin with the following claim:

\begin{Claim}
Let $\pi^{-1}(P_k) = a_1 F_1 + \cdots + a_n F_n$ be the irreducible decomposition as a cycle,
where $a_i \in \ZZ_{>0}$.
Renumbering $F_1, \ldots, F_n$, we may assume $F_{P_k} = F_1$. Then 
there are $x_1, \ldots, x_n \in \QQ_{>0}$ such that if we set $M_k = x_1 F_1 + \cdots + x_n F_n$,
then $\deg_H(A - M_k \cdot F_1 ) > 0$ and $\deg_H(A - M_k \cdot F_i ) = 0$ for $i=2, \ldots, n$.
\end{Claim}

\begin{proof}
By Lemma~\ref{lem:Zariski:fiber},
there are $x_1, \ldots, x_n \in \QQ$ such that
\[
\begin{pmatrix}
\deg_H(F_2 \cdot F_1) & \deg_H(F_2 \cdot F_2 ) & \cdots & \deg_H(F_2 \cdot F_n ) \\
\vdots                  &  \vdots                 & \ddots & \vdots                 \\
\deg_H(F_n \cdot F_1 ) & \deg_H(F_n \cdot F_2 ) & \cdots & \deg_H(F_n \cdot F_n )
\end{pmatrix}
\begin{pmatrix}
x_1 \\ \vdots \\ x_n
\end{pmatrix} 
=
\begin{pmatrix}
\deg_H(A \cdot F_2 ) \\ \vdots \\ \deg_H(A \cdot F_n )
\end{pmatrix}
\]
Replacing $x_i$ by $x_i + t a_i$, we may assume that $x_i > 0$ for all $i$.
We set $M_k = x_1 F_1 + \cdots + x_n F_n$.
Then $\deg_H(A - M_k \cdot F_i) = 0$ for all $i = 2, \ldots, n$.
Here we assume that $\deg_H(A - M_k \cdot F_1) \leq 0$.
Then
\[
0 < \deg_H(A \cdot F_1) \leq \deg_H(M_k \cdot F_1),
\]
and hence
\begin{align*}
\deg_H(M_k \cdot M_k) & = \sum_{i=1}^n x_i \deg_H(M_k \cdot F_i) \\
& =
x_1 \deg_H(M_k \cdot F_1) + \sum_{i=2}^n x_i \deg_H(A \cdot F_i) > 0.
\end{align*}
This contradicts to Zariski's lemma (cf. Lemma~\ref{lem:Zariski:vector:sp}).
\end{proof}

Let $M_1, \ldots, M_n$ be effective $\QQ$-Cartier divisors as the above claim.
If we set 
\[
M = M_1 + \cdots + M_l,
\]
then
$M$ is our desired $\QQ$-Cartier divisor.
\end{proof}

The following theorem is a partial answer to the fundamental question.

\begin{Theorem}
Let $\overline{D}$ be a pseudo-effective arithmetic $\RR$-Cartier divisor of $C^{0}$-type. 
If $d \geq 2$ and $D$ is numerically trivial on $X_{\QQ}$, then
$\widehat{\Gamma}^{\times}_{\RR}(X, \overline{D}) \not= \emptyset$.
\end{Theorem}

\begin{proof}
Let us begin with the following claim:

\begin{Claim}
We may assume that $X$ is regular.
\end{Claim}

\begin{proof}
By \cite[Theorem~8.2]{deJong}, there is a generically finite morphism $\mu : Y \to X$ of projective arithmetic varieties
such that $Y$ is regular.
Clearly we have the following:
\[
\begin{cases}
\text{$\overline{D}$ is pseudo-effective}\ \Longrightarrow\ \text{$\mu^*(\overline{D})$ is pseudo-effective}, \\
\text{$D$ is numerically trivial on $X_{\QQ}$}\ \Longrightarrow\ \text{$\mu^*(D)$ is numerically trivial on $Y_{\QQ}$}. \\
\end{cases}
\]
Let $\aDiv_{\Cur}(X)_{\RR}$ be  the vector space over $\RR$ consisting of pairs
$(D, T)$, where $D$ is an $\RR$-Cartier divisor $D$ and $T$ is an $F_{\infty}$-invariant $(1,1)$-current of real type.
We can assign an ordering $\geq$ to $\aDiv_{\Cur}(X)_{\RR}$ in  following way:
\[
(D_1, T_1) \geq (D_2, T_2)\quad\Longleftrightarrow\quad
\text{$D_1 \geq D_1$ and $T_1 \geq T_2$.}
\]
In the same way, we can define $\aDiv_{\Cur}(Y)_{\RR}$ and the ordering on $\aDiv_{\Cur}(Y)_{\RR}$.
Let 
\[
\mu_* : \aDiv_{\Cur}(Y)_{\RR} \to \aDiv_{\Cur}(X)_{\RR}
\]
be a homomorphism given by
$\mu_*(D, T) = (\mu_*(D), \mu_*(T))$.
Let 
\[
N : \Rat(Y)^{\times} \to \Rat(X)^{\times}
\]
be the norm map.
Then it is easy to see the following:
\[
\begin{cases}
\text{$\mu_*(\widehat{\psi}) = \widehat{(N(\psi))}$ for $\psi \in \Rat(Y)^{\times}$}, \\
\text{$\mu_*(\mu^*(\overline{D})) = \deg(Y \to X)\overline{D}$ for $\overline{D} \in \aDiv_{C^0}(X)_{\RR}$},\\
(D_1, T_1) \geq (D_2, T_2) \Longrightarrow\ \mu_*(D_1, T_1) \geq \mu_*(D_2, T_2). \\
\end{cases}
\]
The first equation yields a homomorphism
\[
\mu_* : \aPDiv(Y)_{\RR} \to \aPDiv(X)_{\RR}.
\]
Thus the claim follows from the above formulae.
\end{proof}

First of all,
by Theorem~\ref{thm:pseudo:principal}, $D_{\QQ} \in \PDiv(X_{\QQ})_{\RR}$.
Thus there are $z \in \aPDiv(X)_{\RR}$, a vertical $\RR$-Cartier divisor $E$ and
an $F_{\infty}$-invariant continuous function $\eta$ on $X(\CC)$ such that $\overline{D} = z + (E, \eta)$.

\begin{Claim}
We may assume the following:
\begin{enumerate}
\renewcommand{\labelenumi}{(\alph{enumi})}
\item
$E$ is effective.

\item
There are $P_1, \ldots, P_l \in \Spec(O_K)$ such that
$\Supp(E) \subseteq \pi^{-1}(P_1) \cup \cdots \cup \pi^{-1}(P_l)$.

\item
For each $i=1, \ldots, l$, there is a closed integral curve $F_{P_i}$ on $X$ such that
$F_{P_i} \subseteq \pi^{-1}(P_i)$ and $F_{P_i} \not\subseteq \Supp(E)$.
\end{enumerate}
\end{Claim}

\begin{proof}
Clearly we can choose $P_1, \ldots, P_l \in \Spec(O_K)$ and
$\beta_1, \ldots, \beta_l \in \RR$ such that if we set $E' = E + \beta_1\pi^{-1}(P_1) + \cdots + \beta_l \pi^{-1}(P_l)$,
then $E'$ satisfy the above (a), (b) and (c).
Moreover, since the class group of $O_K$ is finite (cf. Remark~\ref{rem:finiteness:ideal:class}),
there are $n_i \in \ZZ_{>0}$ and $f_i \in O_K$ such that $n_i P_i = f_i O_K$. Thus
$\beta_1\pi^{-1}(P_1) + \cdots + \beta_l \pi^{-1}(P_l) \in \PDiv(X)_{\RR}$, and hence the claim follows.
\end{proof}

Note that $(E, \eta)$ is pseudo-effective by Lemma~\ref{lem:pseudo:plus:principal}.
By Lemma~\ref{lem:pseudo:effective:cont:func},
there is a locally constant function $\lambda$ on $X(\CC)$ such that
$(E, \eta) \geq (E, \lambda)$ and $(E, \lambda)$ is pseudo-effective.
Let us fix an ample arithmetic Cartier divisor $\overline{H} = (H, h)$ on $X$.
Then, by Lemma~\ref{existence:good:nef:div}, there is an effective vertical $\QQ$-Cartier divisor $M$ such that
\[
\deg_H(H - M \cdot E) = 0\quad\text{and}\quad
\deg_{H}(H - M \cdot \Gamma) \geq 0
\]
for all vertical prime divisors $\Gamma$. 

\begin{Claim}
There is a constant $c$ such that if we set $h' = h+c$, then
\[
\adeg\left((H - M, h') \cdot \overline{H}^{d-2} \cdot (\Gamma, 0)\right) \geq 0
\]
for all horizontal prime divisors $\Gamma$ on $X$.
\end{Claim}

\begin{proof}
Note that $\adeg((H, h) \cdot \overline{H}^{d-2} \cdot (\Gamma, 0)) \geq 0$.
Thus it is sufficient to find a constant $c$ such that
\[
\adeg((M, -c) \cdot \overline{H}^{d-2} \cdot (\Gamma, 0)) \leq 0
\]
for all horizontal prime divisors $\Gamma$ on $X$.
We choose $Q_1, \ldots, Q_m \in \Spec(O_K)$ and $\alpha_1, \ldots, \alpha_m \in \RR_{>0}$ such that
$M \leq \sum_{i=1}^m \alpha_i \pi^{-1}(Q_i)$. We also choose a constant $c$ such that
\[
c [K : \QQ] \geq \sum_{i=1}^m \alpha_i \log \#(O_K/Q_i).
\]
Then
\begin{multline*}
\adeg((M, -c) \cdot \overline{H}^{d-2} \cdot (\Gamma, 0)) \\
\leq 
\adeg\left(\left(\sum_{i=1}^m \alpha_i \pi^{-1}(Q_i), -c\right) \cdot \overline{H}^{d-2} \cdot (\Gamma, 0)\right) \\
\leq \sum_{i=1}^m \alpha_i \frac{\deg(H_{\QQ}^{d-2}\cdot \Gamma_{\QQ})}{[K : \QQ]} \log \#(O_K/Q_i) - c\deg(H_{\QQ}^{d-2} \cdot \Gamma_{\QQ})
\leq 0.
\end{multline*}
\end{proof}

Let $\overline{L} = (L, k)$ be an effective $\RR$-Cartier divisor of $C^0$-type.
Then, since 
\[
\adeg\left((H - M, h') \cdot \overline{H}^{d-2} \cdot (L, 0)\right) \geq 0
\]
by the above claim,
we have
\begin{multline*}
\adeg\left((H - M, h') \cdot \overline{H}^{d-2} \cdot (L, k)\right) \\
\geq
\adeg\left((H - M, h') \cdot \overline{H}^{d-2} \cdot (0, k)\right) 
=
\frac{1}{2} \int_{X(\CC)} k c_1(\overline{H})^{d-1} \geq 0.
\end{multline*}
In particular, 
\[
\adeg((H-M, h') \cdot \overline{H}^{d-2} \cdot (E, \lambda)) \geq 0
\]
because $(E, \lambda)$ is pseudo-effective.
Note that
\[
\adeg((H-M, h') \cdot \overline{H}^{d-2} \cdot (E, \lambda)) = \frac{1}{2}\left( \sum_{\sigma \in K(\CC)} \lambda_{\sigma} \right)
\int_{X(\CC)} c_1(\overline{H})^{d-1}.
\]
Therefore,
$\sum_{\sigma \in K(\CC)} \lambda_{\sigma} \geq 0$, and hence,
by Proposition~\ref{prop:arith:curve:case},
there are $u_1, \ldots, u_s \in K^{\times}$ and $\gamma_1, \ldots, \gamma_s \in \RR$
such that $\gamma_1 \widehat{(u_1)} + \cdots + \gamma_s \widehat{(u_s)} \leq (0, \lambda)$.
Thus
\[
\overline{D} = z + (E, \eta) \geq z + (0, \lambda) \geq z + \gamma_1 \widehat{(u_1)} + \cdots + \gamma_s \widehat{(u_s)}.
\]
\end{proof}

\begin{Corollary}
If $d=2$, $\overline{D}$ is pseudo-effective and $\deg(D_{\QQ}) = 0$, then
the Zariski decomposition of $\overline{D}$ exists.
\end{Corollary}

\subsection{Multiplicative generators of approximately smallest sections}
\setcounter{Theorem}{0}
In this subsection, 
we define a notion of multiplicative generators of approximately smallest sections and observe its properties.
It is a sufficient condition to guarantee the fundamental question (cf. Corollary~\ref{cor:ans:fund:question}).
Let $\overline{D}$ be an arithmetic $\RR$-Cartier divisor of $C^0$-type on $X$.
Let us begin with its definition.

\begin{Definition}
\label{def:gen:app:small:sec}
We assume that $\Gamma^{\times}_{\QQ}(X, D) \not= \emptyset$.
Let $\varphi_1, \ldots, \varphi_l$ be 
$\RR$-rational functions  on $X$ (i.e. $\varphi_1, \ldots, \varphi_l \in \Rat(X)^{\times}_{\RR}$).
We say $\varphi_1, \ldots, \varphi_l$ 
are  {\em multiplicative generators of approximately smallest sections for $\overline{D}$}
if, for  a given $\epsilon > 0$, there is $n_0 \in \ZZ_{>0}$ such that, for any integer $n$ with
$n \geq n_0$ and $\Gamma^{\times}(X, nD) \not= \emptyset$, 
we can find $a_1, \ldots, a_l \in \RR$ satisfying 
$\varphi_1^{a_1} \cdots \varphi_l^{a_l} \in \Gamma_{\RR}^{\times}(X, nD)$ and
\[
\Vert \varphi_1^{a_1} \cdots \varphi_l^{a_l} \Vert_{ng,\sup} \leq e^{\epsilon n}
\min \left\{ \Vert \phi \Vert_{ng,\sup} \mid \phi \in \Gamma^{\times}(X, nD) \right\}.
\]
\end{Definition}

First let us see the following proposition.

\begin{Proposition}
\label{prop:multiplicative:generators:def:equiv}
We assume that $\Gamma^{\times}_{\QQ}(X, D) \not= \emptyset$.
Let $\varphi_1, \ldots, \varphi_l$ be 
$\RR$-rational functions  on $X$.
Then the following are equivalent:
\begin{enumerate}
\renewcommand{\labelenumi}{(\arabic{enumi})}
\item
$\varphi_1, \ldots, \varphi_l$ 
are multiplicative generators of approximately smallest sections for $\overline{D}$.

\item
There are
$x_1, \ldots, x_l \in \RR$ such that
$\varphi_1^{x_1} \cdots \varphi_l^{x_l} \in \Gamma_{\RR}^{\times}(X, D)$ and
\[
\Vert \varphi_1^{x_1} \cdots \varphi_l^{x_l} \Vert_{g,\sup} \leq
\inf \left\{ \Vert f \Vert_{g,\sup} \mid f \in \Gamma^{\times}_{\QQ}(X, D) \right\}.
\]
\end{enumerate}
Note that if we set $\psi = \varphi_1^{x_1} \cdots \varphi_l^{x_l}$ in \rom{(2)}, then
$\psi$ forms a multiplicative generator of approximately smallest sections for $\overline{D}$.
\end{Proposition}

\begin{proof}
It is obvious that (2) implies (1), so that we assume (1).
Let $m$ be a positive integer with $\Gamma^{\times}(X, mD) \not= \emptyset$.
Here, let us check the following claim:

\begin{Claim}
\label{claim:prop:multiplicative:generators:def:equiv:1}
$\lim_{n\to\infty}
\left( \min \{ \Vert h \Vert_{nmg,\sup} \mid h \in \Gamma^{\times}(X, nmD) \} \right)^{1/nm}$ exists and
\[
 \lim_{n\to\infty}
\left( \min \{ \Vert h \Vert_{nmg,\sup} \mid h \in \Gamma^{\times}(X, nmD) \} \right)^{1/nm} 
= \inf \left\{ \Vert f \Vert_{g,\sup} \mid f \in \Gamma^{\times}_{\QQ}(X,D) \right\}.
\]
\end{Claim}

\begin{proof}
If we set 
\[
a_n =  \min \{ \Vert h \Vert_{nmg,\sup} \mid h \in \Gamma^{\times}(X, nmD) \},
\] 
then
$a_{n+n'} \leq a_n a_{n'}$ for all $n, n' > 0$.
Thus it is easy to see that $\lim_{n\to\infty} a_n^{1/n} = \inf_{n > 0} \{ a_n^{1/n} \}$, which means
\begin{multline*}
 \lim_{n\to\infty}
\left( \min \{ \Vert h \Vert_{nmg,\sup} \mid h \in \Gamma^{\times}(X, nmD) \} \right)^{1/nm} \\
=
\inf_{n > 0} \left\{ \min \{ \Vert h^{1/nm} \Vert_{g,\sup} \mid h \in \Gamma^{\times}(X, nmD) \} \right\}.
\end{multline*}
On the other hand, by (3)  in Lemma~\ref{lem:rational:tensor:real:number},
\[
\Gamma^{\times}_{\QQ}(X, D) =  \Gamma^{\times}_{\QQ}(X, mD)^{1/m} =  \bigcup_{n > 0}  \Gamma^{\times}(X, nmD)^{1/nm},
\]
and hence the claim follows.
\end{proof}

By Corollary~\ref{cor:existence:inf:R:rational:function},
there exist $x_1, \ldots, x_l \in \RR$ such that if we set
\[
\Phi = \left\{ (a_1, \ldots, a_l) \in \RR^l \mid \varphi_1^{a_1} \cdots \varphi_l^{a_l} \in \Gamma^{\times}_{\RR}(X, D)
\right\},
\]
then $(x_1, \ldots, x_l) \in \Phi$ and
\[
\Vert \varphi_1^{x_1} \cdots \varphi_l^{x_l} \Vert_{g, \sup} =
\inf_{(a_1, \ldots, a_l) \in \Phi} \left\{ \Vert \varphi_1^{a_1} \cdots \varphi_l^{a_l} \Vert_{g, \sup} \right\}.
\]
On the other hand,
by definition,
for  a given $\epsilon > 0$, there is $n_0 \in \ZZ_{>0}$ such that, for any integer $ n \geq n_0$, 
we can find $c_1, \ldots, c_l \in \RR$ satisfying 
$\varphi_1^{c_1} \cdots \varphi_l^{c_l} \in \Gamma_{\RR}^{\times}(X, nmD)$ and
\[
\Vert \varphi_1^{c_1} \cdots \varphi_l^{c_l} \Vert_{nmg,\sup} \leq e^{\epsilon nm}
\min \{ \Vert h \Vert_{nmg,\sup} \mid h \in \Gamma^{\times}(X, nmD) \}.
\]
Thus, as $(c_1/nm, \ldots, c_l/nm) \in \Phi$,
\begin{align*}
\Vert \varphi_1^{x_1} \cdots \varphi_l^{x_l} \Vert_{g, \sup}  & \leq
\Vert \varphi_1^{c_1/nm} \cdots \varphi_l^{c_l/nm} \Vert_{g,\sup} \\
& \leq e^{\epsilon}
\left( \min \{ \Vert h \Vert_{nmg,\sup} \mid h \in \Gamma^{\times}(X, nmD) \}\right)^{1/nm}
\end{align*}
for $n \geq n_0$. Therefore, by Claim~\ref{claim:prop:multiplicative:generators:def:equiv:1}, 
\begin{align*}
\Vert \varphi_1^{x_1} \cdots \varphi_l^{x_l} \Vert_{g, \sup}  & \leq
 e^{\epsilon} \lim_{n\to\infty}
\left( \min \{ \Vert h \Vert_{nmg,\sup} \mid h \in \Gamma^{\times}(X, nmD) \}\right)^{1/nm} \\
& =
 e^{\epsilon}\inf \left\{ \Vert f \Vert_{g,\sup} \mid f \in \Gamma^{\times}_{\QQ}(X,D) \right\}.
\end{align*}
Thus (2) follows because $\epsilon$ is arbitrary.
\end{proof}

By Corollary~\ref{cor:smallest:section:ArCurve}, if $d=1$, then
we can find $\psi \in \Gamma^{\times}_{\RR}(X,D)$ such that
\[
\Vert \psi \Vert_{g, \sup}
= \inf \left\{ \Vert \phi \Vert_{g,\sup} \mid \phi \in \Gamma^{\times}_{\RR}(X,D) \right\}.
\]
Note that the above $\psi$ yields a multiplicative generator of approximately smallest sections.
The same assertion holds if we assume the existence of
multiplicative generators of approximately smallest sections.

\begin{Theorem}
\label{thm:inf:R}
We assume that $\Gamma^{\times}_{\QQ}(X, D) \not= \emptyset$.
If $\overline{D}$ has multiplicative generators of approximately smallest sections,
then there exists $\psi \in  \Gamma^{\times}_{\RR}(X, D)$ such that
\[
\Vert \psi \Vert_{g, \sup}
= \inf \left\{ \Vert \phi \Vert_{g,\sup} \mid \phi \in \Gamma^{\times}_{\RR}(X,D) \right\}.
\]
\end{Theorem}

\begin{proof}
By Proposition~\ref{prop:multiplicative:generators:def:equiv}, it is sufficient to see the following inequality:
\addtocounter{Claim}{1}
\begin{equation}
\label{eqn:thm:inf:R:1}
 \inf \left\{ \Vert f \Vert_{g,\sup} \mid f \in \Gamma^{\times}_{\QQ}(X,D) \right\} \leq \inf \left\{ \Vert \phi \Vert_{g,\sup} \mid \phi \in \Gamma^{\times}_{\RR}(X,D) \right\}.
\end{equation}

Let $\eta \in \Gamma^{\times}_{\QQ}(X, D)$, $D' = D + (\eta)$ and $g' = g - \log \vert \eta \vert^2$.
Then 
\[
\begin{cases}
\Gamma^{\times}_{\QQ}(X, D') = \left\{ f/\eta \mid f \in \Gamma^{\times}_{\QQ}(X, D)\right\},\\
\Gamma^{\times}_{\RR}(X, D') = \left\{ \phi/\eta \mid \phi \in \Gamma^{\times}_{\RR}(X, D)\right\},\\
\text{$\Vert \phi/\eta \Vert_{g', \sup} = \Vert \phi \Vert_{g, \sup}$ for $\phi \in \Gamma^{\times}_{\RR}(X, D)$},
\end{cases}
\]
and hence
\[
\begin{cases}
\inf \left\{ \Vert f' \Vert_{g',\sup} \mid f' \in \Gamma^{\times}_{\QQ}(X, D') \right\}  = \inf \left\{ \Vert f \Vert_{g,\sup} \mid f \in \Gamma^{\times}_{\QQ}(X, D) \right\},\\
\inf \left\{ \Vert \phi' \Vert_{g',\sup} \mid \phi' \in \Gamma^{\times}_{\RR}(X, D') \right\}  = \inf \left\{ \Vert \phi \Vert_{g,\sup} \mid \phi \in \Gamma^{\times}_{\RR}(X, D)\right\}.
\end{cases}
\]
Therefore, in order to see \eqref{eqn:thm:inf:R:1}, we may assume that $D$ is effective,
that is, if we set $D = \sum d_{\Gamma} \Gamma$, then
$d_{\Gamma} \geq 0$ for all $\Gamma$.

Let $\phi$ be an arbitrary element of $\Gamma^{\times}_{\RR}(X,D)$. Then we can find $f_1, \ldots, f_r \in \Rat(X)^{\times}_{\QQ}$ and
$a_1, \ldots, a_r \in \RR$ such that $\phi = f_1^{a_1} \cdots f_r^{a_r}$ and
$a_1, \ldots, a_r$ are linearly independent over $\QQ$.
Let $S$ be the set of codimension one points of $\bigcup_{i=1}^r \Supp((f_i))$.

\begin{Claim}
\label{claim:thm:inf:R:1}
If $\epsilon$ is a positive number, then $\ord_{\Gamma}(\phi^{1/(1 + \epsilon)}) + d_{\Gamma} > 0$
for all $\Gamma \in S$.
\end{Claim}

\begin{proof}
It is sufficient to show that $\ord_{\Gamma}(\phi) + (1 + \epsilon) d_{\Gamma} > 0$
for all $\Gamma \in S$.
First of all, note that $\ord_{\Gamma}(\phi) + d_{\Gamma} \geq 0$. 
If either $\ord_{\Gamma}(\phi) > 0$ or $d_{\Gamma} > 0$, then the assertion is obvious, so that we assume 
$\ord_{\Gamma}(\phi) \leq 0$ and $d_{\Gamma} = 0$.
Then 
\[
\ord_{\Gamma}(\phi) = a_1 \ord_{\Gamma}(f_1) + \cdots + a_r \ord_{\Gamma}(f_r) = 0,
\]
which yields $\ord_{\Gamma}(f_1) = \cdots = \ord_{\Gamma}(f_r) = 0$.
This is a contradiction because $\Gamma \in S$.
\end{proof}

As $\phi^{1/(1 + \epsilon)} = f_1^{a_1/(1+\epsilon)} \cdots f_1^{a_r/(1+\epsilon)}$, by Claim~\ref{claim:thm:inf:R:1}, we can find $\delta > 0$ such that $f_1^{x_1} \cdots f_r^{x_r} \in \Gamma^{\times}_{\RR}(X, D)$
for all $(x_1, \ldots, x_r) \in \RR^r$ with 
\[
\vert x_1 - a_1/(1 + \epsilon) \vert + \cdots + \vert x_r - a_r/(1 + \epsilon) \vert \leq \delta.
\]
We choose a sequence $\left\{ \pmb{t}_n = (t_{n1}, \ldots, t_{nr}) \right\}_{n=1}^{\infty}$ of $\QQ^r$ such that
\[
\vert t_{n1} - a_1/(1 + \epsilon) \vert + \cdots + \vert t_{nr} - a_r/(1 + \epsilon) \vert \leq \delta
\]
and $\lim_{n\to\infty} \pmb{t}_n = (a_1/(1+\epsilon), \ldots, a_r/(1+\epsilon))$.
Then
\[
\inf \left\{ \Vert f \Vert_{g,\sup} \mid f \in \Gamma^{\times}_{\QQ}(X,D) \right\} \leq 
\Vert f_1^{t_{n1}} \cdots f_r^{t_{nr}} \Vert_{g, \sup}
\]
because $f_1^{t_{n1}} \cdots f_r^{t_{nr}} \in \Gamma^{\times}_{\QQ}(X, D)$.
Thus, by using Proposition~\ref{prop:uniform:cont:L:p},
we obtain
\[
\inf \left\{ \Vert f \Vert_{g,\sup} \mid f \in \Gamma^{\times}_{\QQ}(X,D) \right\} \leq \Vert \phi^{1/(1+\epsilon)} \Vert_{g, \sup},
\]
which implies $\inf \left\{ \Vert f \Vert_{g,\sup} \mid f \in \Gamma^{\times}_{\QQ}(X,D) \right\} \leq \Vert \phi \Vert_{g, \sup}$ 
by Proposition~\ref{prop:uniform:cont:L:p} again. Therefore, we have \eqref{eqn:thm:inf:R:1}.
\end{proof}

As a corollary, we have the following:

\begin{Corollary}
\label{cor:ans:fund:question}
We assume the following:
\begin{enumerate}
\renewcommand{\labelenumi}{(\arabic{enumi})}
\item
$\widehat{\Gamma}^{\times}_{\QQ}(X, \overline{D} + (0, \epsilon)) \not= \emptyset$ for any $\epsilon > 0$.

\item
$\overline{D}$ has multiplicative generators of approximately smallest sections.
\end{enumerate}
Then $\widehat{\Gamma}^{\times}_{\RR}(X, \overline{D}) \not= \emptyset$.
\end{Corollary}

\begin{proof}
By the above theorem,
there
exists $\psi \in  \Gamma^{\times}_{\RR}(X, D)$ such that
\[
\Vert \psi \Vert_{g, \sup}
= \inf \left\{ \Vert \phi \Vert_{g,\sup} \mid \phi \in \Gamma^{\times}_{\RR}(X,D) \right\}.
\]
Since $\widehat{\Gamma}^{\times}_{\QQ}(X, \overline{D} + (0, \epsilon)) \not= \emptyset$,
we can find $\phi \in \Gamma^{\times}_{\QQ}(X, D)$ with $\Vert \phi \Vert_{g,\sup} \leq e^{\epsilon/2}$, and hence
$\Vert \psi \Vert_{g, \sup} \leq e^{\epsilon/2}$. Therefore, $\Vert \psi \Vert_{g, \sup} \leq 1$, as required.
\end{proof}

\begin{Remark}
(1) We assume that $D \in \Div(X)_{\QQ}$. Then
$\Gamma^{\times}_{\QQ}(X,D)$ is dense in $\Gamma^{\times}_{\RR}(X,D)$, that is,
for $f_1^{a_1} \cdots f_r^{a_r} \in \Gamma^{\times}_{\RR}(X,D)$ with $a_1, \ldots, a_r \in \RR$ and
$f_1, \ldots, f_r \in \Rat(X)^{\times}_{\QQ}$, there is a sequence $\{ (a_{1n}, \ldots, a_{rn}) \}_{n=1}^{\infty}$ in $\QQ^r$
such that $f_1^{a_{1n}} \cdots f_r^{a_{rn}} \in \Gamma^{\times}_{\QQ}(X,D)$ and
$\lim_{n\to\infty} (a_{1n}, \ldots, a_{rn}) = (a_{1}, \ldots, a_{r})$.
In particular, $\Gamma^{\times}_{\QQ}(X, D) \not= \emptyset$ if and only if $\Gamma^{\times}_{\RR}(X, D) \not= \emptyset$.
This fact can be checked as follows.
Clearly we may assume that 
$a_1, \ldots, a_r$ are linearly independent over $\QQ$.
Let $S$ be the set of codimension one points of $\bigcup_{i} \Supp((f_i))$ and
$D = \sum_{\Gamma} d_{\Gamma} \Gamma$ ($d_{\Gamma} \in \QQ$).
If $(\QQ a_1 + \cdots + \QQ a_r) \cap \QQ = \{0 \}$, then it is easy to see that $\ord_{\Gamma}(f_1^{a_1} \cdots f_r^{a_r}) + d_{\Gamma} > 0$ for
all $\Gamma \in S$. Thus the assertion follows.
If $(\QQ a_1 + \cdots + \QQ a_r) \cap \QQ = \QQ$, then we may assume that $a_1 \in \QQ$ and
$(\QQ a_2 + \cdots + \QQ a_r) \cap \QQ = \{0 \}$.
Thus, as before, 
we can find a sequence $\{ (a_{2n}, \ldots, a_{rn}) \}_{n=1}^{\infty}$ in $\QQ^{r-1}$
such that $f_2^{a_{2n}} \cdots f_r^{a_{rn}} \in \Gamma^{\times}_{\QQ}(X,(f_1^{a_1}) + D)$ and
$\lim_{n\to\infty} (a_{2n}, \ldots, a_{rn}) = (a_{2}, \ldots, a_{r})$, as required.

(2) The assertion of (1) does not hold in general.
For example, let $\PP^1_{\ZZ} = \Proj(\ZZ[T_0,T_1])$ and $a \in \RR_{>0} \setminus \QQ$.
Then $\Gamma_{\RR}^{\times}(X, a(T_1/T_0)) \not= \emptyset$ and
$\Gamma_{\QQ}^{\times}(X, a(T_1/T_0)) = \emptyset$.
Indeed, $z^a \in \Gamma_{\RR}^{\times}(X, a(T_1/T_0))$, where $z = T_0/T_1$.
Moreover, if $\Gamma_{\QQ}^{\times}(X, a(T_1/T_0)) 
\not= \emptyset$, then there are $n \in \ZZ_{>0}$ and $f \in \QQ(z)$ such that
$(f) \geq na(z)$.
In particular, $f \in \QQ[z]$, so that we can set $f(z) = \sum_{i=s}^t c_i z^i$, where $0 \leq s \leq t$, $c_s \not= 0$ and
$c_t \not= 0$. Note that $\ord_0(f) = s$ and $\ord_{\infty}(f) = -t$. Thus $na \leq s \leq t \leq na$, and hence $na = s = t$.
This is a contraction because $a \in \RR_{>0} \setminus \QQ$.
\end{Remark}

Finally let us consider a sufficient condition for multiplicative generators of approximately smallest sections.
Let us fix an $F_{\infty}$-invariant continuous volume form $\Omega$ on $X$ with
$\int_{X(\CC)} \Omega = 1$.
We assume that $\Gamma^{\times}_{\QQ}(X, D) \not= \emptyset$.
The natural inner product $\langle\ ,\ \rangle_{n\overline{D}}$ on $H^{0}(X, nD) \otimes {\RR}$
is given by
\[
\langle \varphi, \psi \rangle_{n\overline{D}} := \int_{X(\CC)} \varphi \bar{\psi} \exp(-ng) \Omega
\qquad(\varphi, \psi \in H^{0}(X, nD)).
\]
For $\varphi_{1}, \ldots, \varphi_{l} \in H^0(X, D)$ and $A=(a_{1}, \ldots, a_{l}) \in \ZZ^{l}_{\geq 0}$,
$\varphi_{1}^{a_{1}} \cdots \varphi_{l}^{a_{l}}$ is denoted by
$\pmb{\varphi}^{A}$ for simplicity.
Note that $\pmb{\varphi}^{A} \in H^0(X, \vert A \vert D)$, where
$\vert A \vert = a_1 + \cdots + a_l$.

\begin{Definition}
We say $\varphi_{1}, \ldots, \varphi_{l} \in H^{0}(X, D) \setminus \{ 0 \}$ is a {\em well-posed generators for
$\overline{D}$} 
if, for $n \gg 1$,
there is a subset $\Sigma_{n}$ of $\{ A = (a_{1}, \ldots, a_{l}) \in \ZZ_{\geq 0}^{l} \mid a_{1} + \cdots + a_{l} = n \}$ with the following properties:
\begin{enumerate}
\renewcommand{\labelenumi}{(\arabic{enumi})}
\item
$\{ \pmb{\varphi}^{A} \mid A \in \Sigma_{n} \}$ forms a basis of $H^{0}(X, nD) \otimes {\QQ}$ over $\QQ$.

\item
Let $\langle \pmb{\varphi}^A \mid A \in \Sigma_{n} \rangle_{\ZZ}$ be the $\ZZ$-submodule generated by
$\{ \pmb{\varphi}^A \mid A \in \Sigma_{n} \}$ in $H^0(X, nD)$, that is,
$\langle \pmb{\varphi}^A \mid A \in \Sigma_{n} \rangle_{\ZZ} = \sum\nolimits_{A \in \Sigma_{n}} \ZZ \ \pmb{\varphi}^{A}$.
Then
\[ 
\limsup_{n\to\infty} \left(\#\left( H^{0}(X, nD)/\langle \pmb{\varphi}^A \mid A \in \Sigma_{n} \rangle_{\ZZ} \right)\right)^{1/n} = 1.
\]

\item
For a finite subset $\{ \psi_1, \ldots, \psi_r \}$ of $H^0(X, nD)_{\RR}$,
the square root of the Gramian of $\psi_1, \ldots, \psi_r$ with respect to $\langle\ ,\ \rangle_{n\overline{D}}$
is denoted by $\vol(\{ \psi_1, \ldots, \psi_r \})$ (for details, see Conventions and terminology~\ref{CV:vol:parallelotope}).
Then
\[
\liminf_{n\to\infty} \min \left\{ \left. \left(\frac{\vol(\{ \pmb{\varphi}^{B} \mid B \in \Sigma_n\})}
{\sqrt{\langle \pmb{\varphi}^{A}, \pmb{\varphi}^{A}\rangle_{n\overline{D}}}\vol(\{ \pmb{\varphi}^{B} \mid B \in \Sigma_n \setminus \{ A \}\})}\right)^{1/n}
\ \right| \ A \in \Sigma_n \right\}
 = 1.
\]
\end{enumerate}
\end{Definition}

\begin{Proposition}
If $\varphi_{1}, \ldots, \varphi_{l} \in H^{0}(X, D) \setminus \{ 0 \}$ are well-posed generators for
$\overline{D}$, then $\varphi_{1}, \ldots, \varphi_{l}$ are multiplicative generators of approximately smallest sections for $\overline{D}$
\end{Proposition}

\begin{proof}
For a given $\epsilon > 0$, we set $\epsilon' = \epsilon/6$. First of all,
there is a positive integer $n_0$ such that
\[
r_n = \#\left( H^{0}(X, nD)/\langle \pmb{\varphi}^A \mid A \in \Sigma_{n} \rangle_{\ZZ} \right) \leq e^{n\epsilon'}
\]
and
\[
\frac{\vol(\{ \pmb{\varphi}^{B} \mid B \in \Sigma_n\})}
{\sqrt{\langle \pmb{\varphi}^{A}, \pmb{\varphi}^{A}\rangle}\vol(\{ \pmb{\varphi}^{B} \mid B \in \Sigma_n \setminus \{ A \}\})}
\geq e^{-n\epsilon'}
\]
for all $n \geq n_0$ and $A \in \Sigma_n$.
Let $W_A$ be the subspace generated by 
$\{ \pmb{\varphi}^{B} \}_{B \in \Sigma_n \setminus \{ A \}}$ over $\RR$.
If $\theta_A$ is
the angle between $\pmb{\varphi}^{A}$ and $W_A$, then, by Lemma~\ref{lem:vol:ratio}
\[
\sin(\theta_A) = \frac{\vol(\{ \pmb{\varphi}^{B} \mid B \in \Sigma_n\})}
{\sqrt{\langle \pmb{\varphi}^{A}, \pmb{\varphi}^{A}\rangle}\vol(\{ \pmb{\varphi}^{B} \mid B \in \Sigma_n \setminus \{ A \}\})},
\]
and hence
\begin{align*}
\cos(\theta_A) & = \sqrt{1 - \sin^2(\theta_A)} \\
& \leq \sqrt{1 - e^{-2n\epsilon'}} \leq 1 - (1/2)e^{-2n\epsilon'}
\end{align*}
for all $A \in \Sigma_n$. Therefore, for $y \in W_A$,
\begin{align*}
\vert \langle \pmb{\varphi}^A, y \rangle \vert & \leq \cos(\theta_A) \sqrt{\langle \pmb{\varphi}^A, \pmb{\varphi}^A\rangle}\sqrt{\langle y, y\rangle} \\
&\leq (1 - (1/2)e^{-2n\epsilon'})\sqrt{\langle \pmb{\varphi}^A, \pmb{\varphi}^A\rangle}\sqrt{\langle y, y\rangle}.
\end{align*}

Let $\phi \in \Gamma^{\times}(X, n\overline{D})$.
Then we can find $a_{A} \in \QQ$ ($A \in \Sigma_n$) such that $\phi = \sum_{A \in \Sigma_n} a_A \pmb{\varphi}^A$.
Note that $r_n a_{A} \in \ZZ$ for all $A \in \Sigma_n$.
Let us choose $A_0 \in \Sigma_n$ such that $a_{A_0} \not= 0$.
We set $y = \sum_{A \in \Sigma_n \setminus \{ A_0 \}} a_{A} \pmb{\varphi}^A$.
Then $\phi = a_{A_0} \pmb{\varphi}^{A_0} + y$. Since $e^{n\epsilon'} \vert a_{A_0} \vert \geq \vert r_n a_{A_0} \vert \geq 1$,
\begin{align*}
\langle \phi, \phi \rangle & = a_{A_0}^2 \langle \pmb{\varphi}^{A_0}, \pmb{\varphi}^{A_0} \rangle + 2 a_{A_0} \langle \pmb{\varphi}^{A_0}, y \rangle + \langle y, y \rangle \\
& \geq a_{A_0}^2 \langle \pmb{\varphi}^{A_0}, \pmb{\varphi}^{A_0} \rangle + \langle y, y \rangle - 2 \vert a_{A_0} \vert \vert \langle \pmb{\varphi}^{A_0}, y \rangle \vert \\
& \geq  a_{A_0}^2 \langle \pmb{\varphi}^{A_0}, \pmb{\varphi}^{A_0} \rangle + \langle y, y \rangle - 2 \vert a_{A_0} \vert 
\sqrt{\langle \pmb{\varphi}^{A_0}, \pmb{\varphi}^{A_0} \rangle} \sqrt{ \langle y, y \rangle} (1 - (1/2)e^{-2n\epsilon'}) \\
& = (1 - (1/2)e^{-2n\epsilon'})\left(\vert a_{A_0} \vert \sqrt{\langle \pmb{\varphi}^{A_0}, \pmb{\varphi}^{A_0} \rangle} - \sqrt{ \langle y, y \rangle}\right)^2 \\
& \qquad\qquad\qquad\qquad\qquad
+ (1/2)e^{-2n\epsilon'}\left(  a_{A_0}^2 \langle \pmb{\varphi}^{A_0}, \pmb{\varphi}^{A_0} \rangle + \langle y, y \rangle \right) \\
& \geq (1/2)e^{-4 n \epsilon'} (e^{n\epsilon'} a_{A_0})^2  \langle \pmb{\varphi}^{A_0}, \pmb{\varphi}^{A_0} \rangle \geq   
(1/2)e^{-4 n \epsilon'} \langle \pmb{\varphi}^{A_0}, \pmb{\varphi}^{A_0} \rangle.
\end{align*}
On the other hand, by Gromov's inequality (cf. \cite[Proposition~3.1.1]{MoArZariski}),
choosing a larger $n_0$ if necessarily,
$\Vert \psi \Vert^2_{\sup} \leq e^{n\epsilon'} \langle \psi, \psi \rangle$ for all $n \geq n_0$ and $\psi \in H^0(X, nD)$.
Moreover, we may also assume that
$2 \leq e^{n\epsilon'}$ for all $n \geq n_0$.
Thus
\begin{align*}
e^{n\epsilon} \Vert \phi \Vert_{\sup}^2 & = e^{6n\epsilon'} \Vert \phi \Vert_{\sup}^2 
\geq 2e^{5n \epsilon'} \Vert \phi \Vert_{\sup}^2 \\
& \geq 2e^{5n \epsilon'} \langle \phi, \phi \rangle \geq e^{n\epsilon'} \langle \pmb{\varphi}^{A_0}, \pmb{\varphi}^{A_0} \rangle \\
& \geq \Vert \pmb{\varphi}^{A_0} \Vert_{\sup}^2,
\end{align*}
as required.
\end{proof}

\begin{Example}
Let $\PP^d_{\ZZ} = \Proj(\ZZ[T_0, T_1, \ldots, T_d])$, $H_i = \{ T_i = 0 \}$ and $z_i = T_i/T_0$ for $i=0, 1, \ldots, d$.
Let $\overline{D} = (H_0, g)$ be an arithmetic Cartier divisor of $C^0$-type on $\PP^d_{\ZZ}$.
Moreover, let $\Omega$ be an $F_{\infty}$-invariant continuous volume form on $\PP^d(\CC)$.
We assume that there are continuous functions $a$ and $b$ on $\RR_{\geq 0}^d$ such that
$g(z_1, \ldots, z_d) = a(\vert z_1 \vert, \ldots, \vert z_d \vert)$ and
\[
\Omega = \left( \frac{\sqrt{-1}}{2\pi} \right)^d b(\vert z_1 \vert, \ldots, \vert z_d \vert) dz_1 \wedge d\bar{z}_1 \wedge \cdots dz_d \wedge d\bar{z}_d.
\]
Arithmetic Cartier divisors considered in \cite{MoArBig} satisfy the above condition.

Here let us see that $1, z_1, \ldots, z_d$ are well-posed generator for $\overline{D}$.
We set 
\[
\Sigma_n = \{ (a_1, \ldots, a_d) \in \ZZ^d_{\geq 0} \mid a_1 + \cdots + a_d  \leq n\}.
\]
Then $\{ \pmb{z}^{A} \}_{A \in \Sigma_n}$  forms a free basis of $H^0(\PP^d_{\ZZ}, nH_0)$.
Moreover, if we set 
\[
z_i = r_i \exp(2\pi\sqrt{-1}\theta_i) \quad (i=1, \ldots, d),
\]
then
\begin{multline*}
\langle \pmb{z}^{A}, \pmb{z}^{A'} \rangle_{ng}
=
\int_{\RR_{\geq 0}^d \times [0,1]^d} \left( \prod_{i=1}^d 2 r_i^{A_i + A'_i + 1} \exp(2\pi\sqrt{-1}(A_i -A'_i)) \right)  \\
\times \exp(-na(r_1, \ldots, r_d)) b(r_1, \ldots, r_d)
dr_1 \cdots d r_d d \theta_1 \cdots d \theta_d,
\end{multline*}
which implies $\langle \pmb{z}^{A}, \pmb{z}^{A'} \rangle_{ng} = 0$ for
$A, A' \in \Sigma_n$ with $A \not= A'$, and hence
\[
\vol(\{ \pmb{z}^B \mid B \in \Sigma_n\}) = \sqrt{\langle \pmb{z}^{A}, \pmb{z}^{A} \rangle} \vol(\{ \pmb{z}^B \mid B \in \Sigma_n \setminus \{ A \}\})
\]
for all $A \in \Sigma_n$.
\end{Example}

\end{document}